\documentclass[aos]{imsart}

\RequirePackage{amsthm,amsmath,amsfonts,amssymb,comment}
\RequirePackage[authoryear]{natbib}
\RequirePackage[colorlinks,citecolor=blue,urlcolor=blue]{hyperref}
\RequirePackage{graphicx}

\usepackage{bm}
\usepackage{bbm}

\usepackage{amssymb,fancyhdr}
\usepackage{amscd}
\usepackage{mathrsfs,graphics}
\usepackage{multibib}
\newcites{app}{References}

\startlocaldefs
\theoremstyle{plain}

\newtheorem{remark}{Remark}

\newtheorem{theorem}{Theorem}[section]
\newtheorem{lemma}[theorem]{Lemma}
\newtheorem{assumption}{Assumption}
\theoremstyle{remark}

\DeclareMathOperator*{\argmin}{argmin}
\DeclareMathOperator*{\argmax}{argmax}
\def\ae{\color{blue}}

\def \R{\mathbb{R}}
\def \E{\mathbb{E}}
\def\inp#1#2{\langle #1, #2 \rangle}
\def \D{\mathcal{D}}
\def \M{\mathcal{M}}
\def \logit{\text{logit }}
\def \bW{\mathbf{W} }

\makeatletter
\newcommand*{\addFileDependency}[1]{
  \typeout{(#1)}
  \@addtofilelist{#1}
  \IfFileExists{#1}{}{\typeout{No file #1.}}
}
\makeatother

\newcommand*{\myexternaldocument}[1]{%
    \externaldocument[s-]{#1}%
    \addFileDependency{#1.tex}%
    \addFileDependency{#1.aux}%
}
\usepackage{xr}

\myexternaldocument{sup}

\endlocaldefs

\begin{document}

\begin{frontmatter}
\title{Nonparametric estimation of a covariate-adjusted counterfactual treatment regimen response curve}
\runtitle{Nonparametric regimen-response curve estimation}

\begin{aug}
\author[A]{\fnms{Ashkan}~\snm{Ertefaie}\ead[label=e1]{ashkan\_ertefaie@urmc.rochester.edu}},
\author[A]{\fnms{Luke}~\snm{Duttweiler}\ead[label=e2]{luke\_duttweiler@urmc.rochester.edu}\orcid{0000-0000-0000-0000}}
\author[A]{\fnms{Brent A.}~\snm{Johnson}\ead[label=e3]{brent\_johnson@urmc.rochester.edu}\orcid{0000-0000-0000-0000}}
\and
\author[B]{\fnms{Mark J.}~\snm{van der Laan}\ead[label=e4]{laan@berkeley.edu}}
\address[A]{Department of Biostatistics,
University of Rochester\printead[presep={,\ }]{e1,e2,e3}}

\address[B]{Department of Biostatistics,
University of California at Berkeley\printead[presep={,\ }]{e4}}
\end{aug}

\begin{abstract}
Flexible estimation of the mean outcome under a treatment regimen (i.e., value function) is the key step toward personalized medicine. We define our target parameter as a conditional value function given a set of baseline covariates which we refer to as a stratum based value function. We focus on  semiparametric class of decision rules and propose a sieve based nonparametric covariate adjusted regimen-response curve estimator within that class. Our work contributes in several ways. First, we propose an  inverse probability weighted  nonparametrically efficient estimator of the smoothed regimen-response curve function. We show that asymptotic linearity is achieved when the nuisance functions are undersmoothed sufficiently. Asymptotic and finite sample criteria for undersmoothing are proposed. Second, using Gaussian process theory, we propose simultaneous confidence intervals  for the smoothed regimen-response curve function.  Third, we provide consistency and convergence rate for the optimizer of the regimen-response curve estimator; this enables us to estimate an optimal semiparametric rule. The latter is important as the optimizer corresponds with the optimal dynamic treatment regimen. Some finite-sample properties are explored with simulations.   

\end{abstract}



\end{frontmatter}

\section{Introduction}

Marginal structural models have been widely used in causal inference as a tool for estimating the mean outcome under a give decision rule \citep{robins1998marginal} and constructing optimal treatment regimens. This approach requires modeling the potential outcome given a set of baseline covariates, including treatment and possibly some subset of baseline information on a subject.  Existing methods typically impose a parametric causal model on the mean of the potential outcome and draw inference under the assumption that the model is correctly specified  \citep{robins2000marginal, murphy2001marginal, joffe2004model, orellana2010dynamic}. The latter is particularly restrictive when the conditioning set of treatment and baseline covariates includes some continuous variables. Although model selection approaches have been proposed 
\citep{van2003unified}, it is likely that finite sample size prevents recovering the true model and leads to poor out-of-sample performance \citep{shinozaki2020understanding}.  \cite{neugebauer2007nonparametric} proposed a nonparametric marginal structural model framework 
to improve the causal interpretability of the estimated parameters of a potentially 
misspecified parametric working model. However, when the working model is misspecified, the causally interpretable parameters might not be the parameters of interest and the resulting treatment regimen might be suboptimal.


The parameters of the proposed marginal structural model can be estimated using an inverse probability weighted loss functions where weights  are the conditional probability of following a particular decision rule. The consistency and causal interpretability of the corresponding estimators rely on the correctly specified weight function. To mitigate the chance of model misspecification, one may decide to model the weight function nonparametrically using data-adaptive techniques. However, this choice may result in nonregular estimators of the primary parameters in the marginal structural model. This lack of regularity is because the rate of convergence for inverse probability weighted estimators rely on a rate of convergence of the weight function which will be slower than the desired root-$n$ rate when data-adaptive techniques are used. To overcome this concern, authors have used so-called doubly robust estimators  for the parameters in the working marginal structural model; that is, estimators that are consistent as long as one of either the treatment model or the outcome models are correctly specified \citep{robins2004optimal, wahed2004optimal, orellana2010dynamic,rosenblum2011marginal, petersen2014targeted, ertefaie2015identifying}. However, doubly-robust nonparametric estimators may suffer from two important issues: (1) the performance of the estimators depends on the modeling choice of nuisance parameters and they can be irregular with large bias and a slow rate of convergence  \citep{van2014targeted, benkeser2017doubly}; and (2) the quality of the constructed decision rule still relies on correctly-specified  marginal structural model. 

Alternative to the marginal structural models framework, several authors have proposed methods to estimate nonparametric decision rules that circumvent the need for specifying a parametric outcome model.  
\cite{zhao2012estimating} proposed an inverse probability weighting approach to estimate a nonparametric decision rule \citep{dudik2011doubly, zhao2015new}. Doubly robust augmented inverse probability weighting procedures have also been proposed \citep{rubin2012statistical,zhang2012robust, liu2018augmented, athey2021policy}. Due to the nonparametric nature of the estimated decision rule, inference for the resulting value function (i.e., mean outcome under the estimated rule) is challenging. \cite{zhao2012estimating} developed Fisher consistency, the excess risk bound, and the convergence rate of the value function to provide a theoretical guarantee for their proposed method. \cite{athey2021policy} derived bounds for  the regret function under the estimated nonparametric treatment strategy and showed that the bounds decay as $n^{-1/2}$ as long as the Vapnik–Chervonenkis (VC) dimension of the class of decision rules does not grow too fast. Other recent germane contributions on estimating optimal treatment regimes and conditional average treatment effects include \cite{swaminathan2015batch}, \cite{kallus2018balanced}, \cite{nie2021quasi}, \cite{zhao2022selective}, and references therein.

We define a   stratum-based value function as the conditional expectation of the potential outcome under a treatment regimen given a prespecified set of baseline covariates $\bm{V}$. The value  is a function of regimens and we are interested in understanding the functional association between the mean outcome and the regimens within a class which we refer to as the {\it regimen-response curve}.  We consider a current status data setting with two levels of coarsening at random (i.e., treatment assignment and censoring). We overcome the two important shortcomings of the existing methods by: (1) allowing the {regimen-response curve} to be estimated nonparametrically; and (2) estimating the weight functions of the inverse probability weighted loss function using a data-adaptive technique. Recently, \cite{ertefaie2022nonparametric} proposed a nonparametric efficient inverse probability weighted estimator for the average causal effect where the weights are estimated using  undersmoothed highly adaptive lasso. We will generalize below the same technique to estimate conditional average causal effects via nonparametric marginal structural models.    An important byproduct of our nonparametric regimen-response curve estimator is a method to construct semiparametric decision rules where the parameters of the rule are allowed to be a nonparametric function of the set baseline covariates (i.e., ${V}$-specific optimal decision rule). The parameters are defined as the optimizer of the regimen-response curve which allows us to provide a rate of convergence for the corresponding functional parameter estimators using empirical process theory.  
Our theoretical contributions can be summarized as follows: 
\begin{itemize}
    \item Considering a kernel smoothed regimen-response curve function as the target parameter, we show that the proposed inverse probability weighted estimator (1) is nonparametrically efficient when the weight functions are properly undersmoothed;  (2) 
      convergences to a Gaussian process at root-$n$ rate; and (3) is uniformly consistent at a derived rate. 
    \item Considering the regimen-response curve function as the target parameter, we derive the asymptotic normality of our estimator and show that the bias caused by the kernel smoothing of the estimator vanishes  under some mild assumptions and properly undersmoothed kernel and weight functions.
    \item Using a novel empirical process algorithm, we show that the minimizer of the estimated regimen-response curve is a consistent estimator of the minimizer of the true regimen-response curve and derive the convergence rate.  This result is of independent interest as it derives a rate of convergence for an optimizer of a nonparametrically estimated function.   
\end{itemize}

In contrast to the existing literature which targets the value function under the estimated rule, our estimand is the stratum based value function. Deriving the asymptotic properties of our estimand is more challenging because it will be nonpathwise differentiable functional when $\bm{V}$ includes at least one continuous variable \citep{bickel1998efficient}.  Moreover, existing methods can either estimate a parametric or nonparametric rules where in each case flexibility or interpretability is sacrificed. Our proposed semiparametric rule lands in between and provides a reasonable trade-off between flexibility and interpretability. For example, clinicians might have  strong opinion about an important tailoring variable and want to ensure to capture the correct functional form of that variable in the decision rule. In such settings, we can include that variable in the set $\bm{V}$ and provide a  stratum-based decision rule where  the decision rule will be a nonparametric function of $\bm{V}$.   


\section{Preliminaries}

\subsection{Notation and Formulation}\label{sec:notation}

Let $T$ be a univariate failure time, $\bm{W} \in \mathcal{W} \subset \R^{p}$ be a   vector of all available baseline covariates measured before treatment $A \in \{0,1\}$, and $C$ be a censoring variable.  Let $\bm{S}\in \mathcal{S} \subset \R^{q}$ be a  subvector of $\bm{W}$ and $\bm{V} = \bm{W} \setminus \bm{S}$ where $\bm{V} \in \mathcal{V} \subset \R^{p-q}$. For simplicity of notation, we denote $p-q$ as $r$.  A  decision rule $d^{\theta}  \in \mathcal{D} $ is defined as a function that maps values of $\bm{S}$ to a treatment option $\{0,1\}$. In our formulation $\mathcal{D}$ is a class of deterministic rules indexed by a vector of coefficients.  Specifically, we consider $d^{\theta} =I\{\bm{S}^\top \theta>0\}$ with $\theta \in \Theta \subset \R^q$ where $I(.)$ is an indicator function. 
Let $T^\theta$ be the potential outcome that would have been observed under the decision rule $d^\theta$ for a given value of $\theta \in \Theta$. Suppose that we have the full data $X = \{(T^\theta: d^\theta \in \mathcal{D}), \bm{W}\} \sim P_X \in \mathcal{M}^F$  where $\mathcal{M}^F$ is a nonparametric full data model. Define 
$E_{P_X}\{I(T^\theta >t) \mid \bm{V}\}=\Psi_{P_X}(\theta,\bm{V}) \in \mathcal{F}_{\lambda}$  where $\mathcal{F}_{\lambda}$ is a class of cadlag functions with  sectional variation norm bounded by $\lambda$ and $t$ is a predetermined time point.
The full-data risk of $\Psi$ is defined as $E_{P_X} \int_\theta \ell_{\theta}(\Psi,x) dF(\theta)$
 where $\ell$ is the log-likelihood loss function,
$\ell_{\theta}(\Psi) =  I(T^\theta >t) \log \left(\frac{\Psi}{1-\Psi}\right) +\log(1-\Psi),$ and $dF(\theta)$ is a weight function/measure on the Euclidean set of all $\theta$-values (e.g., multivariate normal with possibly large variance).
Define a V-specific optimal decision rule that maximizes the $t$-year survival probability as $\theta_{0}(\bm{V}) = \argmax_{\theta \in \Theta} \Psi_{P_X}(\theta,\bm{V})$. Hence,  we allow $\theta$ to be a functional parameter mapping $\mathcal{V}$ onto $\R^q$. 

The observed data $\mathcal{O} = \{Y = \min(T,C), \Delta = I(T<C), A,\bm{W}\}$   follow some probability distribution $P_0$ that lies in some nonparametric model $\mathcal{M}$. 
Suppose we observe $n$ independent, identically distributed trajectories of $\mathcal{O}$. 
Let $\Delta^c = \Delta+(1-\Delta)I\{Y>t \}$ 
and $G: \mathcal{M} \rightarrow \mathcal{G}$ be a functional nuisance parameter where $\mathcal{G}=\{G(P): P \in \mathcal{M}\}$. 
In our setting, $G(P)\{(A,\Delta^c) \mid \bm{W}\}$ denotes the joint treatment and censoring mechanism under a distribution $P$. We refer to $G(P_0)$ as $G_0$ and $G(P)$ as $G$. Under coarsening at random assumption, $G(P)\{(A,\Delta^c)\mid {X}\}=G(P)\{(A,\Delta^c) \mid \bm{W}\}$. Denote $G^a \equiv G^a(P)(A \mid \bm{W})$ and $G^c \equiv G^c(P)(\Delta^c \mid \bm{W},A)$. We denote the latter two functions under $P_0$ as $G^a_0$ and $G^c_0$.
Also define $Q:  \mathcal{M} \rightarrow \mathcal{Q}$ and $\mathcal{Q}=\{Q(P): P \in \mathcal{M}\}$ where $Q_0(\theta,\bm{W})= \E_{P_0}\{I(Y^{\theta}>t) \mid \bm{W}\}$.
 Let $\prod$ be a projection operator in the Hilbert space $L_0^2(P)$ with inner product $ \inp{h_1}{h_2} =E_P(h_1h_2)$. 

The loss function $\int_\theta \ell_{\theta}(\Psi) dF(\theta) $ is based on the full-data  which is not observable, and thus, cannot be used to define our full data functional parameter $\Psi_{P_X}$. However, under consistency ($T_i=T^{A_i}$) and no unmeasured confounder ($A \perp T^a \mid \bm{W}$ and $\Delta^c \perp T \mid A, \bm{W}$) assumptions, it can be identified using the observed data distribution $P_0$. Specifically, we define an inverse probability weighting mapping of the full-data loss as
\begin{align}
L_{G}(\Psi)=\int_{\theta}  \left[   \frac{\Delta^c I( A= d^\theta)\xi(\theta,\bm{V}) }{G\{(A,\Delta^c) \mid \bm{W}\} } \{ I(Y >t) \log \left(\frac{\Psi}{1-\Psi}\right) +\log(1-\Psi)\}   \right] dF(\theta). \label{eq:lossf}
\end{align}

Accordingly, define the regimen-response curve as $\Psi_{P_0} = \argmin_{\Psi \in \mathcal{F}_\lambda} E_{P_0}  L_{G_0}(\Psi)$. Following our convention, we denote $\Psi_{P_0}$ as $\Psi_{0}$. We define a highly adaptive estimator of $\Psi_{P_0}$ as $\Psi_n = \argmin_{\Psi \in \mathcal{F}_{\lambda_n}} P_n  L_{G_n}(\Psi)$ where $\lambda_n$ is a data adaptively selected $\lambda$ and $G_n$ is an estimator of $G_0$. The choice of  function $\xi$ is inconsequential to the derivation of the efficient influence function of our target parameter but it may impact the level of undersmoothing needed in finite sample. One may consider  $\xi(\theta,\bm{V})=1$ or  $\xi(\theta,\bm{V})=G\{(A,\Delta^c) \mid \bm{V}\}$. The latter choice corresponds to the stabilized weight. 

\subsection{Highly adaptive lasso}

The highly adaptive lasso is a nonparametric regression approach that can be used to estimate infinite dimensional functional parameters \cite{benkeser2016highly, van2017generally, van2017uniform}. It forms a linear combination of indicator basis functions  to minimize the expected value of a loss function under the constraint that the constraint that the $L_1$-norm of the coefficient vector is bounded by a finite constant value.

  Let $\D[0,\tau]$ be
the Banach space of d-variate real valued cadlag functions (right-continuous with left-hand limits) on a cube $[0,\tau] \in \R^d$. For each function $f \in \D[0,\tau]$ define the supremum norm as $\| f \|_{\infty}=\sup_{w \in [0,\tau]} |f(w)|$. For any subset $s$ of $\{0,1,\ldots,d\}$, we partition $[0,\tau]$ in $\{0\} \{ \cup_s (0_s,\tau_s]\}$ and define the sectional variation norm of such $f$ as 
\[
\| f\|^*_\zeta = |f(0)|+\sum_{s \subset\{1,\ldots,d\}} \int_{0_s}^{\tau_s} |df_s(u_s)|,
\]
where the sum is over all subsets of $\{0,1,\ldots,d\}$. For a given subset  $s \subset \{0,1,\ldots,d\}$, define $u_s =(u_j :j \in s)$ and $u_{-s}$ as the complement of $u_s$. Then, $f_s: [0_s,\tau_s] \rightarrow \R$ defined as $f_s(u_s) = f(u_s,0_{-s})$. Thus, $f_s(u_s)$  is a section of $f$ that sets the components in the complement of $s$ equal to zero and varies only along the variables in $u_r$. 

Let $G^a \equiv G(P)(A \mid W)$ and $G^c \equiv G(P)(\Delta^c \mid W,A)$ denote the treatment and censoring mechanism under an arbitrary distribution $P \in \M$. Assuming that our nuisance functional parameter $G^a,G^c \in  \D[0,\tau]$  has finite sectional variation norm, we can represent $ \text{logit } G^a$ as \citep{gill1993inefficient} 
\begin{align}
\text{logit } G^a(w):&=\text{logit } G^a(0)+\sum_{s \subset\{1,\ldots,d\}} \int_{0_s}^{w_s} dG^a_s(u_s) \nonumber \\
      &=\text{logit } G^a(0)+\sum_{s \subset\{1,\ldots,d\}} \int_{0_s}^{\tau_s} I(u_s \leq w_s)dG^a_s(u_s). \label{eq:hal1}
\end{align}
The representation (\ref{eq:hal1}) can be approximated using a discrete measure that puts mass on each observed $W_{s,i}$ denoted by $\alpha_{s,i}$. Let   $\phi_{s,i}(c_s)= I(w_{i,s} \leq c_s)$ where $w_{i,s}$ are support points of $dG^a_s$. We then have
\[
 \text{logit } G^a_\alpha = \alpha_0+\sum_{s \subset\{1,\ldots,d\}}\sum_{i=1}^{n} \alpha_{s,i} \phi_{s,i},
\]
where $ |\alpha_0|+\sum_{s \subset\{1,\ldots,d\}}\sum_{i=1}^{n} |\alpha_{s,i}|$ is an approximation of the sectional variation norm of $f$. The loss based highly adaptive lasso estimator $\beta_n$ is defined as
\[
\alpha_n (\lambda)= \arg \min_{\alpha: |\alpha_0|+\sum_{s \subset\{1,\ldots,d\}}\sum_{i=1}^{n} |\alpha_{s,i}|<\lambda} P_n L( \text{logit } G^a_\alpha).
\]
where $L(.)$ is a given loss function. Denote $ G^a_{n,\lambda} \equiv G^a_{\alpha_n(\lambda)}$ as  the highly adaptive lasso estimate of $G^a$. Similarly, we can define $ G^{c}_{n,\lambda}$ as the highly adaptive lasso estimate $G^c$.  When the nuisance parameter correspond to a conditional probability of a binary variable (e.g., propensity score with a binary treatment or a binary censoring indicator), log-likelihood loss can be used. In practice, $\lambda$ is unknown and must be obtained using data. We refer to $\lambda_n$ as a value of $\lambda$ that is determined using data. 

\section{Dynamic marginal structural models with highly adaptive lasso} \label{sec:est}
\subsection{The challenge} \label{sec:challenge}
The target functional parameter $\Psi_0$ is not a pathwise differentiable function which makes the inference challenging. The existing literature approximates this functional by pathwise differentiable function $\Psi_{\beta}$, where $\beta$ is a vector of unknown parameters \cite{van2006causal, van2007statistical, murphy2001marginal,robins2008estimation, orellana2010dynamic, neugebauer2007nonparametric}. The working model provides a user-specified
summary 
 of the unknown true function relating the expectation of the potential outcome
under a strategy.
Hence, the quality of the constructed decision rule depends on how well the working model approximates the true function. To mitigate this concern we propose  a kernel smoothed highly adaptive lasso estimator of $\Psi_0(\bm{v}_0,\theta)$. Specifically, we define an regimen-response curve estimator as
 \begin{align}
  \Psi_{nh}(\bm{v}_0,\theta) = P_n  U(P_n,\bm{v}_0,\Psi_n) \label{eq:propest} 
\end{align}
where $U(P_n,\bm{v}_0,\Psi_n) = (P_n K_{h,\bm{v}_0})^{-1} \{K_{h,\bm{v}_0} \Psi_n(\theta)\} $ and  $\Psi_n$ is a highly adaptive lasso estimate of $\Psi_0$ obtained based on the loss function (\ref{eq:lossf}). In this formulation, $K_{h,\bm{v}_0}(v) = h^{-r} K\{(v-\bm{v}_0)h^{-1}\}$ with $K(v) = K(v_1,\cdots,v_r)$ where $K(v)$ is a kernel satisfying $\int K(v) dv=1 $. Moreover, we assume that $K(v) = \prod_{j=1}^r K_u(v_j)$ is a product kernel  defined by a univariate kernel $K_u(.)$ and $K(v)$ is a J-orthogonal kernel such that  $\int K_u(v) dv=1 $ and $\int K_u(v) v^j  dv=0 $ for $j=1,\cdots,J$. 

We study the asymptotic behaviour of our estimator relative to  $\Psi_0(\bm{v_0},\theta)$  and a kernel smoothed parameter
\[
\Psi_{0h}(\bm{v}_0,\theta) = P_0  U(P_0,\bm{v}_0,\Psi_0), 
\]
where $h$ is a fixed bandwidth. 
In general, when the nuisance parameter $G_0$ is estimated using data adaptive regression techniques, the resulting regimen-response curve estimator $\Psi_{nh}(\bm{v}_0,\theta)$ won't be asymptotically linear which makes inference challenging. For example, consider the target parameter $\Psi_{0h}(\bm{v}_0,\theta)$, we have
 \begin{align}
  \Psi_{nh}(\bm{v}_0,\theta)-\Psi_{0h}(\bm{v}_0,\theta) =& (P_n-P_0) U(P_0,\bm{v}_0,\Psi_0) + \nonumber \\
                                                                          &P_0\{U(P_n,\bm{v}_0,\Psi_n)-U(P_0,\bm{v}_0,\Psi_0)\} + \nonumber\\
                                                                          &(P_n-P_0) \{U(P_n,\bm{v}_0,\Psi_n)-U(P_0,\bm{v}_0,\Psi_0)\}. \label{eq:aslin}
  \end{align}
The first term on the right-hand side of (\ref{eq:aslin}) is exactly linear and the third term is negligible under Donsker condition  of $\Psi_n$. The problem arises because  data adaptive regression techniques have a rate of convergence slower than the desired root-$n$ rate, and thus, the bias of $P_0\{U(P_n,\bm{v}_0,\Psi_n)-U(P_0,\bm{v}_0,\Psi_0)\}$ diverges to infinity.

\subsection{The intuition} We show that when the nuisance parameters are properly undersmoothed, the asymptotic linearity of our estimator $\Psi_{nh}(\bm{v}_0,\theta)$ can be retrieved. Specifically, in the proof of Theorem \ref{th:movh}, we show that
\begin{align}
    P_0\{U(P_n,\bm{v}_0,\Psi_n)-U(P_0,\bm{v}_0,\Psi_0)\} = &(P_n-P_0)\tilde D^*_{h,\bm{v_0}}(P_0)+o_p(n^{-1/2}) \nonumber\\
&\hspace{.2in} +P_n \frac{K_{h,\bm{v}_0}\Delta^c I( A= d^\theta)}{G_nP_n K_{h,\bm{v}_0}} \left\{I(Y\geq t)-\Psi_n \right\} \label{eq:scorepsi}\\
&\hspace{.2in} + P_n \frac{K_{h,\bm{v}_0}I( A= d^\theta)}{G_n P_n  K_{h,\bm{v}_0}}(Q_n-\Psi_n)\left(\Delta^c-G_{n}^c \right) \label{eq:scoredelta}\\
&\hspace{.2in} + P_n  \frac{K_{h,\bm{v}_0 }(Q_n-\Psi_n)}{G_{n}^a P_n K_{h,\bm{v}_0}} \left\{I( A= d^\theta)-G_{n}^a\right\}, \label{eq:scorea}
\end{align}
where $\tilde D^*_{h,\bm{v_0}}(P_0)$ includes certain components of the canonical gradient of our statistical parameter $\Psi_{0h}(\bm{v_0},\theta)$ and thus, contributes to the efficient influence function of our estimator. The last three terms do not converge to zero at  the appropriate rate and thus induce non-negligible bias. However, the terms (\ref{eq:scorepsi}), (\ref{eq:scoredelta}) and (\ref{eq:scorea}) resemble the form of score equations for $I(Y\geq t)$, $\Delta^c$ and $I( A= d^\theta)$. 

We first discuss the score equations for $G^.$ where $. \in \{a,c\}$. We can generate score $S_g(G_n^.)$ for a nonparametric estimator $G_n^.$ by paths $\{G_{n,\epsilon}^{.g}(w)\}$, such that
\begin{align*}
  \logit G_{n,\epsilon}^{.g}(w) &= \logit G_n^.(0) \{1+\epsilon g(0)\} + \sum_{s \subset\{1,\ldots,d\}}
    \int_{0_s}^{w_s} \phi_{s,u_s}(w) \{1+\epsilon g(s,u_s)\} d \logit G_{n,s}^.(u),
\end{align*}
where $g$ is a uniformly bounded function on $[0,\tau]^p$. Accordingly when a function is estimated using a highly adaptive lasso, the score functions can be generated by a path $\{1+\epsilon g(s,j)\} \beta_{n,s,j}$ for
a bounded vector $g$ as
\begin{equation} \label{eq:score}
  S_g(G_{n,\lambda_n}^.) = \frac{d}{d \logit G^._{n,\lambda_n}} L(
  \logit G^._{n,\lambda_n}) \left\{ \sum_{(s,j)} g(s,j) \beta_{n,s,j} \phi_{s,j}
  \right\},
\end{equation}
where $L(\cdot)$ is the log-likelihood loss function. For example, for the treatment indicator, 
$L(G^a) =  A \log \left(\frac{G^a}{1-G^a}\right) +\log(1-G^a)$. The penalty in the highly adaptive lasso imposes a restriction on the $g$ function. Specifically, the path $\{1+\epsilon g(s,j)\} \beta_{n,s,j}$ can be generated using those $g$ functions that do not change the $L_1$-norm of the parameters. Let $\beta^g_{n,s,j}=\{1+\epsilon g(s,j)\} \beta_{n,s,j}$ denote the set of perturbed parameters. Then, for small enough $\epsilon$ such that $\{1+\epsilon g(s,j)\}>0$,  $\|\beta^g_{n,s,j}\|_1= \|\beta^g_{n,s,j}\|_1 \{1+\epsilon g(s,j)\} = \|\beta^g_{n,s,j}\|_1 + \epsilon g(s,j) \|\beta_{n,s,j}\|_1$. Hence, the restriction is satisfied when the inner product of $g$ and the vector $|\beta|$ is zero (i.e., $<g.|\beta|>=0$). 

We now provide a thought experiment on how undersmoothing the nuisance function estimates using a highly adaptive lasso can eliminate the terms (\ref{eq:scorepsi}), (\ref{eq:scoredelta}) and (\ref{eq:scorea}). We first ignore the $L_1$-norm restriction on the choice of function $g$ which is erroneous but contains the main idea of the approach. Without the restriction, one can choose $g(s_0,j_0) = I(s=s_0,j=j_0)$. The latter choice perturbs one coefficient at a time and corresponds to maximum likelihood estimators.  Then, for any $s_0$ and $j_0$, the score function (\ref{eq:score}) becomes 
\[
S_I(G_{n,\lambda_n}^.) = \frac{d}{d \logit G^._{n,\lambda_n}} L(
  \logit G^._{n,\lambda_n}) \left(  \beta_{n,s_0,j_0} \phi_{s_0,j_0}
  \right).
\]
Because, $\beta_{n,s_0,j_0}$ is a finite constant for all $s_0$ and $j_0$, solving the above score equation is equivalent to solving 
\[
 \frac{d}{d \logit G^._{n,\lambda_n}} L(
  \logit G^._{n,\lambda_n})  \phi_{s_0,j_0}.  
\]
For example, for the treatment indicator, the score equation is given by
$ (A - G_{n,\lambda_n}^a) \phi_{s_0,j_0}$. As we undersmooth the fit, we solve more and more score equations (i.e., the number of score equations increases with the number of features included in the model) and any linear combination of those score equations will also be solved.  Now let's consider the term (\ref{eq:scorea}) and let $f=\frac{K_{h,\bm{v}_0 }(Q_n-\Psi_n)}{G_{n}^a P_n K_{h,\bm{v}_0}}$. Assuming that $Q_0$, $\Psi_0$ and $G_0 = G_0^aG_0^c$ 
   are c\`{a}dl\`{a}g with finite
  sectional variation norm (Assumption \ref{assump:cadlag} in Section \ref{sec:theory}), \cite{gill1993inefficient} showed that $f$ can be approximated as a linear combination of indicator basis functions.  Therefore, if we sufficiently undersmooth $G_{n,\lambda}^a$ we will solve (\ref{eq:scorea}) up to a root-$n$ factor. That is 
  \[
P_n  \frac{K_{h,\bm{v}_0 }(Q_n-\Psi_n)}{G_{n}^a P_n K_{h,\bm{v}_0}} \left\{I( A= d^\theta)-G_{n}^a\right\}= o_p(n^{-1/2}). 
  \]
The same argument can be made to show that the other terms (\ref{eq:scorepsi}) and (\ref{eq:scoredelta}) can be made negligible when $\Psi_n$ and $G_{n}^c$ are properly undersmoothed. Hence, the challenging term will be asymptotically linear with
\begin{align*}
    P_0\{U(P_n,\bm{v}_0,\Psi_n)-U(P_0,\bm{v}_0,\Psi_0)\} = &(P_n-P_0)\tilde D^*_{h,\bm{v_0}}(P_0)+o_p(n^{-1/2}).  
\end{align*}
Now, we study the actual case where the choice function $h$ is restricted to those that do not change the $L_1-$norm of the coefficients. Lemma \ref{lem:ucondition-fixedh} shows than when   (\ref{eq:basis3-fixedh}), (\ref{eq:basis1-fixedh}) and (\ref{eq:basis2-fixedh}) are satisfied, our proposed estimator achieves  asymptotic linearity. As we undersmooth the fit, we start adding features with small coefficients into the model. Conditions (\ref{eq:basis3-fixedh}), (\ref{eq:basis1-fixedh}) and (\ref{eq:basis2-fixedh}) imply that the $L_1-$norm must be increases until one of the corresponding score equations is solved to a precision of $o_p(n^{-1/2})$. 

Here we provide details for the score equations corresponding to the treatment indicator where $S_g(G_{n,\lambda_n}^a) =   (A - G_{n,\lambda_n}^a) \left\{\sum_{(s,j)} g(s,j)
\beta_{n,s,j} \phi_{s,j} \right\}$.
Let $r(g,G_{n,\lambda_n}^a) = \sum_{(s,j)} g(s,j) \lvert \beta_{n,s,j} \rvert$.
For small enough $\epsilon$,
\begin{align*}
 \sum_{(s,j)} \lvert \{1+\epsilon g(s,j)\} \beta_{n,s,j} \rvert &= \sum_{(s,j)}
   \{1 + \epsilon g(s,j)\} \lvert \beta_{n,s,j} \rvert \\
   &=\sum_{(s,j)} \lvert \beta_{n,s,j} \rvert + \epsilon r(g,G_{n,\lambda_n}^a).
\end{align*}
Hence,  for any $g$ satisfying $r(g,G_{n,\lambda_n}^a)=0$ (i.e., $h$ does not change the $L_1-$ norm of the coefficients), we have $P_n
S_g(G_{n,\lambda_n}^a) = 0$. 

Let $D\{f,G_{n,\lambda_n}^a\} = f \cdot (A - G_{n,\lambda_n}^a)$, where $f$ is
defined above, and let
$\tilde{f}$ be an approximation of $f$ using the basis functions that satisfy
condition (\ref{eq:basis1-fixedh}). Then, $D\{\tilde{f}, G_{n,\lambda_n}^a\} \in
\{S_g(G_{n,\lambda_n}^a): \lVert g \rVert_{\infty} < \infty \}$. Thus, there
exists $g^{\star}$ such that $D(\tilde{f}, G_{n,\lambda_n}^a) = S_{g^{\star}}
(G_{n,\lambda_n}^a)$; however, for this particular choice of $g^{\star}$,
$r(g,G_{n,\lambda_n}^a)$ may not be zero (i.e., the restriction on $g$ might be violated). Now, define $g$ such that $ g(s,j)
= g^{\star}(s,j)$ for $(s,j) \neq (s^{\star}, j^{\star})$; $\tilde{g}
(s^{\star}, j^{\star})$ is defined such that
\begin{align}\label{eq:restr}
\sum_{(s,j) \neq (s^{\star},j^{\star})} g^{\star}(s,j) \lvert
 \beta_{n,s,j} \rvert +  g(s^{\star}, j^{\star})
 \lvert \beta_{n, s^{\star}, j^{\star}} \rvert = 0.
\end{align}
That is, $g$ matches $g^{\star}$ everywhere but for a single point
$(s^{\star}, j^{\star})$, where it is forced to take a value such that
$r(g,G_{n,\lambda_n}^a)=0$. As a result, for such a choice of $g$, $P_n S_{h}
(G_{n,\lambda_n}^a) = 0$ by definition. Below, we show that $P_n
S_{g}(G^a_{n,\lambda_n}) - P_n D(\tilde{f}, G_{n,\lambda_n}^a) = o_p(n^{-1/2})$
which then implies that $P_n D(\tilde{f}, G_{n,\lambda_n}^a)
= o_p(n^{-1/2})$. We note that the choice of $(s^{\star}, j^{\star})$ is
inconsequential.
\begin{align*}
   P_n S_{g}(G_{n,\lambda_n}^a) - P_n D\{\tilde{f}, G_{n,\lambda_n}^a\} &= P_n
   S_{g}(G_{n,\lambda_n}^a) - P_n S_{g^*}(G_{n,\lambda_n}^a) \\ &=P_n \left\{
   \frac{d}{d\logit G^a_{n,\lambda_n}} L( \logit G^a_{n,\lambda_n})
   \left[\sum_{(s,j)} \left\{g(s,j) - g^{\star}(s,j)\right\} \beta_{n,s,j}
   \phi_{s,j} \right]  \right\} \\ &= P_n
   \left[\frac{d}{d\logit G^a_{n,\lambda_n}} L( \logit G^a_{n,\lambda_n}) \left\{
   g(s^{\star},j^{\star}) - g^{\star}(s^{\star},j^{\star})\right\}
   \beta_{n,s^{\star},j^{\star}} \phi_{s^{\star},j^{\star}} \right] \\ 
   &=O_p \left(P_n \left[\frac{d}{d\logit G^a_{n,\lambda_n}} L(
 \logit G^a_{n,\lambda_n})\phi_{s^{\star},j^{\star}} \right] \right) \\
 & = o_p(n^{-1/2}).
\end{align*}
The details of the forth equality is given in the proof of Lemma  \ref{lem:ucondition-fixedh}
and the last equality follows from the assumption that $\min_{(s,j) \in
\mathcal{J}_n } \lVert P_n \frac{d}{d\logit G^a_{n,\lambda_n}}
L(\logit G^a_{n,\lambda_n}) (\phi_{s,j}) \rVert = o_p(n^{-1/2})$ for $L(\cdot)$
being log-likelihood loss (i.e., condition (\ref{eq:basis1-fixedh})). As $P_n S_{g}(G^a_{n,\lambda_n}) = 0$, it
follows that $P_n D(\tilde{f},G^a_{n,\lambda_n}) = o_p(n^{-1/2})$. Using this
result, under mild assumptions, we showed that $P_n D(f,G^a_{n,\lambda_n})= o_p(n^{-1/2})$ indicating that the term (\ref{eq:scorea}) will be asymptotically negligible (see the proof of Lemma \ref{lem:ucondition-fixedh} for details).

\subsection{The estimator.} To improve the finite sample performance of our estimator we propose to estimate the nuisance parameters using cross-fitting ~\citep{klaassen1987consistent, zheng2011cross,
chernozhukov2017double}. We split the data at random into $B$
mutually exclusive and exhaustive sets of size approximately $n B^{-1}$. Let $P_{n,b}^0$ and $P_{n,b}^1$ denote the  the empirical distribution of a training  and a validation sample, respectively. For a given
$\lambda$ and $h$, exclude a single (validation) fold of data and fit the highly adaptive lasso estimator
using data from the remaining $(B-1)$ folds; use this model to estimate the
nuisance parameters for samples in the holdout (validation) fold. By repeating the process $B$ times, we will have estimates of the nuisance parameters for all sample units. Accordingly, we define the cross-fitted IPW estimator
 \begin{align}
  \Psi_{nh}^{\textsuperscript{CF}}(\bm{v}_0,\theta) =  B^{-1}
\sum_{b=1}^B  P_{n,b}^1 U(P_{n,b}^1,\bm{v}_0,\Psi_{n,\lambda,b}) \label{eq:propest} 
\end{align}
where $U(P_{n,b}^1,\bm{v}_0,\Psi_{n,\lambda,b}) = (G_{n,\lambda,b} P_{n,b}^1 K_{h,\bm{v}_0})^{-1} \{K_{h,\bm{v}_0} \Delta^c I( A= d^\theta) \Psi_{n,\lambda,b}\} $,  $\Psi_{n,\lambda,b}$ and $G_{n,\lambda,b}$ are  highly adaptive lasso estimates of $\Psi_0$ and $G_0$ applied
to the training sample for the b\textsuperscript{th} sample split for a given
$\lambda$. The estimator uses a ($J-1$)-orthogonal kernel with bandwidth $h$ centered at $\bm{v_0}$.

\subsection{Data-adaptive bandwidth selector} \label{sec:bandw}
The proposed estimator $\Psi_{nh}(\bm{v}_0,\theta)$ approaches to $\Psi_0(\bm{v}_0,\theta)$ as $h_n \rightarrow 0$ at the cost of increasing the variance of the estimator. Mean square error (MSE) provides an ideal criterion for bias-variance trade-off. However, because the bias is unknown the latter  cannot be used directly. Here we propose an approach that circumvent the need for knowing bias while targeting the optimal bias-variance trade-off. Let $\sigma_{nh}$ be a consistent estimator of $\sigma_{0} = [E\{D^{*2}_{\bm{v}_0}(P_0)\}]^{1/2}$ where $D^{*2}_{\bm{v}_0}(P_0)$ is the efficient influence function for $\Psi_{nh}(\bm{v}_0,\theta)$ (see Theorem \ref{th:movh}). The goal is to find an $h$ that minimizes the MSE or equivalently set the derivative of the MSE to zero. That is
\[
\{\Psi_{nh}(\bm{v}_0,\theta) - \Psi_0(\bm{v}_0,\theta)\} \frac{\partial }{\partial h} \Psi_{nh}(\bm{v}_0,\theta) + \frac{\sigma_{nh}}{(nh^r)^{1/2}} \frac{\partial }{\partial h}\sigma_{nh}/(nh^r)^{1/2}=0
\]
We know that the optimal (up to a constant) bias-variance trade-off implies  $\{\Psi_{nh}(\bm{v}_0,\theta) - \Psi_0(\bm{v}_0,\theta)\} \approx \sigma_{nh}/(nh^r)^{1/2}$  \citep{van2018cv}. Hence, the optimal $h$ will also solve one of the following equations
\begin{align}
    \frac{\partial }{\partial h}\Psi_{nh}(\bm{v}_0,\theta) \pm \kappa \left(\frac{\partial }{\partial h}\sigma_{nh}/(nh^r)^{1/2}\right) =0, \label{eq:hcri}
\end{align}
where $\kappa$ is a positive constant. Consider a scenario where $\Psi_{nh}(\bm{v}_0,\theta)$ shows an increasing trend as $h_n$ approaches zero. This implies that $\Psi_0(\bm{v}_0,\theta)>\Psi_{nh}(\bm{v}_0,\theta)- \kappa \sigma_{nh}/(nh^r)^{1/2}$. 
We then define our optimal finite sample bandwidth selector as
\begin{align*}
    h_n = \argmax_h \Psi_{nh}(\bm{v}_0,\theta)- \kappa \sigma_{nh}/(nh^r)^{1/2}.
\end{align*}
Similarly, when $\Psi_{nh}(\bm{v}_0,\theta)$ shows decreasing  trend as $h_n$ approaches zero, we define $h_n = \argmin_h \Psi_{nh}(\bm{v}_0,\theta)+ \kappa \sigma_{nh}/(nh^r)^{1/2}$. A natural choice for the constant $\kappa$ is $(1-\alpha)$-quantile of the standard normal distribution (i.e., $\zeta_{1-\alpha}$). Hence, $h_n$ is the bandwidth value that minimizes the difference between the true value $\Psi_0(\bm{v}_0,\theta)$ and the corresponding lower (upper) confidence bound.

\subsection{Undersmoothing criteria} \label{sec:criteria}
The asymptotic linearity results of Theorems \ref{th:movh} and \ref{th:fixedh} rely on estimating the corresponding nuisance parameters using an undersmoothed highly adaptive lasso. Specifically, Theorem \ref{th:movh} requires that 
\begin{itemize}
    \item[(a)] $ \left|P_n D_{CAR}^a(P_n,G_n^a,Q_n,\Psi_n,\theta)\right|=o_p\left((nh^r)^{-1/2}\right)$,
    \item[(b)] $ \left|P_n D_{CAR}^c(P_n,G_n,Q_n,\Psi_n,\theta)\right|=o_p\left((nh^r)^{-1/2}\right)$,
    \item[(c)] $ \left| P_n D_{CAR}^\Psi(P_n,G_n,\Psi_n,\theta)  \right|= o_p\left((nh^r)^{-1/2}\right)$,
\end{itemize}
where
 $D_{CAR}^a(P_n,G_n^a,Q_n,\Psi_n,\theta) = \frac{K_{h,\bm{v}_0}}{P_n K_{h,\bm{v}_0}}  (Q_n-\Psi_n)\left(\frac{G_n^a- I( A= d^\theta)}{G_n}\right)$,
 $D_{CAR}^c(P_n,G_n,Q_n,\Psi_n,\theta) = \frac{K_{h,\bm{v}_0}I( A= d^\theta)}{P_n K_{h,\bm{v}_0}}  (Q_n-\Psi_n)\left(\frac{G_n^c-\Delta^c }{G_n}\right)$
 and $D_{CAR}^\Psi(P_n,G_n,\Psi_n,\theta)=\frac{K_{h,v}\Delta^c I( A= d^\theta) }{G_n P_n K_{h,v}} \{ \Psi_n -I(Y >t) \} $.
 Motivated by our theoretical results, we propose the following practical L1-norm bound selection criteria for $G_n^a$, $G_n^c$ and $\Psi_n$:

\begin{align}
    \lambda_{n,G^a} &= \argmin_{\lambda} \left| B^{-1} \sum_{b=1}^B P_{n,b}^1  D_{CAR}^a(P_{n,b}^1,G_{n,\lambda,b}^a,G_{n,b}^c,Q_{n,b},\Psi_{n,b},\theta)\right|, \label{eq:lamga}\\
    \lambda_{n,G^c} &= \argmin_{\lambda} \left| B^{-1} \sum_{b=1}^B P_{n,b}^1  D_{CAR}^c(P_{n,b}^1,G_{n,b}^a,G_{n,\lambda,b}^c,Q_{n,b},\Psi_{n,b},\theta)\right|, \label{eq:lamgc}\\
    \lambda_{n,\Psi} &= \min \left\{ \lambda:  \left| B^{-1} \sum_{b=1}^B  P_{n,b}^1 D_{CAR}^\Psi(P_{n,b}^1,G_{n,b}^a,G_{n,b}^c,\Psi_{n,\lambda,b},\theta)  \right| \leq \frac{\sigma_{nh}}{(n \log n)^{1/2}} \right\}, \label{eq:lamm}
\end{align}
where $G_{n,b}^a$, $G_{n,b}^c$, $\Psi_{n,b}$ and $Q_{n,b}$ are  cross-validated highly adaptive lasso estimates of the corresponding nuisance parameters with the L1-norm bound based on
the global cross-validation selector. 

To achieve the Gaussian process convergence results in Theorem \ref{th:fixedh} (with a fixed $h$), we would require stronger conditions. Specifically, the conditions listed above must hold uniformly for any $ \bm{v} \in \mathcal{V}$, that is, $\sup_{v \in \mathcal{V}} \left|P_n D_{CAR}^a(P_n,G_n^a,Q_n,\Psi_n,\theta)\right|=o_p(n^{-1/2})$,\\ $\sup_{v \in \mathcal{V}} \left|P_n D_{CAR}^c(P_n,G_n,Q_n,\Psi_n,\theta)\right|=o_p(n^{-1/2})$ and $\sup_{v \in \mathcal{V}} \left| P_n D_{CAR}^\Psi(P_n,G_n,\Psi_n,\theta)  \right|= o_p(n^{-1/2})$. Accordingly, the practical criteria are given by
\begin{align}
    \lambda_{n,G^a}^{u} &= \argmin_{\lambda} \sup_{v \in \tilde{\mathcal{V}}}\left| B^{-1} \sum_{b=1}^B P_{n,b}^1  D_{CAR}^a(P_{n,b}^1,G_{n,\lambda,b}^a,G_{n,b}^c,Q_{n,b},\Psi_{n,b},\theta)\right|, \label{eq:lamgua}\\
    \lambda_{n,G^c}^{u} &= \argmin_{\lambda} \sup_{v \in \tilde{\mathcal{V}}}\left| B^{-1} \sum_{b=1}^B P_{n,b}^1  D_{CAR}^c(P_{n,b}^1,G_{n,b}^a,G_{n,\lambda,b}^c,Q_{n,b},\Psi_{n,b},\theta)\right|, \label{eq:lamguc}\\
    \lambda_{n,\Psi}^{u} &= \min \left\{ \lambda: \sup_{v \in \tilde{\mathcal{V}}} \left| B^{-1} \sum_{b=1}^B  P_{n,b}^1 D_{CAR}^\Psi(P_{n,b}^1,G_{n,b}^a,G_{n,b}^c,\Psi_{n,\lambda,b},\theta)  \right|\leq \frac{\sigma_{nh}}{(n \log n)^{1/2}} \right\}, \label{eq:lammu}
\end{align}
where $\tilde{\mathcal{V}}=(\tilde{\bm{v}}_{01},\cdots,\tilde{\bm{v}}_{0\alpha})$ includes $\alpha$ sample points of $\mathcal{V}$. The asymptotic linearity result of Theorem \ref{th:fixedh} with a fixed bandwidth $h$  relies on similar undersmoothing criteria as those listed in (a)--(c) but with $h$ being replaced by one, and thus, the same practical undersmoothing criteria as(\ref{eq:lamga})-(\ref{eq:lamm}) can be considered. 

In Theorems \ref{th:movh} and \ref{th:fixedh}, we show that the proposed  estimators are asymptotically efficient when they solve the efficient influence function, that is, $P_n D_{CAR}^\diamond = o_p\left((nh^r)^{-1/2}\right)$ and $P_n D_{CAR}^\diamond = o_p(n^{-1/2})$, respectively, for $\diamond \in \{a,c,\Psi\}$. The $\argmin$ criteria proposed in this section correspond to the most efficient  estimators for a given data. Increasing $\lambda$ results in decreasing the empirical mean of pseudo score functions $\frac{K_{h,\bm{v}_0}}{P_n K_{h,\bm{v}_0}}  (Q_n-\Psi_n)\left(G_{n,\lambda}^a- I( A= d^\theta)\right)$ and $\frac{K_{h,\bm{v}_0}I( A= d^\theta)}{P_n K_{h,\bm{v}_0}}  (Q_n-\Psi_n)\left(G_{n,\lambda}^c-\Delta^c \right)$ and increasing the variance of $(G_{n,\lambda}^a)^{-1}$ and $(G_{n,\lambda}^c)^{-1}$. At a certain point in the grid of $\lambda$, decreases in the empirical mean of the pseudo score functions are insufficient for satisfying the $\argmin$ conditions, which $\left|P_n D_{CAR}^\diamond \right|$ starts increasing on account of $(G_{n,\lambda}^a)^{-1}$ and $(G_{n,\lambda}^c)^{-1}$. Unlike $\left|P_n D_{CAR}^a\right|$ and $\left|P_n D_{CAR}^c\right|$, $\left|P_n D_{CAR}^\Psi\right|$ decreases as $\lambda$ increases which motivates the undersmoothing selection criteria (\ref{eq:lamm}) and (\ref{eq:lammu}). Specifically, the sectional variation norm $\lambda_{n,\Psi}$ and $\lambda_{n,\Psi}^{u}$ are defined as the smallest value (larger than that chosen by the cross-validation selector) for which the left-hand side of the condition is
less than $\sigma_{nh}/(n \log n)^{1/2}$ where $\sigma_{nh}$ be a consistent estimator of $\sigma_{0} = [E\{D^{*2}_{\bm{v}_0}(P_0)\}]^{1/2}$ with  $D^{*2}_{\bm{v}_0}(P_0)$ being the efficient influence function for $\Psi_{nh}(\bm{v}_0,\theta)$. The cutpoint will lead to a bias that is of a smaller magnitude of the standard error, and thus, won't affect the coverage.

\begin{remark}
    The variance estimator $\sigma_{nh}^2$ is calculated using the efficient influence function. One can, instead, use a conservative variance estimator obtained by an influence function where the nuisance parameters are assumed to be known. That is 
    \begin{align} \label{eq:conservative-variance}
        \sigma_{nh}^2 &= P_n \left\{\frac{K_{h,\bm{v}_0}}{P_n K_{h,\bm{v}_0}}\Psi_n(\theta) - \Psi_{0}(\bm{v}_0,\theta)\right\}^2,
    \end{align}
    where  $\Psi_{0}(\bm{v}_0,\theta)$ can be replaced by their corresponding estimators. The variance estimator (\ref{eq:conservative-variance}) can also be used in Section \ref{sec:bandw} where we proposed an adaptive bandwidth selector. 
\end{remark}

\section{Theoretical results} \label{sec:theory}

The asymptotic properties of our estimator relies on the following assumptions. 

\begin{assumption}[Complexity of nuisance functions] \label{assump:cadlag}
  The functions $Q_0$, $\Psi_0$, $G_{0}^a$ and $G_{0}^c$ are c\`{a}dl\`{a}g with finite
  sectional variation norm.
\end{assumption}

Assumption \ref{assump:cadlag}  characterizes the global smoothness assumption of the true functions which is much weaker than than local smoothness assumptions like those characterized by an assumption
utilizing H{\"o}lder balls~\citep[e.g.,][]{robins2008higher, robins2017minimax,
mukherjee2017semiparametric}. The assumption facilitates 
the fast
convergence rate of $n^{-1/3}
\log(n)^{p/2}$ obtainable by the highly adaptive lasso (regardless of dimensionality
$p$) \citep{van2017generally,van2017uniform,bibaut2019fast}. The latter rate is generally faster than the typical
minimax rates that appear under differing, but similarly positioned, assumptions
regarding function classes in nonparametric estimation. In fact, we expect that almost all true conditional mean models that we have to deal with in real data will satisfy Assumption \ref{assump:cadlag}. The sectional variation is defined in more detail in Appendix \ref{sec:secnorm}.

\begin{assumption} (causal inference and censoring assumptions) \label{assump:basic}
For $i=1,\dotsc,n$,
\begin{itemize}
\item[(a)] \emph{Consistency} of the observed outcome: $T_i = T_i^{A_i}$;
\item[(b)] \emph{Unconfoundedness} of the treatment and censoring mechanisms: $A_i \perp
  T_i^a | \mathbf{W}_i,~\forall a \in \{0,1\}$, and $\Delta_i^c \perp T_i \mid A_i, \bm{W}_i$;
\item[(c)] \emph{Positivity} of the treatment and censoring mechanisms: $\min_{a,\delta} G(A=a,\Delta^c=\delta \mid \bW)>\gamma$ for all $\mathbf{W} \in \mathcal{W}$ where $\gamma$ is a small positive constant.
\end{itemize}
\end{assumption}

Assumption \ref{assump:basic} lists  standard causal identifiability assumptions. The consistency assumption ensures that a subject's potential outcome is not affected by other subjects' treatment levels and there are no different versions of treatment. Assumption \ref{assump:basic}b is a version of the coarsening at random assumption that imposes conditional independence between (1) treatment and potential outcomes given the measured covariates; and (2) censoring indicator and the observed outcome given the treatment and measured covariates. Positivity states that every subject has some chance of receiving treatment level $a$ with censoring status $\delta$ regardless of covariates. Using Assumption \ref{assump:basic} the full data risk of $\Psi$ can be identified using the observed data as
\[
E_{P_X} \int_\theta \ell_{\theta}(\Psi,x) dF(\theta) = E_{P_0}  L_{G_0}(\Psi),
\]
where $L_{G_0}(\Psi)$ and $\ell_{\theta}(\Psi,x)$ were defined in Section \ref{sec:notation}. Even if the Assumption \ref{assump:basic}a-b  do not hold, it may still be of interest to estimate $E_{P_X} \int_\theta \ell_{\theta}(\Psi,x) dF(\theta)$ using $P_n L_{G_0}(\Psi)$ as an adjusted measure of association.

\begin{remark}
    The censoring at random assumption implies that $C \perp T \mid A,\bm{W}$ on $C<T$ which is weaker than $C \perp T \mid A,\bm{W}$ \citep[Chapter 1]{van2003unified}. However, because the additional restriction can not be  identified using the observed data, it will not impose any restriction on the tangent space and thus, the statistical model remains nonparametric. 
\end{remark}

\begin{assumption}[Undersmoothing]\label{assump:proj}
The level of undersmoothing satisfies the following conditions:
\begin{itemize}
\item[(a)] Fixed $h$: Let ${f}_\phi^\Psi$, ${f}_\phi^a$ and ${f}_\phi^c$   be the projections of $f^\Psi = \Delta^c I( A= d^\theta)/G_0$, $f^a = (Q_0-\Psi_0)/ G_{0}$ and $f^c = I(A=d^\theta)(Q_0-\Psi_0)/ G_{0}$ onto a linear
  span of basis functions $\phi_{s,j}$ in $L^2(P)$, for $\phi_{s,j}$
  satisfying conditions (\ref{eq:basis3-fixedh}), (\ref{eq:basis1-fixedh}) and (\ref{eq:basis2-fixedh}) of Lemma  \ref{lem:ucondition-fixedh}. Then, $\lVert f^\Psi - {f}_\phi^\Psi
  \rVert_{2,\mu} = O_p(n^{-1/4})$, $\lVert f^a - {f}_\phi^a
  \rVert_{2,\mu} = O_p(n^{-1/4})$ and $\lVert f^c - {f}_\phi^c
  \rVert_{2,\mu} = O_p(n^{-1/4})$.
\item[(b)] Converging $h$: The functions $f_h^\Psi = K_{h,\bm{v}_0}f^\Psi$, $f_h^a= K_{h,\bm{v}_0}f^a$ and $f_h^c= K_{h,\bm{v}_0}f^c$ are such that $\tilde f^\Psi = h^rf_h^\Psi$, $\tilde f^a=h^rf_h^a$  and $\tilde f^c=h^rf_h^c$ are c\`{a}dl\`{a}g with finite sectional
variation norm. Let $\tilde{f}_\phi^\Psi$, $\tilde{f}_\phi^a$ and $\tilde{f}_\phi^c$ be projections of $\tilde f^\Psi$, $\tilde f^a$  and $ \tilde f^c$ onto the linear
span of the basis functions $\phi_{s,j}$ in $L^2(P)$, where $\phi_{s,j}$
satisfies condition (\ref{eq:basis3}),  (\ref{eq:basis1}) and (\ref{eq:basis3}) of Lemma \ref{lem:ucondition-movh}. Then, $\lVert \tilde f^\Psi - \tilde{f}_\phi^\Psi
  \rVert_{2,\mu} = O_p(n^{-1/3})$, $\lVert \tilde f^a - \tilde{f}_\phi^a
  \rVert_{2,\mu} = O_p(n^{-1/3})$ and $\lVert \tilde f^c - \tilde{f}_\phi^c
  \rVert_{2,\mu} = O_p(n^{-1/3})$.
\end{itemize}
where $\mu$ is an appropriate measure (i.e., the Lebesgue or the counting measure).
\end{assumption}

Assumption \ref{assump:proj} is arguably the most important of the conditions to retain the asymptotic linearity of $ \Psi_{nh}(\bm{v}_0,\theta)$. Assumption \ref{assump:proj} (a) states that when the bandwidth $h$ is fixed and the estimated coarsening mechanisms (i.e., $G^a$ and $G^c$) and $\Psi_n$ are sufficiently undersmoothed, the generated features can approximate functions $f^a$, $f^c$ and $f^\Psi$ with $n^{-1/3}$ rate. Lemma \ref{lem:ucondition-fixedh} in the supplementary material provides theoretical undersmoothing conditions required. Assumption \ref{assump:proj} (b) correspond to the case where the bandwidth is allows to converge to zero and requires similar conditions as in part (a).  Under Assumption \ref{assump:cadlag}, ${f}_\phi^\Psi$, ${f}_\phi^a$ and ${f}_\phi^c$ (and $\tilde f^\Psi$, $\tilde f^a$ and $\tilde f^c$) will fall into the class of c\`{a}dl\`{a}g functions with finite sectional variation norm. Hence these functions can be approximated using the highly adaptive lasso generated basis functions at $n^{-1/3}$ rate (up to a $\log n$ factor). Importantly, the required rate in Assumption \ref{assump:proj} (a) is $O_p(n^{-1/4})$ which is slower than the $O_p(n^{-1/3})$ rate obtained by the highly adaptive lasso estimator. In Remark \ref{rem:slowrate}, we discuss that, under a certain smoothness assumption,  the required rate in Assumption \ref{assump:proj} (b) can also be slowed down to $O_p(n^{-1/4})$.  Consequently, this key assumption is not particularly strong. 


\begin{assumption}  \label{assump:margin}
There exist constants $\kappa \geq 0$ and $C>0$ such that, for all $l>0$, $pr(0<|\bm{S}^\top{{\theta}}_0|<l) 	\leq C l^\kappa$.
\end{assumption}

Assumption \ref{assump:margin} is the margin assumption of \cite{audibert2007fast} which can also be found in \cite{tsybakov2004optimal}. The assumption is needed to derive the rate of convergence of $\theta_{nh}(\bm{v}) = \argmin_{\theta \in \Theta} \Psi_{nh}(\bm{v},\theta)$ (Theorem \ref{th:thetrate}). The extreme case of $\kappa=\infty$ corresponds to a case where there is a margin around zero, that is, $pr(0<|\bm{S}^\top{{\theta}}_0|<l) 	=0$. As shown in the proof of Theorem \ref{th:thetrate} the latter leads to the fastest rate of convergence for $\theta_{nh}(\bm{v})$.

\begin{theorem}\label{th:halrate}
Let $\Psi$ be a functional parameter that is identified as $\Psi_0 = \argmin_{\Psi \in \mathcal{F}_{\lambda}}P_0 L_{G_0}(\Psi)$ where the loss function $L$ is defined in \ref{eq:lossf}. Assume the functional parameter space $\mathcal{F}_\lambda$ contained in all $p$-variate cadlag functions with a bound on its variation norm $\lambda$. Let $\Psi_n = \argmin_{\Psi \in \mathcal{F}_{\lambda_n}}P_n L_{G_n}(\Psi)$ where $\lambda_n$ is the cross-validated selector of $\lambda$ and let $d_0(\Psi_n,\Psi_0) = P_0 L_{G_0}(\Psi_n)-P_0 L_{G_0}(\Psi_0)$. Also assume $\tilde{\mathcal{F}}_{\tilde \lambda} = \{L_{G_0}(\Psi)-L_{G_0}(\Psi_0): \Psi \in \mathcal{F}_{\lambda}\}$ falls in to the class of cadlag functions with a bound on its variation norm $\tilde \lambda$. Then, under Assumptions \ref{assump:cadlag} and \ref{assump:basic}, $d_0(\Psi_n,\Psi_0) = O_p(n^{-2/3} (\log n)^{4(p-1)/3})$.
\end{theorem}

Theorem \ref{th:halrate} generalizes the results of \cite{van2017generally} to settings where the estimation of $\Psi_0$ involves estimation of nuisance parameters $G_0$.  The proof of the results requires  techniques to handle  the additional complexities due to the presence of such nuisance parameters.  The theorem  relies on Assumptions \ref{assump:cadlag} without the need for undersmoothing the nuisance function estimators.

\begin{theorem}\label{th:movh}
Let $ \Psi_0(\bm{v}_0,\theta)=E \{I(Y^\theta>t) \mid \bm{V}=\bm{v}_0\}$ be $J$-times continuously differentiable at $\bm{v}_0$. Suppose the support of $\bW$ is uniformly bounded, i.e., $\mathcal{W} \subseteq
[0,\tau]^p$ for some finite constant $\tau$. Let $\Psi_{n,\lambda_{n}^\Psi}$, $G^a_{n,\lambda_{n}^a}$ and  $G^c_{n,\lambda_{n}^c}$ be highly adaptive lasso
estimators of $\Psi_0$, $G_{0}^a$ and $G_{0}^c$ with $L_{1}$-norm bounds equal to $\lambda_{n}^\Psi$, $\lambda_{n}^a$ and $\lambda_{n}^c$, respectively. Assuming that \begin{itemize}
    \item[(i)] the bandwidth $h$ satisfies $h^r \rightarrow 0$ and $nh^{r3/2} \rightarrow \infty$,
    \item[(ii)] Assumptions \ref{assump:cadlag}, \ref{assump:basic} and and \ref{assump:proj} hold,
\end{itemize}
we have
\begin{equation*}
   \Psi_{nh}(\bm{v}_0,\theta)-  \Psi_0(\bm{v}_0,\theta) = (P_n-P_0) D_{\bm{v}_0}^*(P_0) + h^J B_0(J,\bm{v}_0)+
    o_p\left((nh^r)^{-1/2}\right),
\end{equation*}
where $\Psi_{nh}(\bm{v}_0,\theta)$ is obtained using a ($J-1$)-orthogonal kernel, $G_{n,\lambda_{n}^a,\lambda_{n}^c} =G^a_{n,\lambda_{n}^a}G^c_{n,\lambda_{n}^c} $ and $D_{\bm{v}_0}^*(P_0)$ is the corresponding efficient influence function defined in the proof of the theorem. Moreover, assuming that $nh^{r+2J} \rightarrow 0$,$ (nh^r)^{1/2}\{\Psi_{nh}(\bm{v}_0,\theta)-  \Psi_0(\bm{v}_0,\theta)\}$ converges to a mean zero normal random variable. 
\end{theorem}

Theorem \ref{th:movh} gives conditions under which the proposed kernel estimator $\Psi_{nh}(\bm{v}_0,\theta)$ is consistent for $\Psi_0(\bm{v}_0,\theta)$, and it also gives the corresponding rate of convergence. In general this result follows if the bandwidth decreases with sample size sufficiently slowly, and if  the nuisance functions $G$ is estimated sufficiently well. The standard local linear kernel smoothing requires that $nh^{3r} \rightarrow \infty$ to control the variance \citep[Chapter 5]{wasserman2006all}. Using the entropy number of the class of c\`{a}dl\`{a}g functions with finite sectional variation norm, we showed that bandwidth can converge to zero at much faster rate (i.e., $nh^{3r/2} \rightarrow \infty$) than the standard results. Condition (ii) is the standard causal assumptions. Condition (iii) indicates the level of undersmoothing of the nuisance parameter estimators required to achieve asymptotic linearity of our estimator. Section \ref{sec:uniform} in the supplementary material provides uniform convergence results with a rate for our regimen-response curve estimator $\Psi_{nh}(\bm{v}_0,\theta)$ for all $\bm{v}_0 \in \mathcal{V}$. 

\begin{remark}[The optimal bandwidth rate]\label{rem:opthrate}
      Theorem \ref{th:movh}  shows that in order to eliminate the bias term, we must undersmooth the kernel such that $nh^{r+2J} \rightarrow 0$ (i.e., $h<n^{-\frac{1}{r+2J}}$). In combination with rate condition (i) in the theorem, we have $n^{-\frac{2}{3r}}<h<n^{-\frac{1}{r+2J}}$. Note that for the latter constraint to hold we must have $J>r/4$. Also, the theoretical undersmoothing conditions in Lemma \ref{lem:ucondition-movh} require $J>r$.    Hence, the constraint implies  that the optimal bandwidth rate is $h_n = n^{-\frac{1}{r+2J}}$ with $J>r$. 
\end{remark}

\begin{remark}[Assumption \ref{assump:proj} (b)] \label{rem:slowrate}
  We can allow slower rate of convergence of $O_p(n^{-1/4})$ in Assumption \ref{assump:proj} (b) (i.e., $\lVert \tilde f^\Psi - \tilde{f}_\phi^\Psi
  \rVert_{2,\mu} = O_p(n^{-1/4})$, $\lVert \tilde f^a - \tilde{f}_\phi^a
  \rVert_{2,\mu} = O_p(n^{-1/4})$ and $\lVert \tilde f^c - \tilde{f}_\phi^c
  \rVert_{2,\mu} = O_p(n^{-1/4})$) at the cost of requiring higher level of smoothness for $ \Psi_0(\bm{v}_0,\theta)=E \{I(Y^\theta>t) \mid \bm{V}=\bm{v}_0\}$. Specifically, the optimal bandwidth rate would be $h_n = n^{-\frac{1}{r+2J}}$ with $J>\frac{5}{2}r$.
\end{remark}

\begin{theorem}\label{th:fixedh}
Let $\Psi_{0h}(\bm{v}_0,\theta)=(P_0 K_{h,\bm{v}_0})^{-1} \{K_{h,\bm{v}_0} \Psi_0(\theta)\}$. Suppose the support of $\bm{W}$ is uniformly bounded, i.e., $\mathcal{W} \subseteq
[0,\tau]^p$ for some finite constant $\tau$. Let $\Psi_{n,\lambda_{n}^\Psi}$, $G^a_{n,\lambda_{n}^a}$ and  $G^c_{n,\lambda_{n}^c}$ be highly adaptive lasso
estimators of $\Psi_0$, $G_{0}^a$ and $G_{0}^c$ with $L_{1}$-norm bounds equal to $\lambda_{n}^\Psi$, $\lambda_{n}^a$ and $\lambda_{n}^c$, respectively. Under
Assumptions ~\ref{assump:cadlag}, \ref{assump:basic} and \ref{assump:proj},  the estimator
$\Psi_{nh}(\bm{v}_0,\theta)$ will be asymptotically linear. That is
\begin{equation*}
   \Psi_{nh}(\bm{v}_0,\theta)-  \Psi_{0h}(\bm{v}_0,\theta) = (P_n-P_0) D_{h,\bm{v}_0}^*(P_0) +
    o_p(n^{-1/2}),
\end{equation*}
where $G_{n,\lambda_{n}^a,\lambda_{n}^c} =G^a_{n,\lambda_{n}^a}G^c_{n,\lambda_{n}^c} $ and $D_{h,\bm{v}_0}^*(P_0)$ is the corresponding efficient influence function defined in the proof of the theorem.
Moreover,   
$\sqrt n \{\Psi_{nh}(\theta)  -  \Psi_{0h}(\theta)\}$ 
converges weakly as a random
element of the cadlag function space endowed with the supremum norm to a Gaussian process with covariance
structure implied by the covariance function $\rho (\bm{v},\bm{v}') = P_0 D^*_{h,\bm{v}}(P_0) D^*_{h,\bm{v}'}(P_0)$.

\end{theorem}

Theorem \ref{th:fixedh} strengthens  the results of Theorem \ref{th:movh} when the bandwidth is fixed. Specifically, it  shows that for a fixed bandwidth $h>0$,  $\sqrt n \{\Psi_{nh}(\theta)-  \Psi_{0h}(\theta)\} $ convergences to a Gaussian process. This is an important results because in enables us to construct simultaneous confidence intervals for the entire curve $\Psi_{0h}(\bm{v},\theta),$ $\bm{v} \in \mathcal{V}$ and $\theta \in \Theta$.

\begin{theorem}\label{th:thetrate}
Let $\theta_0(\bm{v}) = \argmin_{\theta \in \Theta} \Psi_0(\bm{v},\theta)$ and $\theta_{nh}(\bm{v}) = \argmin_{\theta \in \Theta} \Psi_{nh}(\bm{v},\theta)$ where both $\Psi_0$ and $\Psi_{nh}$ are defined in Theorem \ref{th:movh}. Then, under Assumptions \ref{assump:cadlag}-\ref{assump:margin}, 
$\|\theta_{nh}(\bm{v}_0)-\theta_0(\bm{v}_0)\|_{2}=O_p\left( n^{\frac{(r-4J)(2\kappa+4)}{(8J+4r)(3\kappa+8)}} \right)$.
\end{theorem}

Characterizing the minimizer (or in general optimizer) of the estimated regimen-response curve is of interest as it forms an optimal individualized decision rule.    Under the margin assumption (i.e., Assumption \ref{assump:margin}) and using a novel iterative empirical process theory, in Theorem \ref{th:thetrate}, we derive the convergence rate of a minimizer of a nonparametrically estimated function $\Psi_{nh}(\bm{v},\theta)$.


\section{Simulation Studies}

We examine the performance of our estimator through two simulation studies. Within these studies, we compare the results of our estimator against theoretical values and random forest based estimation, demonstrating the practical benefits of using undersmoothed highly adaptive lasso based estimation when the data generating mechanism is unknown.

In both scenarios $W \sim \text{Unif}(-1, 1)$ and $V \sim \text{Unif}(-1.3, 1.3)$. Additionally $\Delta|W,V \sim \text{Bernoulli}\{\text{expit}(0.5W-0.5V)\}$ and $Y|A,W,V \sim \text{Bernoulli}\{\text{expit}(0.5W + V^2 + 2VW - 2AVW)\}.$ For each scenario we sample $n \in \{240, 480, 960, 2400\}$  observations, applying our estimator for each sample. The results displayed below give average effects  360 replicates of each scenario and sample size pair where we set $\xi(\theta,\bm{V})=1$. The latter implies that $\Psi_n$ is obtained by solving a set of simpler score equations (compared with $\xi(\theta,\bm{V})=G\{(A,\Delta^c)\mid \bm{V}\}$) which may call for more undersmoothing in order to solve $D_{CAR}^\Psi(P_n,G_n,\Psi_n)$. 

The fundamental difference between the scenarios comes in the treatment assignment mechanism. In Scenario 1, we have $A \sim \text{Bernoulli}(0.5)$, giving us a randomized trial scenario, whereas in Scenario 2, we have $A \sim \text{Bernoulli}\{\text{expit}(0.7W + 0.5V + 0.5VW)\}.$ The highly adaptive lasso estimate of $G^a(A\mid W,V)$ requires solving fewer score functions in Scenario 1 compared with Scenario 2. In fact the only score equation that to be solved in the former is $P_n\{A-G^a_n(A\mid W,V)\}=0$. Hence, higher level of undersmooting is required to solve $D_{CAR}^a(P_n,G_n^a,Q_n,\Psi_n,\theta)$.  In general, the simpler the $G\{(A,\Delta^c)\mid W,V\}$, the more undersmooting is needs to solve $D_{CAR}^a(P_n,G_n^a,Q_n,\Psi_n,\theta)$ and $D_{CAR}^c(P_n,G_n^a,Q_n,\Psi_n,\theta)$. 

In all simulations we are targeting the estimands $\Psi_{0h}({v}_0, \theta)$ and $\Psi_{0}({v}_0, \theta)$ using the estimator $\Psi_{nh}({v}_0, \theta)$ where ${v}_0 = 0.5$ and we consider the bandwidths $h \in n^{-1/3} \times\{0.5,1,1.5,2,2.5,3,4\}$. Recall that the optimal bandwidth rate is $n^{\frac{-1}{2J+r}}$ with $J>r$. In our simulation studies $r=1$ and we are setting $J=r$ leading to a rate $n^{-1/3}$ which is faster than $n^{-1/5}$ obtained by setting $J=2>r$. The choice of rate matters in cases where the target parameter is $\Psi_0$ and the goal is to show that the scaled biases remain constant and converges are nominal even under the faster rate of $n^{-1/3}$ (Figures \ref{fig:convergh} and \ref{fig:convergopth}).    Coverage results are also given for $\Psi_0({v}_0, \theta)$ where for each sample $h_n$ is chosen from the possible bandwidths as described in Section \ref{sec:bandw}.
In all cases highly adaptive lasso was implemented using the \textsf{R} package \textsf{hal9001} considering quadratic basis functions for up to 2-way interactions between all predictor variables.

\begin{figure}[ht]
    \centering
    \includegraphics[width = 1\textwidth]{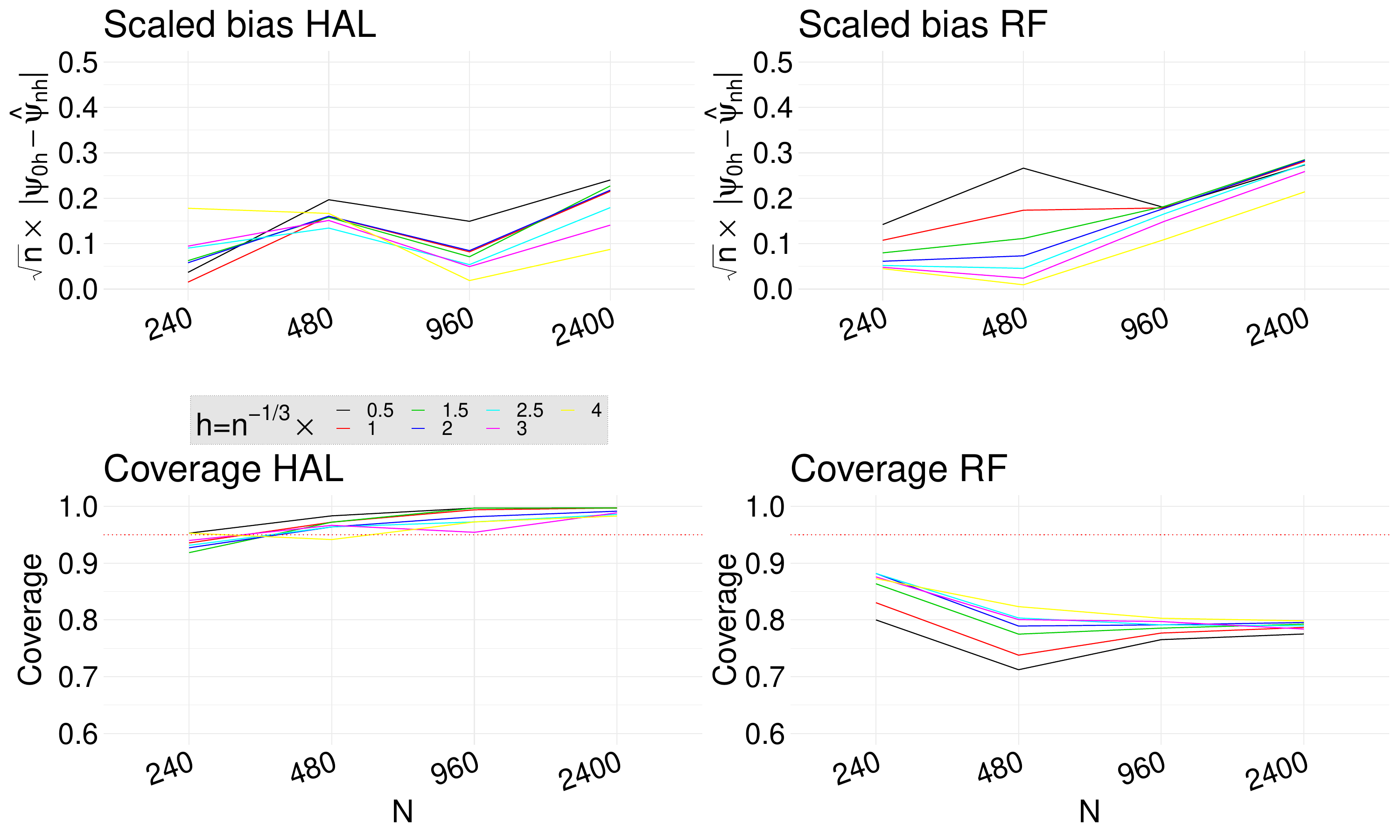}
    \caption{Simulation studies Scenario 1: The target parameter of interest is $\Psi_{0h}$. The plots show the scaled bias and coverage rate when the weight functions are estimated using an undersmoothed highly adaptive lasso (HAL) and a random forest (RF). } 
    \label{fig:fixedh}
\end{figure}

We estimate the weight functions (i.e., propensity score and censoring indicator models) using an undersmoothed highly adaptive lasso (HAL) and a random forest (RF). In both cases, $\Psi_n$ is a highly adaptive lasso estimate of $\Psi_0$ obtained based on the loss function (\ref{eq:lossf}). Figure \ref{fig:fixedh} shows the scaled bias and the coverage rate of the estimators when the target parameter is $\Psi_{0h}$. The results confirm our theoretical result presented in Theorem \ref{th:fixedh}. The scaled bias if the estimator with undersmoothed HAL is considerably smaller than the one obtained using RF. Importantly, while the coverage rate of undersmoothed HAL estimator remains close to the nominal rate of 95\%, the coverage rate of RF based estimator sharply declines as sample size increases. To assess our theoretical result of Theorem \ref{th:movh}, we plotted the the scaled bias and the coverage rate of the estimators when the target parameter is $\Psi_{0}$. Figure \ref{fig:convergh} shows that the undersmoothed HAL outperforms RF based estimators. Figures \ref{fig:supfixedh}-\ref{fig:supconvergh} in the Supplementary Material show that similar results hold in Scenario 2. Finally, Figure \ref{fig:convergopth} shows that our proposed data-adaptive bandwidth selector performs very well when combined with undersmoothed HAL estimators of weight functions. Specifically, even with smaller sample size of $n=240$, our approach results in nearly 95\%  coverage rate and as the sample size increases it quickly recovers the nominal rate. Figure \ref{fig:psinps1} displays $\Psi_{nh}(v_0=0.5,\theta)$ against $\theta$ values. As the sample size increases the bias reduces across all the bandwidth values. Figure \ref{fig:thetrate} in the Supplementary Material shows the scaled bias of $\theta_{nh}$ when the weight functions are estimated using an undersmoothed highly adaptive lasso. The plot shows that the scaled bias is relatively constant across different sample sizes thereby confirming our result in Theorem \ref{th:thetrate} (by setting $J=r=1$ and $\kappa\rightarrow \infty$).

\begin{figure}[ht]
    \centering
    \includegraphics[width = 1\textwidth]{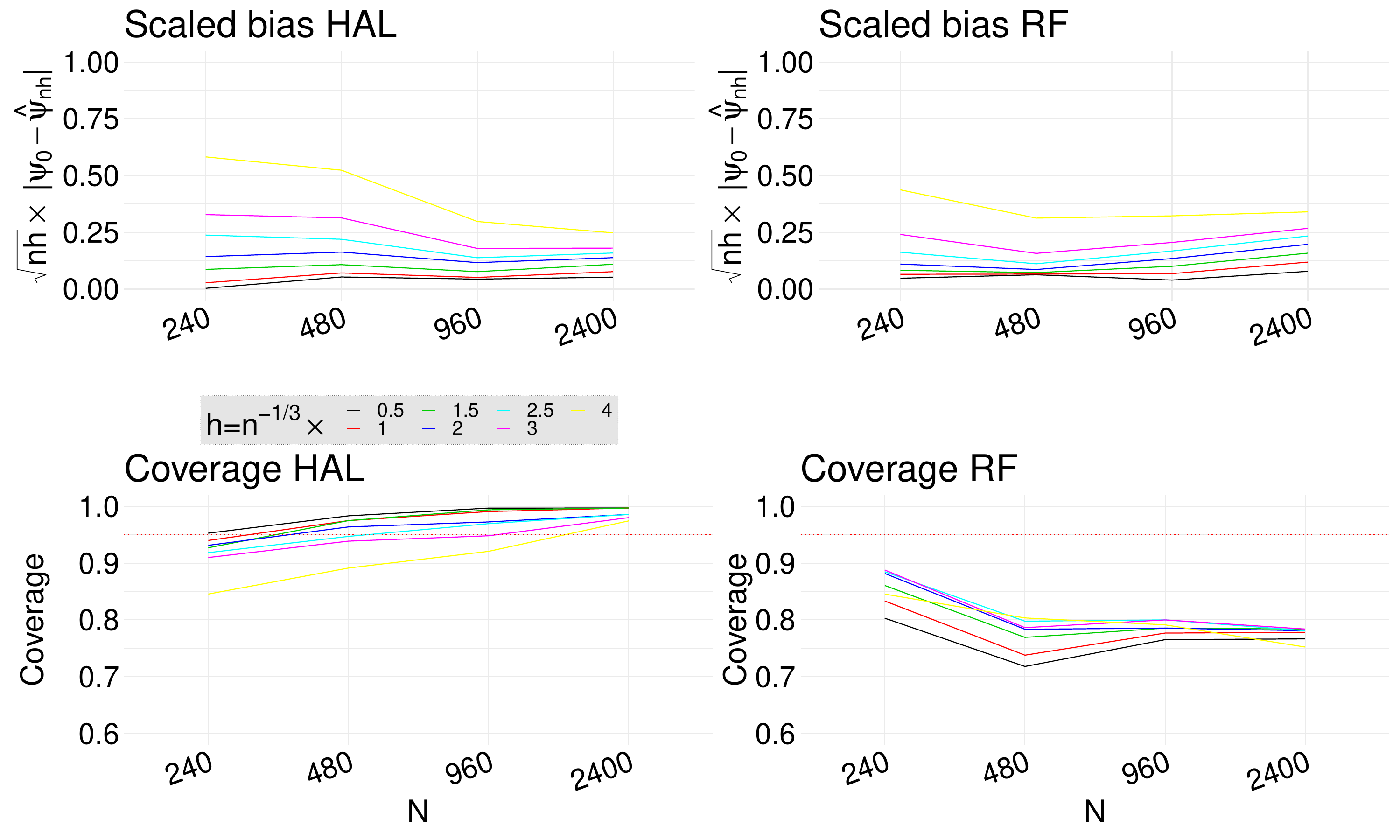}
    \caption{Simulation studies Scenario 1: The target parameter of interest is $\Psi_{0}$. The plots show the scaled bias and coverage rate when the weight functions are estimated using an undersmoothed highly adaptive lasso (HAL) and a random forest (RF).} 
    \label{fig:convergh}
\end{figure}

\begin{figure}[ht]
    \centering
    \includegraphics[width = .9\textwidth]{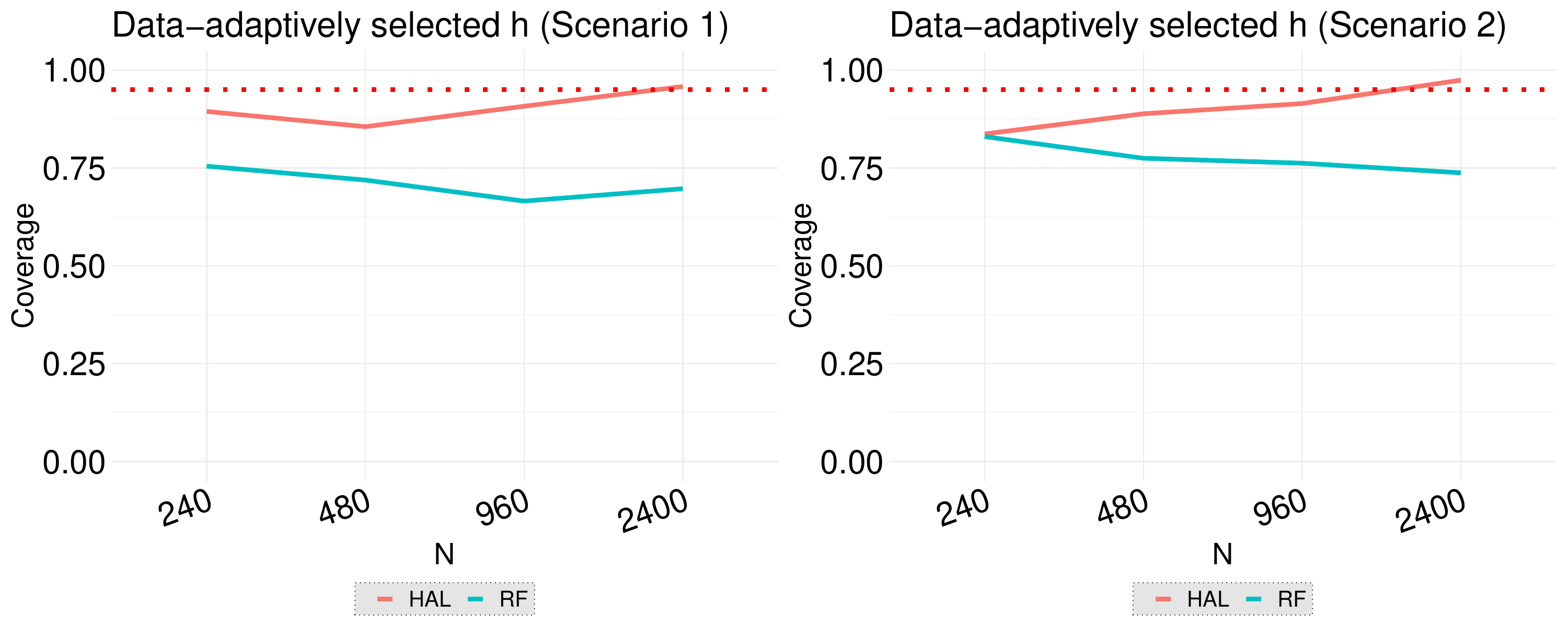}
    \caption{Simulation studies: The target parameter of interest is $\Psi_{0}$. The dotted line shows the nominal rate. The optimal bandwidth $h$ is selected using the proposed data adaptive approach.  The plots show the coverage rates in Scenarios 1 and 2 where the weight functions are estimated using an undersmoothed highly adaptive lasso (HAL) and a random forest (RF).} 
    \label{fig:convergopth}
\end{figure}

\begin{figure}[ht]
    \centering
    \includegraphics[width = 1\textwidth]{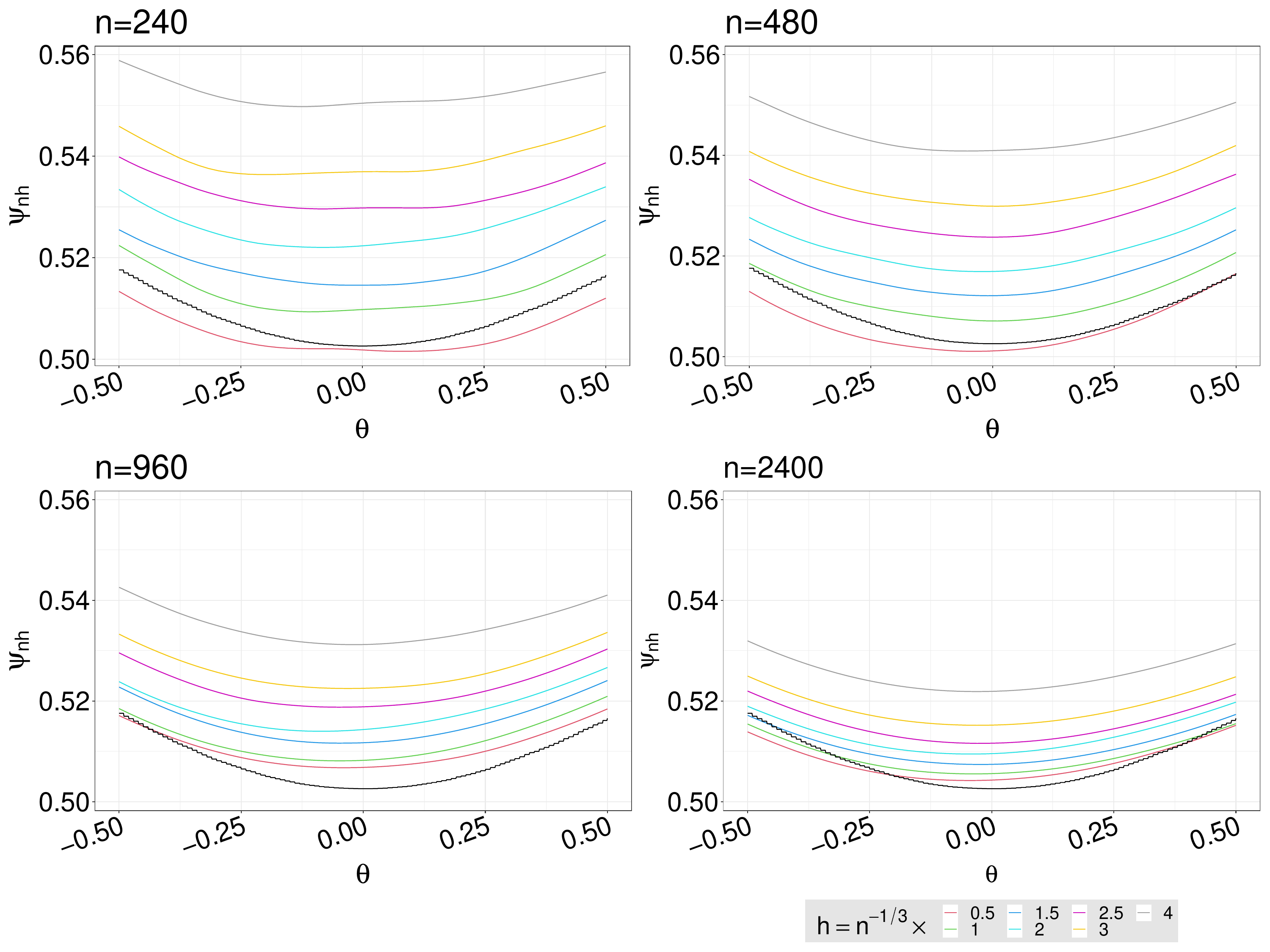}
    \caption{Simulation studies Scenario 1: The target parameter of interest is $\Psi_{0}$ (black solid line). The plots show $\Psi_{nh}$ for different bandwidths when the weight functions are estimated using an undersmoothed highly adaptive lasso and $v_0=0.5$.} 
    \label{fig:psinps1}
\end{figure}

\section{Concluding Remarks}

\subsection{Multi-stage decision making} A natural and important extension of our work is to learn dynamic treatment regimens in multi-stage problems. In settings with many decision points (e.g., mobile health), it would be beneficial to consider a stochastic class of decision rules rather than a deterministic one \citep{kennedy2019nonparametric,luckett2019estimating}.

\subsection{Score based undersmoothing criteria} The practical undersmoothing criteria proposed in Section \ref{sec:criteria} requires calculating the efficient influence function and deriving the relevant $D_{CAR}$ terms. The latter can be challenging and tedious, for example, in multi-stage settings (i.e., time-varying treatment) with many decision points. This motivates the derivation of score based criteria which requires only the score of the corresponding nuisance function. This will, in fact, resemble closely the theoretical undersmoothing conditions listed in Lemmas \ref{lem:ucondition-movh}, \ref{lem:ucondition-movh} and \ref{lem:ucondition-uconvergence} in Appendix and Supplementary materials. For example, for the treatment assignment model, we can consider
\begin{equation}\label{eq:score}
 \lambda_{n,G^a} = \argmin_{\lambda}  B^{-1} \sum_{b=1}^B \left[ \sum_{(s,j) \in
 \mathcal{J}_n} \frac{1}{ \lVert \beta_{n,\lambda,b} \rVert_{L_1}}
 \bigg\lvert P_{n,b}^1 \tilde S_{s,j}(\phi,G_{n,\lambda,b})
 \bigg\rvert \right],
\end{equation}
in which $\lVert \beta_{n,\lambda} \rVert_{L_1} = \lvert \beta_{n,\lambda,0}
\rvert + \sum_{s \subset\{1, \ldots, d\}} \sum_{j=1}^{n} \lvert
\beta_{n,\lambda,s,j} \rvert$ is the $L_1$-norm of the coefficients
$\beta_{n,\lambda,s,j}$ in the highly adaptive lasso estimator $G_{n,\lambda}$
for a given $\lambda$, and $\tilde S_{s,j}(\phi, G_{n,\lambda,b}) =
\phi_{s,j}(W) \{A_{} - G_{n,\lambda,b}(1 \mid W)\}\{G_{n,\lambda,b}
(1 \mid W)\}^{-1}$. Note that $S_{s,j}(\phi, G_{n,\lambda,b}) $ involves the score equation for the propensity score model (i.e., $\phi_{s,j}(W) \{A_{} - G_{n,\lambda,b}(1 \mid W)\}$) and the inverse of the estimate propensity score. A more detailed discussion of the score based criteria is provided in Section 4.2 of \cite{ertefaie2022nonparametric}. However, the theoretical properties of these criteria and their performance in regimen-curve estimation problems are unknown and merit further research.

\subsection{Near positively violation} In general, near positivity violation can negatively impact the performance of causal estimators. In particular, it can lead to an increased bias and variance in inverse probability weighted estimators and undersmoothing may exacerbate the issue.  One possible remedy is to consider truncation where individuals whose estimated weight functions are smaller than a given cutpoint are removed from the analyses \citep{crump2009dealing}. There are also other methods to trim the population in an interpretable way \citep{traskin2011defining, fogarty2016discrete}. Another possible way to improve the performance of our approach under near positivity violation is to fit a highly adaptive lasso by enforcing the positivity $\min (G^a_n, 1 - G^a_n) > \gamma$ and $\min (G^c_n, 1 - G^c_n) > \gamma$, for a positive constant $\gamma$. One can even make $\gamma$ data dependent (i.e., consider $\gamma_n$). This is an interesting direction from both methodological and practical point of view.

\subsection{Variable selection} In our proposed method, we assume that the set of variables included in regimen response-curve function and the decision rule (i.e., $\bm{V}$ and $\bm{S}$, respectively) are a priori specified. Particularly, in health research, variable selection in decision making is important as it may reduce the treatment costs and burden on patients. Moreover,  the quality of
the constructed rules can be severely hampered by  inclusion of many  extraneous variables \citep{jones2022valid}. Therefore, it will be interesting to systematically investigate how to perform variable selection in regimen response-curve function and the decision rule and to provide theoretical guarantees.

\begin{appendix}
\section{Theoretical undersmoothing conditions} \label{app:theoreticalconditions}

\begin{lemma}[Theoretical undersmoothing conditions when $h_n \rightarrow 0$]\label{lem:ucondition-movh}
Let $\Psi \equiv \Psi(\theta,\bm{V})$, $G^a \equiv G(P)(A \mid \bm{W})$ and $G^c \equiv G(P)(\Delta^c \mid \bm{W},A)$ denote $E(Y^{d^\theta}\mid V)$, the treatment and censoring mechanism under an arbitrary distribution $P \in \M$.
Let $\Psi_{n,\lambda_n^\Psi}$, $G^a_{n,\lambda_n^a}$ and $G^c_{n,\lambda_n^c}$ be the highly adaptive lasso estimators of $\Psi$, $G^a$ and $G^c$ using
$L_1$-norm bound $\lambda_n^\Psi$, $\lambda_n^a$ and $\lambda_n^c$, respectively. Choosing $\lambda_n^\Psi$, $\lambda_n^a$ and $\lambda_n^c$ such that
\begin{align}
        \min_{(s,j) \in \mathcal{J}_n^\Psi} {\bigg \Vert} P_n \frac{d}{d
    \Psi_{n,\lambda_n^\Psi}}
    L( \Psi_{n,\lambda_n^\Psi}) (\phi_{s,j}) {\bigg \Vert} &= o_p\left(n^{-1/2}h^{r/2}\right), \label{eq:basis3}\\
    \min_{(s,j) \in \mathcal{J}_n^a} {\bigg \Vert} P_n \frac{d}{d\logit
    G_{n,\lambda_n^a}^a}
    L(\logit G_{n,\lambda_n^a}^a) (\phi_{s,j}) {\bigg \Vert}&= o_p\left(n^{-1/2}h^{r/2}\right),  \label{eq:basis1} \\
    \min_{(s,j) \in \mathcal{J}_n^c} {\bigg \Vert} P_n \frac{d}{d\logit
    G_{n,\lambda_n^c}^c}
    L(\logit G_{n,\lambda_n^c}^c) (\phi_{s,j}) {\bigg \Vert} &= o_p\left(n^{-1/2}h^{r/2}\right), \label{eq:basis2} 
\end{align}
where, in (\ref{eq:basis3}), $L(\cdot)$ is the loss function  (\ref{eq:lossf}), and, in (\ref{eq:basis1}), $L(\cdot)$  is log-likelihood loss. Also,   $\mathcal{J}_n^.$ is a set of indices
for the basis functions with nonzero coefficients in the corresponding model. Let $D^\Psi(f_h^\Psi,\Psi_{n})
= f_h^\Psi \cdot \{I(Y>t) - \Psi_{n}\}$,  $D^a(f_h^a,G_{n}^a)
= f_h^a \cdot (A - G_{n}^a)$, and $D^c(f^c,G_{n}^c)
= f_h^c \cdot (\Delta^c - G_{n}^c)$. The functions $f_h^\Psi$, $f_h^a$ and $f_h^c$ are such that $\tilde f^\Psi = h^rf_h^\Psi$, $\tilde f^a=h^rf_h^a$  and $\tilde f^c=h^rf_h^c$ are c\`{a}dl\`{a}g with finite sectional
variation norm. Let $\tilde{f}_\phi^\Psi$, $\tilde{f}_\phi^a$ and $\tilde{f}_\phi^c$ be projections of $\tilde f^\Psi$, $\tilde f^a$  and $ \tilde f^c$ onto the linear
span of the basis functions $\phi_{s,j}$ in $L^2(P)$, where $\phi_{s,j}$
satisfies condition (\ref{eq:basis3}) and (\ref{eq:basis1}), respectively. 
Then, under the optimal bandwidth rate $h = n^{\frac{-1}{r+2J}}$ and $J>r$, we have $P_n D^\Psi({f}_h^\Psi,\Psi_{n}) =
o_p\left((nh^r)^{-1/2}\right)$,  $P_n D^a({f}_h^a,G_{n}^a) = o_p\left((nh^r)^{-1/2}\right)$ and $P_n D^c({f}_h^c,G_{n}^c) = o_p\left((nh^r)^{-1/2}\right)$.
\end{lemma}

\begin{lemma}[Theoretical undersmoothing conditions for a fixed bandwidth $h$]\label{lem:ucondition-fixedh}
Let $\Psi \equiv \Psi(\theta,\bm{V})$, $G^a \equiv G(P)(A \mid \bm{W})$ and $G^c \equiv G(P)(\Delta^c \mid \bm{W},A)$ denote $E(Y^{d^\theta}\mid V)$, the treatment and censoring mechanism under an arbitrary distribution $P \in \M$.
Let $\Psi_{n,\lambda_n^\Psi}$, $G^a_{n,\lambda_n^a}$ and $G^c_{n,\lambda_n^c}$ be the highly adaptive lasso estimators of $\Psi$, $G^a$ and $G^c$ using
$L_1$-norm bound $\lambda_n^\Psi$, $\lambda_n^a$ and $\lambda_n^c$, respectively. Choosing $\lambda_n^\Psi$, $\lambda_n^a$ and $\lambda_n^c$ such that
\begin{align}
        \min_{(s,j) \in \mathcal{J}_n^\Psi} {\bigg \Vert} P_n \frac{d}{d
    \Psi_{n,\lambda_n^\Psi}}
    L( \Psi_{n,\lambda_n^\Psi}) (\phi_{s,j}) {\bigg \Vert} &= o_p\left(n^{-1/2}\right), \label{eq:basis3-fixedh}\\
    \min_{(s,j) \in \mathcal{J}_n^a} {\bigg \Vert} P_n \frac{d}{d\logit
    G_{n,\lambda_n^a}^a}
    L(\logit G_{n,\lambda_n^a}^a) (\phi_{s,j}) {\bigg \Vert}&= o_p\left(n^{-1/2}\right),  \label{eq:basis1-fixedh} \\
    \min_{(s,j) \in \mathcal{J}_n^c} {\bigg \Vert} P_n \frac{d}{d\logit
    G_{n,\lambda_n^c}^c}
    L(\logit G_{n,\lambda_n^c}^c) (\phi_{s,j}) {\bigg \Vert} &= o_p\left(n^{-1/2}\right), \label{eq:basis2-fixedh} 
\end{align}
where, in (\ref{eq:basis3-fixedh}), $L(\cdot)$ is the loss function  (\ref{eq:lossf}), and, in (\ref{eq:basis1-fixedh}) and (\ref{eq:basis2-fixedh}), $L(\cdot)$  is log-likelihood loss. Also,   $\mathcal{J}_n^.$ is a set of indices
for the basis functions with nonzero coefficients in the corresponding model. Let $D^\Psi(f^\Psi,\Psi_{n})
= f^\Psi \cdot \{I(Y>t) - \Psi_{n}\}$,  $D^a(f^a,G_{n}^a)
= f^a \cdot (A - G_{n}^a)$, and $D^c(f^c,G_{n}^c)
= f^c \cdot (\Delta^c - G_{n}^c)$. The functions $f^\Psi$, $f^a$ and $f^c$ are c\`{a}dl\`{a}g with finite sectional
variation norm. Let $\tilde{f}_\phi^\Psi$, $\tilde{f}_\phi^a$ and $\tilde{f}_\phi^c$ be projections of $\tilde f^\Psi$, $\tilde f^a$  and $ \tilde f^c$ onto the linear
span of the basis functions $\phi_{s,j}$ in $L^2(P)$, where $\phi_{s,j}$
satisfies condition (\ref{eq:basis3-fixedh}) and (\ref{eq:basis1-fixedh}), respectively. 
Then,  we have $P_n D^\Psi({f}^\Psi,\Psi_{n}) =
o_p\left(n^{-1/2}\right)$,  $P_n D^a({f}^a,G_{n}^a) = o_p\left(n^{-1/2}\right)$ and $P_n D^c({f}^c,G_{n}^c) = o_p\left(n^{-1/2}\right)$.
\end{lemma}

\begin{lemma}[Theoretical undersmoothing conditions for uniform convergence with fixed bandwidth $h$]\label{lem:ucondition-uconvergence}
Let $\Psi \equiv \Psi(\theta,\bm{V})$, $G^a \equiv G(P)(A \mid \bm{W})$ and $G^c \equiv G(P)(\Delta^c \mid \bm{W},A)$ denote $E(Y^{d^\theta}\mid V)$, the treatment and censoring mechanism under an arbitrary distribution $P \in \M$.
Let $\Psi_{n,\lambda_n^\Psi}$, $G^a_{n,\lambda_n^a}$ and $G^c_{n,\lambda_n^c}$ be the highly adaptive lasso estimators of $\Psi$, $G^a$ and $G^c$ using
$L_1$-norm bound $\lambda_n^m$, $\lambda_n^a$ and $\lambda_n^c$, respectively. Choosing $\lambda_n^\Psi$, $\lambda_n^a$ and $\lambda_n^c$ such that
\begin{align}
        \min_{(s,j) \in \mathcal{J}_n^\Psi} \sup_{v \in {\mathcal{V}}} {\bigg \Vert} P_n \frac{d}{d
    \Psi_{n,\lambda_n^\Psi}}
    L( \Psi_{n,\lambda_n^\Psi}) (\phi_{s,j}) {\bigg \Vert} &= o_p(n^{-1/2}), \label{eq:basis3u}\\
    \min_{(s,j) \in \mathcal{J}_n^a} \sup_{v \in {\mathcal{V}}}{\bigg \Vert} P_n \frac{d}{d\logit
    G_{n,\lambda_n^a}^a}
    L(\logit G_{n,\lambda_n^a}^a) (\phi_{s,j}) {\bigg \Vert}&= o_p(n^{-1/2}),  \label{eq:basis1u} \\
    \min_{(s,j) \in \mathcal{J}_n^c} \sup_{v \in {\mathcal{V}}}{\bigg \Vert} P_n \frac{d}{d\logit
    G_{n,\lambda_n^c}^c}
    L(\logit G_{n,\lambda_n^c}^c) (\phi_{s,j}) {\bigg \Vert} &= o_p(n^{-1/2}), \label{eq:basis2u} 
\end{align}
where, in (\ref{eq:basis3u}), $L(\cdot)$ is the loss function  (\ref{eq:lossf}), and, in (\ref{eq:basis1u}) and (\ref{eq:basis2u}), $L(\cdot)$  is log-likelihood loss. Also,   $\mathcal{J}_n^.$ is a set of indices
for the basis functions with nonzero coefficients in the corresponding model. Let $D^\Psi(f^\Psi,\Psi_{n})
= f^\Psi \cdot \{I(Y>t) - \Psi_{n}\}$,  $D^a(f^a,G_{n}^a)
= f^a \cdot (A - G_{n}^a)$, and $D^c(f^c,G_{n}^c)
= f^c \cdot (\Delta^c - G_{n}^c)$. Here, $f^\Psi$, $f^a$  and $f^c$ are c\`{a}dl\`{a}g with finite sectional
variation norm, and we let $\tilde{f}^\Psi$, $\tilde{f}^a$ and $\tilde{f}^c$ be projections of $f^\Psi$, $f^a$  and $f^c$ onto the linear
span of the basis functions $\phi_{s,j}$ in $L^2(P)$, where $\phi_{s,j}$
satisfies condition (\ref{eq:basis3u}) and (\ref{eq:basis1u}), respectively. Assuming $\lVert f^\Psi - \tilde{f}^\Psi
\rVert_{2,P_0} = O_p(n^{-1/4})$, $\lVert f^a - \tilde{f}^a
\rVert_{2,P_0} = O_p(n^{-1/4})$ and $\lVert f^c - \tilde{f}^c
\rVert_{2,P_0} = O_p(n^{-1/4})$, it follows that $P_n D^\Psi(\tilde{f}^\Psi,\Psi_{n}) =
o_p(n^{-1/2})$, $P_n D^a(\tilde{f}^a,G_{n}^a) =
o_p(n^{-1/2})$, $P_n D^c(\tilde{f}^c,G_{n}^c) =
o_p(n^{-1/2})$ $P_n D^\Psi({f}^\Psi,\Psi_{n}) =
o_p(n^{-1/2})$,  $P_n D^a({f}^a,G_{n}^a) = o_p(n^{-1/2})$ and $P_n D^c({f}^c,G_{n}^c) = o_p(n^{-1/2})$.
\end{lemma}

\section{Sectional variation norm} \label{sec:secnorm}
Suppose that the function of interest falls into a c\`{a}dl\`{a}g class of functions on a p-dimensional cube $[0,\tau] \subset \mathcal{R}^p$ with finite sectional variation norm where $\tau$ is the upper bound of all supports which is assumed to be finite. The p-dimensional cube $[0,\tau]$ can be represented as a union of lower dimensional cubes (i.e., $l$-dimensional with $l\leq p$) plus the origin. That is  $[0,\tau] = \{\cup_s(0,\tau_s]\} \cup \{0\}$ where $\cup_s$ is over all subsets $s$ of $\{1,2,\cdots,p\}$. 

Denote $\D[0,\tau]$ as the Banach space of $p$-variate real-valued c\`{a}dl\`{a}g
functions on a cube $[0,\tau] \in \R^p$. For each function $f \in \D[0,\tau]$, we define the $s-$th section $f$ as $f_s(u) = f\{u_1 I(1 \in s),\cdots, u_d \in I(p \in s)\}$.  By definition, $f_s(u)$  varies along the variables in $u_s$ according to $f$ while setting other variables to zero. We define the sectional
variation norm of a given $f$  as
\begin{equation*}
  \lVert f \rVert^{*}_\zeta = \lvert f(0) \rvert +\sum_{s
  \subset\{1,\ldots,p\}} \int_{0_s}^{\tau_s} \lvert df_s(u) \rvert,
\end{equation*}
where the sum is over all subsets of the coordinates $\{0,1,\ldots,p\}$. The term $ \int_{0_s}^{\tau_s} \lvert df_s(u) \rvert$ denotes the $s-$specific variation norm. To clarify the definition, let's consider a simple case of $p=1$. In this case, the variation norm of function $f$ over $[0,\tau]$ is defined as
\[
\lVert f \rVert^{*}_\zeta = \sup_{Q \in \mathcal{Q}} \sum_{i=1}^{n_q-1}  \lvert f(c_{i+1}) - f(c_{i}) \rvert
\]
where $c_0=0$, $Q$ is a partition of $[0,\tau]$ such as $(0,c_1],(c_1,c_2],\cdots,(c_{n-1},c_{n_q}=\tau]$ and $\mathcal{Q}$ is the class of all possible partitions. For differentiable functions $f$, the sectional variation norm is $\lVert f \rVert^{*}_\zeta = \int_{0}^{\tau} \lvert f'(u)\rvert du $. 

For $p=2$, the variation norm of function $f$ over $[0,\tau]\subset \mathcal{R}^2$ is defined as
\[
\lVert f \rVert^{*}_\zeta = \sup_{Q \in \mathcal{Q}} \sum_{i=0}^{n_{q_1}-1}\sum_{j=0}^{n_{q_2}-1}  \lvert f(u_{i+1},v_{j+1}) - f(u_{i},v_{j+1})- f(u_{i+1},v_{j})+f(u_{i},v_{j})\rvert
\]
where $u_0=v_0=0$, $Q$ is a partition of $[0,\tau]$ such as $(0,u_1],(u_1,u_2],\cdots,(u_{n_{q_2}-1},u_{n_{q_1}}=\tau]$, $(0,v_1],(v_1,v_2],\cdots,(v_{n_{q_2}-1},v_{n_{q_2}}=\tau]$ and $\mathcal{Q}$ is the class of all possible partitions.


In practice, we choose the support points $u_{s,i}$ to be the observed values of $\bm{W}$ denoted as $\tilde w_{s,1},\cdots,\tilde w_{s,n}$. Hence the indicator basis functions will be  $\phi_{s,i} (w_s) = I(w_s \geq \tilde w_{s,i})$. For simplicity lets consider $p=1$, then, in a hypothetical example, the design matrix based on indicator basis functions would be

\begin{table}[h]
\centering
\begin{tabular}{lrrrrr}
$W$ & $I(W\geq 1.45)$ & $I(W\geq 0.84)$ & $I(W\geq 1.0)$ & $I(W\geq 0.16)$ & $\cdots$  \\
\hline
1.45 & 1 & 1 & 0 &1 &$\cdots$\\
0.84 & 0 & 1 & 0&1&$\cdots$ \\
1.0 & 1 & 1 & 1  &1&$\cdots$\\
0.16 & 0 & 0 & 0& 1&$\cdots$\\
$\cdots$ & $\cdots$ & $\cdots$ & $\cdots$ &$\cdots$\\
\end{tabular}
\vspace{0em}
\end{table}



\end{appendix}


\begin{funding}
The first author was supported by  National Institute on Drug Abuse under award number R01DA048764, and National Institute of Neurological Disorders and Stroke under award numbers R33NS120240 and R61NS12024. The third author was supported in part by 
NIH Grant R01AI074345.
\end{funding}

\begin{supplement}
\stitle{Proofs of the theoretical results \ref{sec:proofs1}-\ref{sec:ratethetaconvh}}
\sdescription{Supplements \ref{sec:proofs1}-\ref{sec:ratethetaconvh} contain proofs of our theoretical results. }
\end{supplement}
\begin{supplement}
\stitle{Additional figures}
\sdescription{ Supplement \ref{sec:addfig} includes additional figures for the simulation studies.}
\end{supplement}




\bibliographystyle{imsart-nameyear}
\bibliography{chalipwbib}

\pagebreak

\setcounter{section}{0}
\doublespacing
\renewcommand*{\thesection}{S\arabic{section}}
\renewcommand*{\thesubsection}{S\arabic{section}.\arabic{subsection}}
\renewcommand{\theequation}{S\arabic{equation}}
\renewcommand{\thefigure}{S\arabic{figure}}
\renewcommand{\thetable}{S\arabic{table}}
\renewcommand{\thetheorem}{S\arabic{theorem}}
\renewcommand{\thelemma}{S\arabic{lemma}}
\renewcommand{\bibnumfmt}[1]{[S#1]}
\renewcommand{\citenumfont}[1]{S#1}
\setcounter{page}{1}


\begin{center}
  \LARGE{
    \textbf{Supplement to ``Nonparametric estimation of a covariate-adjusted counterfactual treatment regimen response curve''}
  }
\end{center}

\section{Proof of Lemmas \ref{lem:ucondition-movh}, \ref{lem:ucondition-fixedh} and \ref{lem:ucondition-uconvergence} } \label{sec:proofs1}
\begin{proof}[Proof of Lemma~\ref{lem:ucondition-movh}]
We first proof the result for the treatment indicator and $G^a$. Let $\tilde f^a = h^rf_h^a $ and  $D^a\{\tilde f^a,G^a_{n,\lambda_n}\} = \tilde f^a \cdot (A - G^a_{n,\lambda_n})$, where $\tilde f^a$ is
a c\`{a}dl\`{a}g function with a finite sectional variation norm, and let
$\tilde{f}_\phi^a$ be an approximation of $\tilde f^a$ using the basis functions $\phi$ that satisfy condition (\ref{eq:basis1}).

\begin{align*}
    P_n D^a(f_h^a, G^a_{n,\lambda_n}) &= h^{-r}P_n D^a(\tilde f^a, G^a_{n,\lambda_n}) \\
                     & = h^{-r}P_n D^a(\tilde f^a - \tilde{f}_\phi^a, G^a_{n,\lambda_n}) +h^{-r}P_n D^a( \tilde{f}_\phi^a, G^a_{n,\lambda_n}) \\
                     & = h^{-r}(P_n-P_0) D^a(\tilde f^a - \tilde{f}_\phi^a, G^a_{n,\lambda_n}) +h^{-r}P_0 D^a(\tilde f^a - \tilde{f}_\phi^a, G^a_{n,\lambda_n})+h^{-r}P_n D^a( \tilde{f}_\phi^a, G^a_{n,\lambda_n}) \\
                     & = h^{-r} O_p(n^{-1/2} n^{-1/6}) + h^{-r} O_p(n^{-1/3}n^{-1/3})+h^{-r}P_n D^a( \tilde{f}_\phi^a, G^a_{n,\lambda_n}).
\end{align*}
The last equality follows from $\|\tilde f^a - \tilde{f}_\phi^a\|_2 = O_p(n^{-1/3})$ and Lemma \ref{lem:vw} as
\begin{align*}
(P_n-P_0) D^a(\tilde f^a - \tilde{f}_\phi^a, G^a_{n,\lambda_n}) &\leq   \sup_{ \|\tilde f^a - \tilde{f}_\phi^a\|_2 \leq n^{-1/3}} | (P_n-P_0)D^a(\tilde f^a - \tilde{f}_\phi^a, G^a_{n,\lambda_n}) | \nonumber \\
                                 & =O_p\{  n^{-1/2} \mathcal{E}(n^{-1/3},L^2(P))\} \leq O_p(n^{-1/2}n^{-1/6}).
 \end{align*}
Moreover, by the convergence rate of the
highly adaptive lasso estimate (i.e., $\lVert G^a_{0} - G^a_{n,\lambda_n}
\rVert_{2,P_0} = O_p(n^{-1/3})$), 
\[
P_0 D^a(\tilde f^a - \tilde{f}_\phi^a, G^a_{n,\lambda_n}) = P_0 (\tilde f^a - \tilde{f}_\phi^a) (G_0^a-G^a_{n,\lambda_n}) =O_p(n^{-1/3}) O_p(n^{-1/3}).  
\]
Therefore, the proof is complete if we show $$h^{-r} O_p(n^{-1/2} n^{-1/6}) + h^{-r} O_p(n^{-1/3}n^{-1/3})+h^{-r}P_n D^a( \tilde{f}_\phi^a, G^a_{n,\lambda_n}) = o_p\left((nh^r)^{-1/2}\right).$$
Considering the optimal bandwidth rate of $h^r = n^{\frac{-r}{r+2J}}$, the first term satisfies the desired rate when $h^{-r} n^{-1/2} n^{-1/6} < n^{-1/2} h^{-r/2}$ which implies
$\frac{-1}{6} < \frac{-r}{2r+4J}$. The inequality is satisfied for $J>r$. We can similarly show that $h^{-r} O_p(n^{-2/3}) = o_p\left((nh^r)^{-1/2}\right)$ for $J>r$. Hence, for $J>r$,
\[
P_n D^a(f_h^a, G^a_{n,\lambda_n}) = o_p\left((nh^r)^{-1/2}\right) + h^{-r}P_n D^a( \tilde{f}_\phi^a, G^a_{n,\lambda_n}).
\]
In the following we show that $h^{-r}P_n D^a( \tilde{f}_\phi^a, G^a_{n,\lambda_n})$ is also of the order $o_p\left((nh^r)^{-1/2}\right)$.

For simplicity of notation, let $G^{a\dagger} = \logit G^a$. Define a set of score
functions generated by a path $\{1+\epsilon g(s,j)\} \beta_{n,s,j}$ for
a uniformly bounded vector $g$ as
\begin{equation*}
  S_g(G^a_{n,\lambda_n}) = \frac{d}{dG^{a\dagger}_{n,\lambda_n}} L(
  G^{a\dagger}_{n,\lambda_n}) \left\{ \sum_{(s,j)} g(s,j) \beta_{n,s,j} \phi_{s,j}
  \right\}.
\end{equation*}
When $L(\cdot)$ is the log-likelihood loss function,
$L(G) =  A \log \left(\frac{G}{1-G}\right) +\log(1-G)$.
Thus, $L(G^{a\dagger}) =  A \log G^{a\dagger} +\log(G^{a\dagger} -1)$ and
$S_g(G^a) =   (A - G^a_{n,\lambda_n}) \left\{\sum_{(s,j)} g(s,j)
\beta_{n,s,j} \phi_{s,j} \right\}$.
Let $r(g,G^a_{n,\lambda_n}) = \sum_{(s,j)} g(s,j) \lvert \beta_{n,s,j} \rvert$.
For small enough $\epsilon$,
\begin{align*}
 \sum_{(s,j)} \lvert \{1+\epsilon g(s,j)\} \beta_{n,s,j} \rvert &= \sum_{(s,j)}
   \{1 + \epsilon g(s,j)\} \lvert \beta_{n,s,j} \rvert \\
   &=\sum_{(s,j)} \lvert \beta_{n,s,j} \rvert + \epsilon r(g,G_{n,\lambda_n}).
\end{align*}
Hence,  for any $g$ satisfying $r(g,G_{n,\lambda_n})=0$, we have $P_n
S_g(G^a_{n,\lambda_n}) = 0$. 

Because $D^a\{\tilde{f}_\phi^a, G^a_{n,\lambda_n}\} \in
\{S_g(G^a_{n,\lambda_n}): \lVert g \rVert_{\infty} < \infty \}$, there
exists $g^{\star}$ such that $D^a(\tilde{f}_\phi^a, G^a_{n,\lambda_n}) = S_{g^{\star}}
(G^a_{n,\lambda_n})$. However, for this particular choice of $g^{\star}$,
$r(g,G^a_{n,\lambda_n})$ may not be zero. Now, define $g$ such that $ g(s,j)
= g^{\star}(s,j)$ for $(s,j) \neq (s^{\star}, j^{\star})$; $\tilde{g}
(s^{\star}, j^{\star})$ is defined such that
\begin{align}\label{eq:restr}
\sum_{(s,j) \neq (s^{\star},j^{\star})} g^{\star}(s,j) \lvert
 \beta_{n,s,j} \rvert +  g(s^{\star}, j^{\star})
 \lvert \beta_{n, s^{\star}, j^{\star}} \rvert = 0.
\end{align}
That is, $g$ matches $g^{\star}$ everywhere but for a single point
$(s^{\star}, j^{\star})$, where it is forced to take a value such that
$r(g,G_{n,\lambda_n})=0$. As a result, for such a choice of $g$, $P_n S_{g}
(G^a_{n,\lambda_n}) = 0$ by definition. Below, we show that $P_n
S_{g}(G^a_{n,\lambda_n}) - P_n D^a\{\tilde{f}_\phi^a, G^a_{n,\lambda_n}\} = o_p\left(n^{-1/2} h^{r/2}\right)$
which then implies that $h^{-r} P_n D^a\{\tilde{f}_\phi^a, G^a_{n,\lambda_n}\}
= o_p\left((nh^r)^{-1/2}\right)$. We note that the choice of $(s^{\star}, j^{\star})$ is
inconsequential.
\begin{align*}
   P_n S_{g}(G^a_{n,\lambda_n}) - P_n D^a\{\tilde{f}_\phi^a, G^a_{n,\lambda_n}\} &= P_n
   S_{g}(G^a_{n,\lambda_n}) - P_n S_{g^*}(G^a_{n,\lambda_n}) \\ &=P_n \left\{
   \frac{d}{dG^{a\dagger}_{n,\lambda_n}} L( G^{a\dagger}_{n,\lambda_n})
   \left[\sum_{(s,j)} \left\{g(s,j) - g^{\star}(s,j)\right\} \beta_{n,s,j}
   \phi_{s,j} \right]  \right\} \\ &= P_n
   \left[\frac{d}{dG^{a\dagger}_{n,\lambda_n}} L( G^{a\dagger}_{n,\lambda_n}) \left\{
   g(s^{\star},j^{\star}) - g^{\star}(s^{\star},j^{\star})\right\}
   \beta_{n,s^{\star},j^{\star}} \phi_{s^{\star},j^{\star}} \right] \\ & = P_n
   \left[\frac{d}{dG^{a\dagger}_{n,\lambda_n}} L( G^{a\dagger}_{n,\lambda_n})
   \kappa(s^{\star},j^{\star}) \phi_{s^{\star},j^{\star}} \right]\\ &=
   \kappa(s^{\star},j^{\star}) P_n \left[\frac{d}{dG^{a\dagger}_{n,\lambda_n}} L(
   G^{a\dagger}_{n,\lambda_n}) \phi_{s^{\star},j^{\star}} \right],
\end{align*}
where the third equality follows from equation (\ref{eq:restr}) above with
\begin{equation*}
 \kappa(s^{\star},j^{\star}) = -\frac{\sum_{(s,j) \neq (s^{\star},j^{\star})}
 g^{\star}(s,j) \lvert \beta_{n,s,j} \rvert }{\lvert
 \beta_{n,s^{\star},j^{\star}} \rvert} \beta_{n,s^{\star},j^{\star}} -
 g^{\star}(s^{\star},j^{\star}) \beta_{n,s^{\star},j^{\star}}.
\end{equation*}
Moreover,
\begin{equation*}
 \left\lvert \kappa(s^{\star},j^{\star})  \right\rvert \leq \sum_{(s,j)}
 \lvert g^{\star}(s,j)  \beta_{n,s,j} \rvert.
\end{equation*}
Assuming $\tilde f_\phi^a$ has finite sectional variation norm, the $L_1$ norm of the
coefficients approximating $\tilde f_\phi^a$ will be finite which implies that $\sum_{(s,j)}
 \lvert g^{\star}(s,j)  \beta_{n,s,j} \rvert$ is finite, and thus
$\lvert \kappa(s^{\star},j^{\star}) \rvert = O_p(1)$. Then,
\begin{align*}
 P_n S_{g}(G^a_{n,\lambda_n}) - P_n D^a(\tilde{f}_\phi^a, G^a_{n,\lambda_n}) &=
 O_p \left(P_n \left[\frac{d}{dG^{a\dagger}_{n,\lambda_n}} L(
 G^{a\dagger}_{n,\lambda_n})\phi_{s^{\star},j^{\star}} \right] \right) \\
 & = o_p\left(n^{-1/2}h^{r/2}\right),
\end{align*}
where the last equality follows from the assumption that $\min_{(s,j) \in
\mathcal{J}_n } \lVert P_n \frac{d}{dG^{a\dagger}_{n,\lambda_n}}
L(G^{a\dagger}_{n,\lambda_n}) (\phi_{s,j}) \rVert = o_p\left(n^{-1/2}h^{r/2}\right)$ for $L(\cdot)$
being log-likelihood loss. As $P_n S_{{g}}(G^a_{n,\lambda_n}) = 0$, it
follows that $P_n D^a(\tilde{f}_\phi^a,G^a_{n,\lambda_n}) = o_p\left(n^{-1/2}h^{r/2}\right)$ and thus $h^{-r}P_n D^a(\tilde{f}_\phi^a,G^a_{n,\lambda_n}) = o_p\left((nh^r)^{-1/2}\right)$. \\
We just showed that $P_n D^a(f_h^a,G^a_{n,\lambda_n}) = o_p\left((nh^r)^{-1/2}\right)$. Similarly we can show that $P_n D^c(f_h^c,G^c_{n,\lambda_n}) = o_p\left((nh^r)^{-1/2}\right)$ and $P_n D^\Psi(f_h^\Psi,\Psi_{n}) = o_p\left((nh^r)^{-1/2}\right)$ which completes the proof. 
\end{proof}
\begin{proof}[Proof of Lemma~\ref{lem:ucondition-fixedh}]
    The prove follows from the proof of Lemma \ref{lem:ucondition-movh} with the bandwidth $h$ replaced by a constant value (e.g., one).
\end{proof}
    \begin{proof}[Proof of Lemma~\ref{lem:ucondition-uconvergence}]
We first proof the result for the treatment indicator and $G^a$. 

\begin{align*}
   \sup_{v \in \mathcal{V}} \left| P_n D^a( f^a, G^a_{n,\lambda_n}) \right|
                      &\leq \sup_{v \in \mathcal{V}}\left|P_n D^a( f^a - \tilde{f}_\phi^a, G^a_{n,\lambda_n})\right| +\sup_{v \in \mathcal{V}}\left|P_n D^a( \tilde{f}_\phi^a, G^a_{n,\lambda_n})\right| \\
                     & \leq \sup_{v \in \mathcal{V}} \left|(P_n-P_0) D^a( f^a - \tilde{f}_\phi^a, G^a_{n,\lambda_n})\right| +\sup_{v \in \mathcal{V}} \left|P_0 D^a( f^a - \tilde{f}_\phi^a, G^a_{n,\lambda_n})\right|\\
                     &+\sup_{v \in \mathcal{V}} \left|P_n D^a( \tilde{f}_\phi^a, G^a_{n,\lambda_n})\right| \\
                     & =  O_p(n^{-1/2} n^{-1/6}) +  O_p(n^{-1/3}n^{-1/3})+\sup_{v \in \mathcal{V}} \left|P_n D^a( \tilde{f}_\phi^a, G^a_{n,\lambda_n})\right|.
\end{align*}
The last equality follows from $\sup_{v \in \mathcal{V}} \|\tilde f^a - \tilde{f}_\phi^a\|_2 = O_p(n^{-1/3})$ and Lemma \ref{lem:vw} as
\begin{align*}
\sup_{v \in \mathcal{V}} (P_n-P_0) D^a(\tilde f^a - \tilde{f}_\phi^a, G^a_{n,\lambda_n}) &\leq   \sup_{ \substack{ v \in \mathcal{V} \\ \|\tilde f^a - \tilde{f}_\phi^a\|_2 \leq n^{-1/3}}} | (P_n-P_0)D^a(\tilde f^a - \tilde{f}_\phi^a, G^a_{n,\lambda_n}) | \nonumber \\
                                 & =O_p\{  n^{-1/2} \mathcal{E}(n^{-1/3},L^2(P))\} \leq O_p(n^{-1/2}n^{-1/6}).
 \end{align*}
Moreover, by the convergence rate of the
highly adaptive lasso estimate (i.e., $\sup_{v \in \mathcal{V}}\lVert G^a_{0} - G^a_{n,\lambda_n}
\rVert_{2,P_0} = O_p(n^{-1/3})$), 
\[
\sup_{v \in \mathcal{V}} \left| P_0 D^a(\tilde f^a - \tilde{f}_\phi^a, G^a_{n,\lambda_n}) \right| = \sup_{v \in \mathcal{V}} \left| P_0 (\tilde f^a - \tilde{f}_\phi^a) (G_0^a-G^a_{n,\lambda_n})\right| =O_p(n^{-1/3}) O_p(n^{-1/3}).  
\]
Therefore, the proof is complete if we show $$ O_p(n^{-1/2} n^{-1/6}) +  O_p(n^{-1/3}n^{-1/3})+\sup_{v \in \mathcal{V}} \left|P_n D^a( \tilde{f}_\phi^a, G^a_{n,\lambda_n})\right| = o_p(n^{-1/2}).$$

\[
\sup_{v \in \mathcal{V}} \left| P_n D^a(f^a, G^a_{n,\lambda_n}) \right| = o_p(n^{-1/2}) + \sup_{v \in \mathcal{V}} \left|P_n D^a( \tilde{f}_\phi^a, G^a_{n,\lambda_n})\right|.
\]
In the following we show that $\sup_{v \in \mathcal{V}} \left|P_n D^a( \tilde{f}_\phi^a, G^a_{n,\lambda_n})\right|$ is also of the order $o_p(n^{-1/2})$.

For simplicity of notation, let $G^{a\dagger} = \logit G^a$. Define a set of score
functions generated by a path $\{1+\epsilon g(s,j)\} \beta_{n,s,j}$ for
a uniformly bounded vector $g$ as
\begin{equation*}
  S_g(G^a_{n,\lambda_n}) = \frac{d}{dG^{a\dagger}_{n,\lambda_n}} L(
  G^{a\dagger}_{n,\lambda_n}) \left\{ \sum_{(s,j)} g(s,j) \beta_{n,s,j} \phi_{s,j}
  \right\}.
\end{equation*}
When $L(\cdot)$ is the log-likelihood loss function,
$L(G) =  A \log \left(\frac{G}{1-G}\right) +\log(1-G)$.
Thus, $L(G^{a\dagger}) =  A \log G^{a\dagger} +\log(G^{a\dagger} -1)$ and
$S_g(G^a) =   (A - G^a_{n,\lambda_n}) \left\{\sum_{(s,j)} g(s,j)
\beta_{n,s,j} \phi_{s,j} \right\}$.
Let $r(g,G^a_{n,\lambda_n}) = \sum_{(s,j)} g(s,j) \lvert \beta_{n,s,j} \rvert$.
For small enough $\epsilon$,
\begin{align*}
 \sum_{(s,j)} \lvert \{1+\epsilon g(s,j)\} \beta_{n,s,j} \rvert &= \sum_{(s,j)}
   \{1 + \epsilon g(s,j)\} \lvert \beta_{n,s,j} \rvert \\
   &=\sum_{(s,j)} \lvert \beta_{n,s,j} \rvert + \epsilon r(g,G_{n,\lambda_n}).
\end{align*}
Hence,  for any $g$ satisfying $r(g,G_{n,\lambda_n})=0$, we have $P_n
S_g(G^a_{n,\lambda_n}) = 0$. 

Because $D^a\{\tilde{f}_\phi^a, G^a_{n,\lambda_n}\} \in
\{S_g(G^a_{n,\lambda_n}): \lVert g \rVert_{\infty} < \infty \}$, there
exists $g^{\star}$ such that $D^a(\tilde{f}_\phi^a, G^a_{n,\lambda_n}) = S_{g^{\star}}
(G^a_{n,\lambda_n})$. However, for this particular choice of $g^{\star}$,
$r(g,G^a_{n,\lambda_n})$ may not be zero. Now, define $g$ such that $ g(s,j)
= g^{\star}(s,j)$ for $(s,j) \neq (s^{\star}, j^{\star})$; $\tilde{g}
(s^{\star}, j^{\star})$ is defined such that
\begin{align}\label{eq:restr}
\sum_{(s,j) \neq (s^{\star},j^{\star})} g^{\star}(s,j) \lvert
 \beta_{n,s,j} \rvert +  g(s^{\star}, j^{\star})
 \lvert \beta_{n, s^{\star}, j^{\star}} \rvert = 0.
\end{align}
That is, $g$ matches $g^{\star}$ everywhere but for a single point
$(s^{\star}, j^{\star})$, where it is forced to take a value such that
$r(g,G_{n,\lambda_n})=0$. As a result, for such a choice of $g$, $P_n S_{g}
(G^a_{n,\lambda_n}) = 0$ by definition. Below, we show that $P_n
S_{g}(G^a_{n,\lambda_n}) - P_n D^a\{\tilde{f}_\phi^a, G^a_{n,\lambda_n}\} = o_p(n^{-1/2} )$
which then implies that $ P_n D^a\{\tilde{f}_\phi^a, G^a_{n,\lambda_n}\}
= o_p(n^{-1/2})$. We note that the choice of $(s^{\star}, j^{\star})$ is
inconsequential.
\begin{align*}
   P_n S_{g}(G^a_{n,\lambda_n}) - P_n D^a\{\tilde{f}_\phi^a, G^a_{n,\lambda_n}\} &= P_n
   S_{g}(G^a_{n,\lambda_n}) - P_n S_{g^*}(G^a_{n,\lambda_n}) \\ &=P_n \left\{
   \frac{d}{dG^{a\dagger}_{n,\lambda_n}} L( G^{a\dagger}_{n,\lambda_n})
   \left[\sum_{(s,j)} \left\{g(s,j) - g^{\star}(s,j)\right\} \beta_{n,s,j}
   \phi_{s,j} \right]  \right\} \\ &= P_n
   \left[\frac{d}{dG^{a\dagger}_{n,\lambda_n}} L( G^{a\dagger}_{n,\lambda_n}) \left\{
   g(s^{\star},j^{\star}) - g^{\star}(s^{\star},j^{\star})\right\}
   \beta_{n,s^{\star},j^{\star}} \phi_{s^{\star},j^{\star}} \right] \\ & = P_n
   \left[\frac{d}{dG^{a\dagger}_{n,\lambda_n}} L( G^{a\dagger}_{n,\lambda_n})
   \kappa(s^{\star},j^{\star}) \phi_{s^{\star},j^{\star}} \right]\\ &=
   \kappa(s^{\star},j^{\star}) P_n \left[\frac{d}{dG^{a\dagger}_{n,\lambda_n}} L(
   G^{a\dagger}_{n,\lambda_n}) \phi_{s^{\star},j^{\star}} \right],
\end{align*}
where the third equality follows from equation (\ref{eq:restr}) above with
\begin{equation*}
 \kappa(s^{\star},j^{\star}) = -\frac{\sum_{(s,j) \neq (s^{\star},j^{\star})}
 g^{\star}(s,j) \lvert \beta_{n,s,j} \rvert }{\lvert
 \beta_{n,s^{\star},j^{\star}} \rvert} \beta_{n,s^{\star},j^{\star}} -
 g^{\star}(s^{\star},j^{\star}) \beta_{n,s^{\star},j^{\star}}.
\end{equation*}
Moreover,
\begin{equation*}
 \left\lvert \kappa(s^{\star},j^{\star})  \right\rvert \leq \sum_{(s,j)}
 \lvert g^{\star}(s,j)  \beta_{n,s,j} \rvert.
\end{equation*}
Assuming $\tilde f_\phi^a$ has finite sectional variation norm, the $L_1$ norm of the
coefficients approximating $\tilde f_\phi^a$ will be finite which implies that $\sum_{(s,j)}
 \lvert g^{\star}(s,j)  \beta_{n,s,j} \rvert$ is finite, and thus
$\lvert \kappa(s^{\star},j^{\star}) \rvert = O_p(1)$. Hence,

\begin{align*}
   \sup_{v \in \mathcal{V}} \left|  P_n S_{g}(G^a_{n,\lambda_n}) - P_n D^a(\tilde{f}_\phi^a, G^a_{n,\lambda_n}) \right| &=
 O_p \left( \sup_{v \in \mathcal{V}} \left| P_n \left[\frac{d}{dG^{a\dagger}_{n,\lambda_n}} L(
 G^{a\dagger}_{n,\lambda_n})\phi_{s^{\star},j^{\star}} \right] \right|\right) \\
 & = o_p(n^{-1/2}),
\end{align*}
where the last equality follows from the assumption that $\min_{(s,j) \in
\mathcal{J}_n } \sup_{v \in \mathcal{V}} \lVert P_n \frac{d}{dG^{a\dagger}_{n,\lambda_n}}
L(G^{a\dagger}_{n,\lambda_n}) (\phi_{s,j}) \rVert = o_p(n^{-1/2})$ for $L(\cdot)$
being log-likelihood loss. As $P_n S_{{g}}(G^a_{n,\lambda_n}) = 0$, it
follows that $P_n D^a(\tilde{f}_\phi^a,G^a_{n,\lambda_n}) = o_p(n^{-1/2})$. \\
We just showed that $P_n D^a(f^a,G^a_{n,\lambda_n}) = o_p(n^{-1/2})$. Similarly we can show that $P_n D^c(f^c,G^c_{n,\lambda_n}) = o_p(n^{-1/2})$ and $P_n D^\Psi(f^\Psi,\Psi_{n}) = o_p(n^{-1/2})$ which completes the proof. 
\end{proof}


\section{Other Lemmas}

\begin{lemma}[Lemma 3.4.2 in \cite{vaart1996weak}] \label{lem:vw}
For a class of functions $\mathcal{F}$ with envelop $\textbf{F} = \sup_{f \in \mathcal{F}} |f(x)|$ bounded from above by $M < \infty$, we have
\[
E_P \sup_{f \in \mathcal{F}, \|f\|_P \leq \delta} n^{1/2}(P_n-P_0)(f) \leq \mathcal{E}(\delta,\mathcal{F},L^2(P)) \left\{1+\frac{\mathcal{E}(\delta,\mathcal{F},L^2(P))}{\delta^2 n^{1/2}}M\right\}.
\]
This implies that $\sup_{f \in \mathcal{F}, \|f\|_P \leq \delta} n^{1/2}(P_n-P_0)(f)$ is bounded in probability by the right-hand side.
\end{lemma}

\begin{lemma}[Proposition 2 in \cite{bibaut2019fast}] \label{lem:laan}
Let $p \geq 2$ and $M>0$. Denote $\mathcal{F}_{p,M}$ the class of cadlag functions on $[0,1]^p$ with sectional variation norm smaller than $M$. Suppose that $P_0$ is such that, for all $1 \leq r \leq \infty$, for all real-valued function $f$ on $[0,1]^p$, $\| f\|_{P_0,b} = c(b)\| f\|_{\mu,b}$, for some $c(b)>0$, and where $\mu$ is the Lebesgue measure. Then, for any $0<\delta<1$, the bracketing entropy integral of $\mathcal{F}_{p,M}$ with respect to $\| .\|_{P_0,b}$ norm is bounded by
\[
\mathcal{E}_{[]}(\delta,\mathcal{F}_{p,M},\| .\|_{P_0,b}) \leq \{C(b,p) M \delta\}^{1/2} \mid \log(\delta/M) \mid ^{p-1},
\]
where $C(b,p)$ is a constant that depends only on $b$ and $p$.
\end{lemma}

\begin{lemma} \label{lem:hnneg}
Consider a sequence of processes $\{n^{1/2}(P_n-P_0) H_{n,h} (f) \}$ where
\[
H_{n,h} (f) = h^{r/2} K_{h,\bm{v}_0}  ( f_n- f_0). 
\]
Assuming the following:
\begin{itemize}
\item[(i)] The function  $f_0$ belongs to a cadlag class of functions with finite sectional variation norm.
\item[(ii)] The function $f_n$ is the highly adaptive lasso estimate of $f_0$.
\item[(iii)] The bandwidth $h_n$ satisfies $h^r \rightarrow 0$ and $nh^{r3/2} \rightarrow \infty$.
\end{itemize}
Then $n^{1/2}(P_n-P_0) H_{n,h} (f_n) = o_p(1)$.
\end{lemma}
\begin{proof}
Let $\tilde K_{h,\bm{v}_0}(v) = h^r K_{h,\bm{v}_0}(v)$. Define  $ h^{-r/2}\tilde H_{n,h}(f) =H_{n,h}(f)$. The function $\tilde H_n(f)$ belongs to a class of function $\mathcal{F}_n = \{\tilde H_{n,h}(f) : h\}$. We have
\begin{align*}
\| \tilde H_n(f_n) \|_{P_0} &\leq \| \tilde K_{h,\bm{v}_0}(v)  \|_{P_0} \|   ( f_n-f_0 )  \|_{P_0} \\
                                                  &= O_p(h^{r/2})O_p(n^{-1/3}) 
\end{align*}
where the last equality follows from the kernel and the highly adaptive lasso rate of convergence.  Using the result of Lemma \ref{lem:laan}, $n^{1/2} (P_n-P_0) \tilde H_n(f_n) = O_p(h^{r/4} n^{-1/6})$. Therefore, 
\begin{align*}
n^{1/2}(P_n-P_0)  H_n(f_n) = O_p(h^{-r/4}n^{-1/6}).
\end{align*}
Thus, to obtain the desired rate we must have $h^{-r/4}n^{-1/6} =o(1)$ which implies that $h^r$ has to converge to zero slower than $n^{-2/3}$ (i.e., $h^r>n^{-2/3}$). Note that under continuity and no smoothness assumptions, the rate for the bandwidth is   $h^r=n^{-1/3}$ which is much slower than $n^{-2/3}$. Thus, under continuity, for every level of smoothness the choice of optimal $h$ will satisfy the rate $h^r=n^{-2/3}$ (i.e., no condition is imposed).

\end{proof}

\begin{lemma} \label{lem:pnterm}
Let the function $f_n$ is the highly adaptive lasso estimate of $f_0$ where $f_0$ belongs to a cadlag class of functions with finite sectional variation norm. Then
\[
P_n  \frac{K_{h,\bm{v}_0}f_0}{P_0 K_{h,\bm{v}_0}} \left(f_0-f_n\right) = P_n  \frac{K_{h,\bm{v}_0}f_n}{P_n K_{h,\bm{v}_0}} \left(f_0-f_n\right)+o_p\left((nh^r)^{-1/2}\right).
\]
\end{lemma}
\begin{proof}
We have {\small
\begin{align*}
    P_n  \frac{K_{h,\bm{v}_0}f_0}{P_0 K_{h,\bm{v}_0}} \left(f_0-f_n\right) &= P_n  \frac{K_{h,\bm{v}_0}f_n}{P_n K_{h,\bm{v}_0}} \left(f_0-f_n\right) + P_n K_{h,\bm{v}_0} (f_0-f_n) \left( \frac{f_0}{P_0 K_{h,\bm{v}_0}}-\frac{f_n}{P_n K_{h,\bm{v}_0}}\right) \\
    &=P_n  \frac{K_{h,\bm{v}_0}f_n}{P_n K_{h,\bm{v}_0}} \left(f_0-f_n\right) + (P_n-P_0) K_{h,\bm{v}_0} (f_0-f_n) \left( \frac{f_0}{P_0 K_{h,\bm{v}_0}}-\frac{f_n}{P_n K_{h,\bm{v}_0}}\right)\\
    &\hspace{2in}+P_0 K_{h,\bm{v}_0} (f_0-f_n) \left( \frac{f_0}{P_0 K_{h,\bm{v}_0}}-\frac{f_n}{P_n K_{h,\bm{v}_0}}\right)\\
    &=P_n  \frac{K_{h,\bm{v}_0}f_n}{P_n K_{h,\bm{v}_0}} \left(f_0-f_n\right)+o_p\left((nh^r)^{-1/2}\right).
\end{align*}}
The last equality follows because $P_0 K_{h,\bm{v}_0} (f_0-f_n) \left( \frac{f_0}{P_0 K_{h,\bm{v}_0}}-\frac{f_n}{P_n K_{h,\bm{v}_0}}\right)=o_p\left((nh^r)^{-1/2}\right)$ by the highly adaptive lasso rate of convergence. 
\end{proof}

\begin{lemma} \label{lm:convrate}
Let $ P_0L_{G_0}(\Psi_n,\Psi_0) \equiv P_0L_{G_0}(\Psi_n) - P_0L_{G_0}(\Psi_0) $ where  $L_{G}$ is the weighted squared-loss function defined in (\ref{eq:lossf}) and $\mathcal{M}_\lambda=\{ m \in \mathcal{M}: \| m\|_{\nu} < \lambda\}$ be the set of cadlag functions with variation norm smaller than $\lambda$. Then, uniformly over $\lambda$, we have
\[
\left | P_0 \{L_{G_0}(\Psi_n,\Psi_0)-L_{G_n}(\Psi_n,\Psi_0)\}\right | \leq C \left\{ d_{01}(G_n,G_0)  \right\}^{1/2} \left\{ P_0L_{G_0}(\Psi_n,\Psi_0) \right\}^{1/2},
\]
where $C$ is a positive finite constant. 
\end{lemma}

\begin{lemma}\label{lem:wrate}
Suppose $\|\theta-\theta_0\|_2 = O_p(\delta_n)$ and $f(\theta) = \frac{h^rK_{h,\bm{v}_0} I( A= d^\theta)}{G_0(\theta)}Q_0(\theta)$. Under Assumption \ref{assump:margin}, $\|f(\theta)-f(\theta_0)\|_2 = O_p(h^{r/2}\delta_n)$.
\end{lemma}
\begin{proof}
We assume that $\bm{S}$ and $Q_0(\theta)$ are bounded uniformly by some constants $C_S$ and $C_Q$. Moreover, by strong positivity assumption $\min\{G_0(\theta,\bm{W}),1-G_0(\theta,\bm{W})\}>\gamma$ for all $\bm{W} \in \mathcal{W}$ and $\theta \in \Theta$.  By definition the norm  can be written  as
\begin{align*}
    \|f(\theta)-f(\theta_0)\|_2 &= \left\|h^rK_{h,\bm{v}_0} \left\{\frac{I( A= d^\theta)}{G_0(\theta)}Q_0(\theta)- \frac{I( A= d^{\theta_0})}{G_0(\theta_0)}Q_0(\theta_0) \right\} \right\|_2\\
    & \leq \|h^rK_{h,\bm{v}_0} \|_2 \left\|(\gamma)^{-1} \max\{|Q_0(\theta_0)|,|Q_0(\theta)|\} I(d^{\theta} \neq d^{\theta_0})\right\|_2\\
    & = \|h^rK_{h,\bm{v}_0} \|_2 \left\|(\gamma)^{-1} \max\{|Q_0(\theta_0)|,|Q_0(\theta)|\} I(d^{\theta} \neq d^{\theta_0}) \{ I(|\bm{S}^\top \theta_0|>t)+I(|\bm{S}^\top \theta_0|<t)\}\right\|_2\\
    & \leq \gamma^{-1} C_{Q} \|h^rK_{h,\bm{v}_0} \|_2 \left\{C t^{\kappa/2}+\left\|I(|\bm{S}^\top \theta_0-\bm{S}^\top \theta|>t) \right\|_2\right\}\\
    & \leq \gamma^{-1} C_{Q} \|h^rK_{h,\bm{v}_0} \|_2 \left\{ C t^{\kappa/2}+ \frac{\|\bm{S}^\top \theta_0-\bm{S}^\top \theta\|_2}{t}\right\}.
\end{align*}
The forth inequality follows from Assumption \ref{assump:margin},  the Cauchy-Schwarz and $|\bm{S}^\top \theta_0-\bm{S}^\top \theta|>|\bm{S}^\top \theta_0|$. The latter holds because because $I(d^\theta \neq d^{\theta_0}) =  1$  when (a) $\bm{S}^\top  {{\theta}}>0$ and $\bm{S}^\top \theta_0<0$; or (b) $\bm{S}^\top \theta<0$ and $\bm{S}^\top  \theta_0>0$. The former and latter imply that $\bm{S}^\top  {\theta}_0-\bm{S}^\top  \theta<\bm{S}^\top  \theta_0<0$ and $0<\bm{S}  \theta_0<\bm{S}^\top  \theta_0-\bm{S}^\top  \theta$, respectively, and thus, $|\bm{S}^\top \theta_0-\bm{S}^\top \theta|>|\bm{S}^\top \theta_0|$.
The fifth inequality holds by Markov inequalities. 
The upper bound is minimized when $t=C_t \|\bm{S}^\top \theta_0-\bm{S}^\top \theta\|_2^{\frac{1}{\kappa/2+1}}$ for a constant $C_t$ which depends on $\kappa$ and $C$. Hence, 
\begin{align} \label{lem:wbound}
    \|f(\theta)-f(\theta_0)\|_2 &\leq C^{\dagger} \|h^rK_{h,\bm{v}_0} \|_2 \left\{ C \|\bm{S}^\top \theta_0-\bm{S}^\top \theta\|_2^{\frac{\kappa}{\kappa+2}}+ \|\bm{S}^\top \theta_0-\bm{S}^\top \theta\|_2^{\frac{\kappa}{\kappa+2}}\right\}\nonumber\\
    &=O_p(h^{r/2}\delta^{\frac{\kappa}{\kappa+2}}). 
\end{align}
Hence where there is a margin around zero, that is, $pr(0<|\bm{S}^\top{{\theta}}_0|<l) 	=0$, the right hand side of (\ref{lem:wbound}) reduces to $O_p(h^{r/2}\delta)$. 
\end{proof}

\section{Proof of Theorem \ref{th:halrate}: Rate of convergence of HAL with unknown nuisance parameters}

Define $ P_0L_{G_0}(\Psi_n,\Psi_0) \equiv P_0L_{G_0}(\Psi_n) - P_0L_{G_0}(\Psi_0) $. Then using the definition of the loss-based dissimilarity, we have
\begin{align*}
0\leq d_0(\Psi_n,\Psi_0 )&=P_0L_{G_0}(\Psi_n,\Psi_0)\\
          &=(P_0-P_n) L_{G_0}(\Psi_n,\Psi_0) + P_n L_{G_0}(\Psi_n,\Psi_0) \\
          &=O_p(n^{-2/3} (\log n)^{4(p-1)/3})+P_n L_{G_n}(\Psi_n,\Psi_0)+ P_n\{L_{G_0}(\Psi_n,\Psi_0)-L_{G_n}(\Psi_n,\Psi_0)\} \\
          &\leq O_p(n^{-2/3} (\log n)^{4(p-1)/3})+(P_n-P_0) \{L_{G_0}(\Psi_n,\Psi_0)-L_{G_n}(\Psi_n,\Psi_0)\} \\
          &\hspace{3in} + P_0 \{L_{G_0}(\Psi_n,\Psi_0)-L_{G_n}(\Psi_n,\Psi_0)\} \\
          &=P_0 \{L_{G_0}(\Psi_n,\Psi_0)-L_{G_n}(\Psi_n,\Psi_0)\} +O_p(n^{-2/3} (\log n)^{4(p-1)/3})
\end{align*}
Moreover, using Lemma \ref{lm:convrate} we have
\begin{align*}
d_0(\Psi_n,\Psi_0 ) &\leq C \left\{ d_{01}(G_n,G_0)  \right\}^{1/2} \left\{ d_0(\Psi_n,\Psi_0 )  \right\}^{1/2} + O_p(n^{-2/3} (\log n)^{4(p-1)/3}) \\
              &= O_p(n^{-2/3} (\log n)^{4(p-1)/3}) +  O_p(n^{-1/3} (\log n)^{2(p-1)/3}) \left\{ d_0(\Psi_n,\Psi_0 )  \right\}^{1/2}
\end{align*}
where $C$ is a positive finite  constant. The last equality follows from the convergence rate of HAL minimum loss based estimator of $g_0$. Let $r_1(n)=O_p(n^{-2/3} (\log n)^{4(p-1)/3})$ and $r_2(n)=O_p(n^{-1/3} (\log n)^{2(p-1)/3})$,  the above equality is equivalent to
\[
 \left\{d_0(\Psi_n,\Psi_0 )\right\}^{1/2} \leq \frac{ \{4 r_1(n) + r_2^2(n) \}^{1/2} +r_2(n) }{2}
\]
which implies that $d_0(\Psi_n,\Psi_0 )=O_p(n^{-2/3} (\log n)^{4(p-1)/3})$.

\section{Proof of Theorem \ref{th:movh}: Asymptotic linearity of  the regimen-response curve estimator}

We represent the difference between $\Psi_{nh}(\bm{v}_0,\theta)$ and $\Psi_0(\bm{v}_0,\theta)$ as
\[
\Psi_{nh}(\bm{v}_0,\theta) -  \Psi_0(\bm{v}_0,\theta)= \{ \Psi_{0h}(\bm{v}_0,\theta) - \Psi_0(\bm{v}_0,\theta)\}+\{\Psi_{nh}(\bm{v}_0,\theta) -  \Psi_{0h}(\bm{v}_0,\theta)\}.
\]
Moreover, we can write
\begin{align} \label{eq:proofth1}
\Psi_{nh}(\bm{v}_0,\theta) -  \Psi_{0h}(\bm{v}_0,\theta) = &P_n  K_{h,\bm{v}_0}\Psi_n \left(\frac{1}{P_n K_{h,\bm{v}_0}}-\frac{1}{P_0 K_{h,\bm{v}_0}} \right) \\
&+ (P_n-P_0)\frac{K_{h,\bm{v}_0}}{P_0 K_{h,\bm{v}_0}}(\Psi_{n}-\Psi_{0}) + (P_n-P_0)\frac{K_{h,\bm{v}_0}}{P_0 K_{h,\bm{v}_0}}(\Psi_{0}) \nonumber\\
    &+P_0\frac{K_{h,\bm{v}_0}}{P_0 K_{h,\bm{v}_0}}(\Psi_{n}-\Psi_{0}). \nonumber
\end{align}

\subsection{Convergence  rate of $ \Psi_{0h}(v_0,\theta) -  \Psi_0(v_0,\theta)$}  
We consider a ($J-1$)-orthogonal kernel with bandwidth $h$ centered at $\bm{v}_0$. Then following Lemma 25.1 in \cite{van2018targeted}, as $h \rightarrow 0$,
\[
\Psi_{0h}(\bm{v}_0,\theta) - \Psi_0(\bm{v}_0,\theta) = h^{J} B_0(J,\bm{v}_0),
\]
where 
\[
B_0(J,\bm{v}_0) = \sum_{\{\eta \in \{0,\cdots,J\}^d: \sum_l \eta_l = J\}} \Psi_0(\bm{v}_0)^\eta \int_s k(s) \frac{\prod_l s_l ^{\eta_l}}{\prod_l \eta_l !} ds.
\]

\subsection{Asymptotic behaviour of the first term on the RHS of (\ref{eq:proofth1})}  
Because, $(P_n-P_0) K_{h,\bm{v}_0} = O_p(n^{-1/2} h^{-r/2})$, 
\begin{align*} 
P_n  K_{h,\bm{v}_0}  \Psi_n \left(\frac{1}{P_n K_{h,\bm{v}_0}}-\frac{1}{P_0 K_{h,\bm{v}_0}}\right) =&
(P_n-P_0)  K_{h,\bm{v}_0}  \Psi_n \left(\frac{1}{P_n K_{h,\bm{v}_0}}-\frac{1}{P_0 K_{h,\bm{v}_0}}\right) \\
&+P_0  K_{h,\bm{v}_0}  \Psi_n \left(\frac{1}{P_n K_{h,\bm{v}_0}}-\frac{1}{P_0 K_{h,\bm{v}_0}}\right)\\
=&P_0  K_{h,\bm{v}_0}  \Psi_n \left(\frac{-(P_n-P_0) K_{h,\bm{v}_0}}{P_n K_{h,\bm{v}_0}P_0 K_{h,\bm{v}_0}}\right)+o_p(n^{-1/2}h^{-r/2})\\
=&P_0  K_{h,\bm{v}_0}  \Psi_n \left(\frac{-(P_n-P_0) K_{h,\bm{v}_0}}{P_0^2 K_{h,\bm{v}_0}}\right)\\
&+P_0  K_{h,\bm{v}_0}  \Psi_n \left(\frac{-\{(P_n-P_0) K_{h,\bm{v}_0}\}^2}{P_n K_{h,\bm{v}_0}P_0 K_{h,\bm{v}_0}}\right)+o_p(n^{-1/2}h^{-r/2})\\
& =\frac{-(P_n-P_0)K_{h,\bm{v}_0}}{P_0K_{h,\bm{v}_0}} P_0  \frac{K_{h,\bm{v}_0}}{P_0K_{h,\bm{v}_0}}(\Psi_n-\Psi_{0})+\frac{-(P_n-P_0)K_{h,\bm{v}_0}}{P_0K_{h,\bm{v}_0}} \Psi_{0h}\\
&+o_p(n^{-1/2}h^{-r/2})\\
& =\frac{-(P_n-P_0)K_{h,\bm{v}_0}}{P_0K_{h,\bm{v}_0}} \Psi_{0h}+o_p(n^{-1/2}h^{-r/2})
\end{align*}
The second and third equalities follows because $\{(P_n-P_0) K_{h,\bm{v}_0}\}^2=O_p(n^{-1} h^{-r})$, and thus, those are of order $o_p(n^{-1/2}h^{-r/2})$. In the last equality we have also used the consistency of $\Psi_n$.


\subsection{Asymptotic negligibility of the second term on the RHS of (\ref{eq:proofth1})} \label{sec:th1p2}
We will show that the second order term $(P_n-P_0) K_{h,\bm{v}_0}  (\Psi_n-\Psi_0) = o_p\left((nh^r)^{-1/2}\right)$.
Let $\tilde K_{h,\bm{v}_0}(v) = h^r K_{h,\bm{v}_0}(v)$. Define $H_{n,h}(m) = K_{h,\bm{v}_0}  (m-\Psi_0)$ and $ h^{-r}\tilde H_{n,h}(m) =H_{n,h}(m)$. The function $\tilde H_n(m)$ belongs to a class of function $\mathcal{F}_n = \{\tilde H_{n,h}(m) : h\}$. We have
\begin{align*}
\| \tilde H_n(\Psi_n) \|_{P_0} &\leq \| \tilde K_{h,\bm{v}_0}(v)  \|_{P_0} \|  ( \Psi_n -\Psi_0) \|_{P_0} \\
                                                  &= O_p(h^{r/2})O_p(n^{-1/3}). 
\end{align*}
The rest follows from Lemma \ref{lem:hnneg}. Therefore, $(P_n-P_0)\frac{K_{h,\bm{v}_0}}{P_0 K_{h,\bm{v}_0}}(\Psi_{n}-\Psi_{0})=o_p\left((nh^r)^{-1/2}\right)$ when the bandwidth $h_n$ satisfies $h^r \rightarrow 0$ and $nh^{r3/2} \rightarrow \infty$.

\subsection{Asymptotic behaviour of the forth term on the RHS of (\ref{eq:proofth1})}  \label{sec:th1p3}
The forth term can be written as 
\begin{align*}
    P_0\frac{K_{h,\bm{v}_0}}{P_0 K_{h,\bm{v}_0}}(\Psi_{n}-\Psi_{0}) =&P_0\frac{K_{h,\bm{v}_0}\Delta^c I( A= d^\theta)}{G_0P_0 K_{h,\bm{v}_0}}(\Psi_{n}-\Psi_{0}) \\
    =&P_0\frac{K_{h,\bm{v}_0}\Delta^c I( A= d^\theta)}{G_0P_0 K_{h,\bm{v}_0}}\{\Psi_{n}-I(Y>t)\}\\
    =&(P_n-P_0)\frac{K_{h,\bm{v}_0}\Delta^c I( A= d^\theta)}{G_0P_0 K_{h,\bm{v}_0}}\{I(Y>t)-\Psi_{0}\}\\
    &-P_n \frac{K_{h,\bm{v}_0}\Delta^c I( A= d^\theta)}{G_0P_0 K_{h,\bm{v}_0}} \{I(Y\geq t)-\Psi_n\}+o_p\left((nh^r)^{-1/2}\right).
\end{align*}
The third equality follows from Section \ref{sec:th1p2}.

The last term on the RHS can be represented as
\begin{align*}
P_n \frac{K_{h,\bm{v}_0}\Delta^c I( A= d^\theta)}{G_0P_0 K_{h,\bm{v}_0}} \{I(Y\geq t)-\Psi_n\} &= P_n \frac{K_{h,\bm{v}_0}\Delta^c I( A= d^\theta)}{G_nP_0 K_{h,\bm{v}_0}} \{I(Y\geq t)-\Psi_n\}\\
      & +P_n \frac{K_{h,\bm{v}_0}\Delta^c I( A= d^\theta)}{P_0 K_{h,\bm{v}_0}} \{I(Y\geq t)-\Psi_n\}\left(\frac{G_n-G_0}{G_0G_n}\right)\\
      &=P_n \frac{K_{h,\bm{v}_0}\Delta^c I( A= d^\theta)}{G_nP_0 K_{h,\bm{v}_0}} \{I(Y\geq t)-\Psi_n\}+ \\
      &(P_n-P_0) \frac{K_{h,\bm{v}_0}\Delta^c I( A= d^\theta)}{P_0 K_{h,\bm{v}_0}} \{I(Y\geq t)-\Psi_n\}\left(\frac{G_n-G_0}{G_0G_n}\right)\\
      & +P_0 \frac{K_{h,\bm{v}_0}\Delta^c I( A= d^\theta)}{P_0 K_{h,\bm{v}_0}} \{I(Y\geq t)-\Psi_n\}\left(\frac{G_n-G_0}{G_0G_n}\right).
\end{align*}
Using Lemma \ref{lem:hnneg} and similar techniques used in Section \ref{sec:th1p2}, we have 
$(P_n-P_0) \frac{K_{h,\bm{v}_0}\Delta^c I( A= d^\theta)}{P_0 K_{h,\bm{v}_0}} \{I(Y\geq t)-\Psi_n\}\left(\frac{G_n-G_0}{G_0G_n}\right)=o_p\left((nh^r)^{-1/2}\right)$.

Let $Q_0 = E \{I(Y\geq t) | A= d^\theta,\bm{W}\}$. Then
\begin{align*}
    P_0 \frac{K_{h,\bm{v}_0}\Delta^c I( A= d^\theta)}{P_0 K_{h,\bm{v}_0}} \{I(Y\geq t)-\Psi_n\}\left(\frac{G_n-G_0}{G_0G_n}\right) =& P_0 \frac{K_{h,\bm{v}_0}\Delta^c I( A= d^\theta)}{P_0 K_{h,\bm{v}_0}} (Q_0-\Psi_n) \left(\frac{G_n-G_0}{G_0G_n}\right)\\
    =&-P_0 \frac{K_{h,\bm{v}_0}G_0}{P_0 K_{h,\bm{v}_0}} (Q_0-\Psi_n)\left(\frac{G_0-G_n }{G_0^2 }\right)\\
    &+P_0 \frac{K_{h,\bm{v}_0}G_0}{P_0 K_{h,\bm{v}_0}} (Q_0-\Psi_n)\left(\frac{(G_0-G_n)^2 }{G_nG_0^2 }\right).
\end{align*}
Using the rate of convergence of the highly adaptive lasso estimator $P_0 \frac{K_{h,\bm{v}_0}G_0}{P_0 K_{h,\bm{v}_0}} (Q_0-\Psi_n)\left(\frac{(G_0-G_n)^2 }{G_nG_0^2 }\right)=o_p\left((nh^r)^{-1/2}\right)$. Moreover,

\begin{align*}
P_0   \frac{K_{h,\bm{v}_0}(Q_0-\Psi_n)}{P_0 K_{h,\bm{v}_0}}\left(\frac{G_0-G_n }{G_0}\right)&= P_0 \frac{K_{h,\bm{v}_0}(Q_0-\Psi_n)}{P_0 K_{h,\bm{v}_0}}\left(\frac{G_{0}^cG_{0}^a-G_{n}^cG_{n}^a }{G_{0}^cG_{0}^a}\right)  \\
&=P_0 \frac{K_{h,\bm{v}_0}(Q_0-\Psi_n)}{P_0 K_{h,\bm{v}_0}} G_{0}^a\left(\frac{G_{0}^c-G_{n}^c }{G_{0}^cG_{0}^a}\right)+P_0 \frac{K_{h,\bm{v}_0}(Q_0-\Psi_n)}{P_0 K_{h,\bm{v}_0}} G_{n}^c\left(\frac{G_{0}^a-G_{n}^a}{G_{0}^cG_{0}^a}\right) \\
& = -(P_n-P_0) \frac{K_{h,\bm{v}_0}I( A= d^\theta)}{G_0^a P_0 K_{h,\bm{v}_0}} (Q_0-\Psi_0)\left(\frac{\Delta^c-G_{0}^c }{G_0^c}\right) \\
& \hspace{.2in}-(P_n-P_0) \frac{K_{h,\bm{v}_0}(Q_0-\Psi_0)}{P_0 K_{h,\bm{v}_0}} \left(\frac{I( A= d^\theta)-G_{0}^a}{G_0^a}\right)\\
&\hspace{.2in} + P_n \frac{K_{h,\bm{v}_0}I( A= d^\theta)}{G_0^a P_0  K_{h,\bm{v}_0}}(Q_0-\Psi_n)\left(\frac{\Delta^c-G_{n}^c }{G_0^c}\right) \\
&\hspace{.2in} + P_n  \frac{K_{h,\bm{v}_0 }(Q_0-\Psi_n)}{P_0 K_{h,\bm{v}_0}} G_{n}^c\left(\frac{I( A= d^\theta)-G_{n}^a}{G_{0}^cG_{0}^a}\right)+o_p\left((nh^r)^{-1/2}\right).
\end{align*}
The first equality follows because {\small
\begin{align*}
(P_n-P_0) \frac{K_{h,\bm{v}_0}I( A= d^\theta)}{G_0^a P_0 K_{h,\bm{v}_0}} (Q_0-\Psi_n) \left(\frac{\Delta^c-G_{n}^c }{G_0^c}\right) &= (P_n-P_0) \frac{K_{h,\bm{v}_0}I( A= d^\theta)}{G_0^a P_0 K_{h,\bm{v}_0}} (Q_0-\Psi_0) \left(\frac{\Delta^c-G_{0}^c }{G_0^c}\right)\\
&\hspace{2.5in}+o_p\left((nh^r)^{-1/2}\right), \\
(P_n-P_0) \frac{K_{h,\bm{v}_0}(Q_0-\Psi_n)}{P_0 K_{h,\bm{v}_0}} G_{n}^c\left(\frac{I( A= d^\theta)-G_{n}^a}{G_0^a G_0^c}\right)&=(P_n-P_0) \frac{K_{h,\bm{v}_0}(Q_0-\Psi_0)}{P_0 K_{h,\bm{v}_0}} \left(\frac{I( A= d^\theta)-G_{0}^a}{G_0^a}\right)\\
&\hspace{2.5in}+o_p\left((nh^r)^{-1/2}\right).
\end{align*}}
Moreover, by Lemma \ref{lem:pnterm}, we have {\small 
\begin{align*}
    P_n \frac{K_{h,\bm{v}_0}I( A= d^\theta)}{G_0^a P_0 K_{h,\bm{v}_0}} (Q_0-\Psi_n)\left(\frac{\Delta^c-G_{n}^c }{G_0^c}\right) &= P_n \frac{K_{h,\bm{v}_0}I( A= d^\theta)}{G_n^a P_0 K_{h,\bm{v}_0}} (Q_n-\Psi_n)\left(\frac{\Delta^c-G_{n}^c }{G_n^c}\right)+o_p\left((nh^r)^{-1/2}\right)\\
    P_n  \frac{K_{h,\bm{v}_0}(Q_0-\Psi_n)}{P_0 K_{h,\bm{v}_0}} G_{n}^c\left(\frac{I( A= d^\theta)-G_{n}^a}{G_{0}^cG_{0}^a}\right) &=P_n  \frac{K_{h,\bm{v}_0}(Q_n-\Psi_n)}{P_n K_{h,\bm{v}_0}} \left(\frac{I( A= d^\theta)-G_{n}^a}{G_{n}^a}\right)+o_p\left((nh^r)^{-1/2}\right)
\end{align*}}

Hence, if we undersmooth $G_{n}^a$ and $G_{n}^c$  such that $P_n \frac{K_{h,\bm{v}_0}I( A= d^\theta)}{G_n^a P_0 K_{h,\bm{v}_0}} (Q_n-\Psi_n)\left(\frac{\Delta^c-G_{n}^c }{G_n^c}\right)=o_p\left((nh^r)^{-1/2}\right)$ and $P_n  \frac{K_{h,\bm{v}_0}(Q_n-\Psi_n)}{P_n K_{h,\bm{v}_0}} \left(\frac{I( A= d^\theta)-G_{n}^a}{G_{n}^a}\right)=o_p\left((nh^r)^{-1/2}\right)$, then
\begin{align*}
P_0   \frac{K_{h,\bm{v}_0}(Q_0-\Psi_n)}{P_0 K_{h,\bm{v}_0}}\left(\frac{G_0-G_n }{G_0}\right)&=  -(P_n-P_0) \frac{K_{h,\bm{v}_0}I( A= d^\theta)}{G_0^a P_0 K_{h,\bm{v}_0}} (Q_0-\Psi_0)\left(\frac{\Delta^c-G_{0}^c }{G_0^c}\right) \\
& \hspace{.2in}-(P_n-P_0) \frac{K_{h,\bm{v}_0}(Q_0-\Psi_0)}{P_0 K_{h,\bm{v}_0}} \left(\frac{I( A= d^\theta)-G_{0}^a}{G_0^a}\right)+o_p\left((nh^r)^{-1/2}\right).
\end{align*}

\subsection{The influence function.} \label{sec:if}
Gathering all the terms, we have, 

\begin{align*}
\Psi_{nh}(\bm{v}_0,\theta)  -  \Psi_0(\bm{v}_0,\theta) =& (P_n-P_0) \frac{K_{h,\bm{v}_0}\Delta^c I( A= d^\theta)}{G_0P_0 K_{h,\bm{v}_0}}I(Y\geq t)  \\
&- (P_n-P_0) \frac{K_{h,\bm{v}_0}I( A= d^\theta)}{G_0^a P_0 K_{h,\bm{v}_0}} Q_0\left(\frac{\Delta^c-G_{0}^c }{G_0^c}\right) \\
& -(P_n-P_0) \frac{K_{h,\bm{v}_0}Q_0}{ P_0 K_{h,\bm{v}_0}} \left(\frac{I( A= d^\theta)-G_{0}^a}{G_0^a}\right) \\
&- (P_n-P_0)\frac{K_{h,\bm{v}_0}}{P_0K_{h,\bm{v}_0}}\Psi_{0h}(\bm{v}_0,\theta)+ h^{J} B_0(J)+o_p\left((nh^r)^{-1/2}\right).
\end{align*}
Thus, the asymptotic normality follows from the choice of bandwidth such that $(nh^r)^{1/2}h^{J} =O(1)$ which implies $h_n=n^{-1/(2J+r)}$. Thus, $(nh^r)^{1/2}\{\Psi_{nh}(\bm{v}_0,\theta)-  \Psi_0(\bm{v}_0,\theta)\}$ converges to a mean zero normal random variable when $\frac{h_n}{n^{-1/(2J+r)}} \rightarrow 0$ (i.e., $h_n$ converges to zero faster than $n^{-1/(2J+r)}$). 

\subsection{The optimal bandwidth rate.} \label{sec:opthrate} The rate constraints obtained  in Lemma \ref{lem:hnneg} (i.e., $h>n^{-\frac{2}{3r}}$) and Section \ref{sec:if} (i.e., $h<n^{-\frac{1}{r+2J}}$) imply that the optimal bandwidth rate is $h_n = n^{-\frac{1}{r+2J}}$. Note that for the constraint $n^{-\frac{2}{3r}}<h<n^{-\frac{1}{r+2J}}$ to hold we must have $J>r/4$ which is a very mild condition. For example, for $r<4$, we can still pick $J=1$.


\section{Proof of Theorem \ref{th:fixedh}: Asymptotic linearity of  the regimen-response curve estimator for a fixed bandwidth $h$}

\begin{align} \label{eq:proofth2}
\Psi_{nh}(\bm{v}_0,\theta)  -  \Psi_{0h}(\bm{v}_0,\theta) = &P_n  K_{h,\bm{v}_0}\Psi_n \left(\frac{1}{P_n K_{h,\bm{v}_0}}-\frac{1}{P_0 K_{h,\bm{v}_0}} \right) \\
&+ (P_n-P_0)\frac{K_{h,\bm{v}_0}}{P_0 K_{h,\bm{v}_0}}(\Psi_{n}-\Psi_{0}) + (P_n-P_0)\frac{K_{h,\bm{v}_0}}{P_0 K_{h,\bm{v}_0}}(\Psi_{0}) \nonumber\\
    &+P_0\frac{K_{h,\bm{v}_0}}{P_0 K_{h,\bm{v}_0}}(\Psi_{n}-\Psi_{0}). \nonumber
\end{align}

\subsection{Asymptotic behaviour of the first term on the RHS of (\ref{eq:proofth2})}  
Following the same steps used in the proof of Theorem \ref{th:movh} by consistency of $\Psi_n$ and because, for fixed $h$,  $(P_n-P_0)K_{h,\bm{v}_0}=O_p(n^{-1/2})$, we have 
\begin{align*}
P_n  K_{h,\bm{v}_0}\Psi_n \left(\frac{1}{P_n K_{h,\bm{v}_0}}-\frac{1}{P_0 K_{h,\bm{v}_0}} \right) =
 \frac{-(P_n-P_0)K_{h,\bm{v}_0}}{P_0K_{h,\bm{v}_0}} \Psi_{0h}(\bm{v}_0,\theta)+o_p(n^{-1/2}). 
\end{align*}
Under assumptions \ref{assump:cadlag} and \ref{assump:basic}, the equality above holds  uniformly on the space $\mathcal{V}$.

\subsection{Asymptotic negligibility of the second term on the RHS of (\ref{eq:proofth1})} \label{sec:th2p2}

We  show that $\sup_{v \in \mathcal{V}} \left|(P_n-P_0) \frac{K_{h,v}}{P_0 K_{h,v}}   \left( \Psi_n - \Psi_0 \right)\right| = o_p(n^{-1/2})$. Let $f_{v}(m) = m-\Psi_0$ and $\mathcal{F}_n = \{f_{v}(m): \sup_{v \in \mathcal{V}} \|f_{v}(m)\|<\delta_n, m \in \D[0,\tau], \|m\|^*_\nu<\infty, v\in\mathcal{V} \}$. Then we have
\begin{align*}
\sup_{v \in \mathcal{V}} \left|(P_n-P_0) \frac{K_{h,v}}{P_0 K_{h,v}}   \left( \Psi_n-\Psi_0\right)\right| \leq \sup_{\substack{v \in \mathcal{V}\\  f_{v}(m) \in \mathcal{F}_n}}  \left|(P_n-P_0) \frac{K_{h,v}}{P_0 K_{h,v}} f_{v}(m) \right| 
\end{align*}
By consistency of $\Psi_n$, we can consider $\delta_n \rightarrow 0$ which implies that the right hand side of the inequality above is of order $o_p(n^{-1/2})$.

\subsection{Asymptotic behaviour of the forth term on the RHS of (\ref{eq:proofth2})}  
Using the same techniques as in the proof of Theorem \ref{th:movh}, we can show 
\begin{align*}
  P_0\frac{K_{h,\bm{v}_0}}{P_0 K_{h,\bm{v}_0}}(\Psi_{n}-\Psi_{0}) = &(P_n-P_0)\frac{K_{h,\bm{v}_0}\Delta^c I( A= d^\theta)}{G_0P_0 K_{h,\bm{v}_0}}\{I(Y>t)-\Psi_{0}\}\\
    &-P_n \frac{K_{h,\bm{v}_0}\Delta^c I( A= d^\theta)}{G_0P_0 K_{h,\bm{v}_0}} \{I(Y\geq t)-\Psi_n\}+o_p(n^{-1/2}).
\end{align*}
This equality holds uniformly over $v \in \mathcal{V}$, because $\sup_{v \in \mathcal{V}} \left|(P_n-P_0) \frac{K_{h,v}}{P_0 K_{h,v}}   \left( \Psi_n - \Psi_0 \right)\right| = o_p(n^{-1/2})$ as shown in Section \ref{sec:th2p2}.
Moreover, following the result in Section \ref{sec:th1p3}, 
\begin{align*}
P_n \frac{K_{h,\bm{v}_0}\Delta^c I( A= d^\theta)}{G_0P_0 K_{h,\bm{v}_0}} \{I(Y\geq t)-\Psi_n\} =& P_n \frac{K_{h,\bm{v}_0}\Delta^c I( A= d^\theta)}{G_nP_0 K_{h,\bm{v}_0}} \{I(Y\geq t)-\Psi_n\}\\
& -(P_n-P_0) \frac{K_{h,\bm{v}_0}I( A= d^\theta)}{G_0^a P_0 K_{h,\bm{v}_0}} (Q_0-\Psi_0)\left(\frac{\Delta^c-G_{0}^c }{G_0^c}\right) \\
& -(P_n-P_0) \frac{K_{h,\bm{v}_0}(Q_0-\Psi_0)}{P_0 K_{h,\bm{v}_0}} \left(\frac{I( A= d^\theta)-G_{0}^a}{G_0^a}\right)\\
& + P_n \frac{K_{h,\bm{v}_0}I( A= d^\theta)}{G_0^a P_0  K_{h,\bm{v}_0}}(Q_0-\Psi_n)\left(\frac{\Delta^c-G_{n}^c }{G_0^c}\right) \\
& + P_n  \frac{K_{h,\bm{v}_0 }(Q_0-\Psi_n)}{P_0 K_{h,\bm{v}_0}} G_{n}^c\left(\frac{I( A= d^\theta)-G_{n}^a}{G_{0}^cG_{0}^a}\right)+o_p(n^{-1/2}).
\end{align*}
The equality above also holds uniformly over $v \in \mathcal{V}$ because
\begin{itemize}
    \item $\sup_{v \in \mathcal{V}} \left| (P_n-P_0) \frac{K_{h,\bm{v}_0}\Delta^c I( A= d^\theta)}{P_0 K_{h,\bm{v}_0}} \{I(Y\geq t)-\Psi_n\}\left(\frac{G_n-G_0}{G_0G_n}\right) \right|=o_p(n^{-1/2})$,
    \item  $\sup_{v \in \mathcal{V}} \left|  P_0 \frac{K_{h,\bm{v}_0}G_0}{P_0 K_{h,\bm{v}_0}} (Q_0-\Psi_n)\left(\frac{(G_0-G_n)^2 }{G_nG_0^2 }\right)\right|=o_p(n^{-1/2})$,
    \item  $\sup_{v \in \mathcal{V}} \left| P_0 K_{h,\bm{v}_0}\{G_n-G_0\} \left( \frac{Q_0-\Psi_n}{G_0P_0 K_{h,\bm{v}_0}}-\frac{Q_n-\Psi_n}{G_nP_n K_{h,\bm{v}_0}}\right)\right|=o_p(n^{-1/2})$.
\end{itemize}

\subsection{Convergence to a Gaussian process}

If we undersmooth $G_{n}^a$, $G_{n}^c$ and $\Psi_n$  such that $\sup_{v \in \mathcal{V}} \left|P_n  \frac{K_{h,\bm{v}_0 }(Q_0-\Psi_n)}{P_0 K_{h,\bm{v}_0}} G_{n}^c\left(\frac{I( A= d^\theta)-G_{n}^a}{G_{0}^cG_{0}^a}\right)\right|=o_p(n^{-1/2})$, $\sup_{v \in \mathcal{V}} \left| P_n \frac{K_{h,\bm{v}_0}I( A= d^\theta)}{G_0^a P_0  K_{h,\bm{v}_0}}(Q_0-\Psi_n)\left(\frac{\Delta^c-G_{n}^c }{G_0^c}\right) \right|= o_p(n^{-1/2})$, and $\sup_{v \in \mathcal{V}} \left| P_n \frac{K_{h,v}\Delta^c I( A= d^\theta) }{G_n P_n K_{h,v}} \{ \Psi_n -I(Y >t) \}  \right|= o_p(n^{-1/2})$, then
\begin{align*}
\sqrt n \{\Psi_{nh}(\bm{v}_0,\theta)  -  \Psi_{0h}(\bm{v}_0,\theta)\} = \sqrt n (P_n-P_0) D_v^*(P_0) + o_p(1) \Rightarrow_d GP.
\end{align*}
where

\begin{align*}
D_{\bm{v}}^*(P_0) = \frac{K_{h,\bm{v}_0}\Delta^c I( A= d^\theta)}{G_0P_0 K_{h,\bm{v}_0}}  &I(Y\geq t) -\frac{K_{h,\bm{v}_0}I( A= d^\theta)}{G_0^a P_0 K_{h,\bm{v}_0}} Q_0\left(\frac{\Delta^c-G_{0}^c }{G_0^c}\right)\\
&-\frac{K_{h,\bm{v}_0}Q_0}{P_0 K_{h,\bm{v}_0}} \left(\frac{I( A= d^\theta)-G_{0}^a}{G_0^a}\right)+\frac{K_{h,\bm{v}_0}}{P_0K_{h,\bm{v}_0}}\Psi_{0h}(\bm{v}_0,\theta).
\end{align*}
Hence, for all $\bm{v} \in \mathcal{V}$, $\sqrt n \{\Psi_{nh}(\bm{v},\theta)  -  \Psi_{0h}(\bm{v},\theta)\}$ converges weakly as a random element of the cadlag function
space endowed with the supremum norm to a Gaussian process $GP$ with covariance
structure implied by the covariance function $\rho (\bm{v},\bm{v}') = P_0 D^*_{\bm{v}}(P_0)      
D^*_{\bm{v}'}(P_0)$.

\section{Uniform consistency of the regimen-response curve estimator}\label{sec:uniform}

In Theorem \ref{th:movh} we showed that our estimator is pointwise consistent. In this section we show that our estimator is uniformly consistent as well in the sense that $\sup_{v \in \mathcal{V}} |\Psi_{nh}(\bm{v}_0,\theta) -  \Psi_0(\bm{v}_0,\theta)| = o_p(1)$ and derive a rate of convergence. Recall that
\[
\Psi_{nh}(\bm{v}_0,\theta) -  \Psi_0(\bm{v}_0,\theta)= \{ \Psi_{0h}(\bm{v}_0,\theta) - \Psi_0(\bm{v}_0,\theta)\}+\{\Psi_{nh}(\bm{v}_0,\theta) -  \Psi_{0h}(\bm{v}_0,\theta)\},
\]
where
\begin{align} 
\Psi_{nh}(\bm{v}_0,\theta) -  \Psi_{0h}(\bm{v}_0,\theta) = &P_n  K_{h,\bm{v}_0}\Psi_n \left(\frac{1}{P_n K_{h,\bm{v}_0}}-\frac{1}{P_0 K_{h,\bm{v}_0}} \right) \\
&+ (P_n-P_0)\frac{K_{h,\bm{v}_0}}{P_0 K_{h,\bm{v}_0}}(\Psi_{n}-\Psi_{0}) + (P_n-P_0)\frac{K_{h,\bm{v}_0}}{P_0 K_{h,\bm{v}_0}}\Psi_{0} \nonumber\\
    &+P_0\frac{K_{h,\bm{v}_0}}{P_0 K_{h,\bm{v}_0}}(\Psi_{n}-\Psi_{0}). \nonumber
\end{align}
We know that $\sup_{v \in \mathcal{V}} |\Psi_{0h}(\bm{v}_0,\theta) - \Psi_0(\bm{v}_0,\theta)| = O_p(h^r)$, $\sup_{v \in \mathcal{V}} (P_n-P_0) K_{h,\bm{v}} = O_p( n^{-1/2}h^{-r/2})$, and 
 $\sup_{v \in \mathcal{V}} \left|(P_n-P_0) \frac{K_{h,v}}{P_0 K_{h,v}}   \left( \Psi_n-\Psi_0\right)\right| = o_p(n^{-1/2}h^{-r/2})$. Moreover, $\sup_{v \in \mathcal{V}} \left| P_0 \frac{K_{h,\bm{v}_0}}{P_0 K_{h,\bm{v}_0}} \left(\Psi_n-\Psi_0\right) \right| = O_p(\sup_{v \in \mathcal{V}} \|\Psi_n-\Psi_0\|)$. Thus, assuming that $\sup_{v \in \mathcal{V}} \|\Psi_n-\Psi_0\|=O_p(b_n)$, we have 
 \[
 \sup_{v \in \mathcal{V}} |\Psi_{nh}(\bm{v}_0,\theta) -  \Psi_0(\bm{v}_0,\theta)| = O_p\left( n^{-1/2}h^{-r/2}+h^r+b_n \right).
 \]
 
 
 \section{Proof of Theorem \ref{th:thetrate}: Consistency of $\theta_{nh}$}
 \label{sec:ratethetaconvh}

The upper bound of $E_P \sup_{f \in \mathcal{F}, \|f\|_P \leq \delta} (n)^{1/2}(P_n-P_0)(f)$ given in Lemma \ref{lem:vw} is either dominated by $\mathcal{E}(\delta,\mathcal{F},L^2(P))$ or $\mathcal{E}^2(\delta,\mathcal{F},L^2(P))/(\delta^2 n^{1/2})$. The switch occurs when $\mathcal{E}(\delta,\mathcal{F},L^2(P)) = \delta^2 n^{1/2}$. By Lemma \ref{lem:laan}, for the cadlag class of functions $\mathcal{F}=\mathcal{F}_{d,M}$, we have $\mathcal{E}(\delta,\mathcal{F},L^2(P)) = \delta^{1/2} |\log \delta|^{p-1}$. Therefore, the switch occurs when $\delta^2 n^{1/2}=\delta^{1/2} |\log \delta|^{p-1}$ which implies $\delta = n^{-1/3}|\log n|^{2(p-1)/3}$  (This is because $\log \delta = -1/3\log n+\log |\log \delta|^{p-1}$ where $\log n$ is the leading term). So the result of Lemma \ref{lem:vw} can be written as
\begin{align*}
    E_P \sup_{f \in \mathcal{F}, \|f\|_P \leq \delta} n^{1/2}(P_n-P_0)(f) \leq &I(\delta \geq n^{-1/3}| \log n|^{2(p-1)/3} )\mathcal{E}(\delta,\mathcal{F},L^2(P)) \\
&+I(\delta \leq n^{-1/3}|\log n|^{2(p-1)/3} ) \left\{\frac{\mathcal{E}(\delta,\mathcal{F},L^2(P))}{\delta^2 n^{1/2}}M\right\}.
\end{align*}

In the following, for the ease of notation, we denote $\theta_{nh}(v_0)\equiv \theta_{nh}$ and $\theta_{0}(v_0)\equiv \theta_{0}$. Define a loss-based dissimilarity $d(\theta_n,\theta_0) = \Psi_0(\theta_{nh}) - \Psi_0(\theta_0)$. Then by definition
\begin{align*}
0 \leq d(\theta_{nh},\theta_0) &= (\Psi_0-\Psi_{nh})(\theta_{nh}) - (\Psi_0-\Psi_{nh})(\theta_0) + \Psi_{nh}(\theta_{nh})-\Psi_{nh}(\theta_0) \\
                            & \leq -\{(\Psi_{nh}-\Psi_0)(\theta_{nh}) - (\Psi_{nh}-\Psi_0)(\theta_0) \}\\
                            &=(P_n-P_0) \{D^*_0(\theta_{nh})-D^*_0(\theta_0)\} + h^{J} B_0(\theta_{nh}-\theta_0) + R(\theta_{nh})-R(\theta_0).
\end{align*}
The last equality follows from Theorem \ref{th:movh}. Assuming  $\frac{h_n}{n^{-1/(2J+r)}} \rightarrow 0$ (i.e., $h_n$ converges to zero faster than $n^{-1/(2J+r)}$), we have 
\begin{align*}
0 \leq d(\theta_{nh},\theta_0)  \leq (P_n-P_0) \{D^*_0(\theta_{nh})-D^*_0(\theta_0)\}  + \{R(\theta_{nh})-R(\theta_0)\}+o(n^{-J/(2J+r)}).
\end{align*}
Let $\tilde d(\theta_{nh},\theta_{0h})=(P_n-P_0) \{D^*_0(\theta_{nh})-D^*_0(\theta_{0h})\} $. Using the rate obtained for the remainder terms in Theorem \ref{th:movh}, we have $d(\theta_{nh},\theta_0) = \tilde d(\theta_{nh},\theta_{0h}) + \{R(\theta_{nh})-R(\theta_0)\} = \tilde d(\theta_{nh},\theta_{0h})+ o_p(n^{-1/2}h^{-r/2})$. We proceed with assuming that the convergence rate of $\tilde d(\theta_{nh},\theta_{0h})$ is the dominating rate. We will show that the assumption holds latter in the proof.

 Recall,
\begin{align*}
 D^*_0(\theta)=  \frac{K_{h,\bm{v}_0}\Delta^c I( A= d^\theta)}{G_0P_0 K_{h,\bm{v}_0}}I(Y\geq t) -   \frac{K_{h,\bm{v}_0}I( A= d^\theta)}{G_0^a P_0 K_{h,\bm{v}_0}} Q_0\left(\frac{\Delta^c-G_{0}^c }{G_0^c}\right) - &\frac{K_{h,\bm{v}_0}Q_0}{P_0 K_{h,\bm{v}_0}} \left(\frac{I( A= d^\theta)-G_{0}^a}{G_0^a}\right) \\
&- \frac{K_{h,\bm{v}_0}}{P_0K_{h,\bm{v}_0}}\Psi_{0h}(\bm{v}_0,\theta). 
\end{align*}
Let $f=D^*_0(\theta)-D^*_0(\theta_0)$ and $h^{-r}\tilde{f} = f$. Then, 
\begin{align}
(P_n-P_0) \tilde f &\leq   \sup_{ \|\tilde f\|_P \leq h^{r/2}} | (P_n-P_0)\tilde f | \nonumber \\
                                 & =O_p\{  n^{-1/2} \mathcal{E}(h^{r/2},L^2(P))\} \label{pth:bound1}
\end{align}
The equality (\ref{pth:bound1}), follows from Markov inequality and  the result of Lemma \ref{lem:wrate}. Hence, $(P_n-P_0)  f = O_p(n^{-1/2} h^{r/4}h^{-r}) = O_p( n^{-1/2} h^{-3r/4})$. 
Using Taylor expansion, we have $d(\theta_{nh},\theta_0) =\Psi_0(\theta_{nh}) - \Psi_0(\theta_0)= (\theta_{nh}-\theta_0)^\top\frac{\partial^2 \Psi_0}{\partial \theta^2} (\theta_{nh}-\theta_0)+o(\|\theta_{nh}-\theta_0\|^2)$ which implies $\|\theta_{nh}-\theta_0\|_{2,\mu}=O_p(n^{-1/4}h^{-3r/8})$. Let $\delta = n^{-1/4}h^{-3r/8}$. Using Lemma \ref{lem:wrate},
\begin{align}
(P_n-P_0) \tilde f &\leq   \sup_{ \|\tilde f\|_P \leq h^{r/2}\delta} |(P_n-P_0)\tilde f| \nonumber \\
                                 & =O_p\{  n^{-1/2} \mathcal{E}(h^{r/2}\delta^{\frac{\kappa}{\kappa+2}},L^2(P))\}, \label{pth:bound2}
\end{align}
Then $d(\theta_{nh},\theta_0) = O_p(n^{-1/2}h^{-3r/4} \delta^{\frac{\kappa}{2\kappa+4}})$. The iteration continues until there is no improvement in the rate. That is $\delta^2 = n^{-1/2}h^{-3r/4} \delta^{\frac{\kappa}{2\kappa+4}}$ which implies 
\[
\delta = \left( n^{-1/2} h^{-3r/4}\right)^{\frac{2\kappa+4}{3\kappa+8}}.
\]
Hence, $\|\theta_{nh}(v_0)-\theta_0(v_0)\|_{2}=O_p\left(\left( n^{-1/2} h^{-3r/4}\right)^{\frac{2\kappa+4}{3\kappa+8}}\right)$. When there is a margin around zero, that is, $pr(0<|\bm{S}^\top{{\theta}}_0|<l) 	=0$, we will have $\|\theta_{nh}(v_0)-\theta_0(v_0)\|_{2,\mu}=O_p\left( n^{-1/3} h^{-r/2}\right)$. 

At each iteration the convergence rate of $d(\theta_{nh},\theta_0)$ and at the fix point it achieves the rate of $d(\theta_{nh},\theta_0)=O_p\left(\left( n^{-1/2} h^{-3r/4}\right)^{\frac{4\kappa+8}{3\kappa+8}}\right)$. Hence, to show that the latter rate dominates the remainder term rate, it is sufficient to show that 
\[
\frac{(nh^r)^{1/2}}{\left( n^{1/2} h^{3r/4}\right)^{\frac{4\kappa+8}{3\kappa+8}}} \rightarrow \infty. 
\]
The condition is satisfied when $h < n^{-\frac{\kappa}{r(3k+4)}}$. The right hand side is a decreasing function of $\kappa$ and as $\kappa \rightarrow \infty$, $n^{-\frac{\kappa}{r(3\kappa+4)}} \rightarrow n^{-\frac{1}{3r}}$. This rate condition is satisfied when $n^{-2/(3r)}<h<n^{-1/(2J+r)}$ where $J>r/4$ and for the optimal rate of $h_n = n^{-1/(2J+r)}$. So, no further condition is imposed to the rate of bandwidth and the optimal rate remains as $h_n = n^{-1/(2J+r)}$. Consequently, 
\[
\|\theta_{nh}(v_0)-\theta_0(v_0)\|_{2}=O_p\left( n^{\frac{(r-4J)(2\kappa+4)}{(8J+4r)(3\kappa+8)}} \right).
\]

\section{Additional figures}\label{sec:addfig}

\begin{figure}[ht]
    \centering
    \includegraphics[width = 1\textwidth]{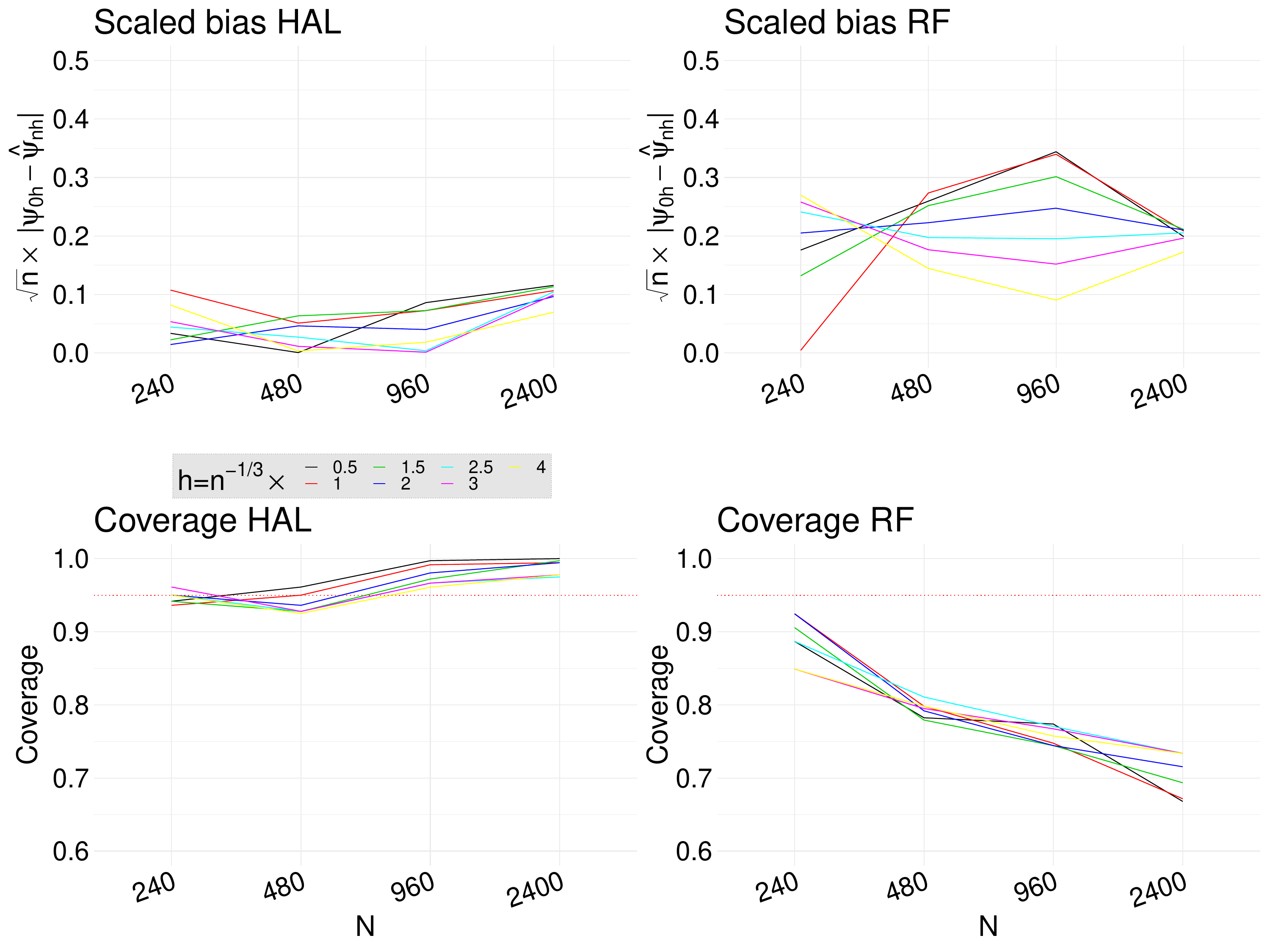}
    \caption{Simulation studies Scenario 2: The target parameter of interest is $\Psi_{0h}$. The plots show the scaled bias and coverage rate when the weight functions are estimated using an undersmoothed highly adaptive lasso (HAL) and a random forest (RF). } 
    \label{fig:supfixedh}
\end{figure}

\begin{figure}[ht]
    \centering
    \includegraphics[width = 1\textwidth]{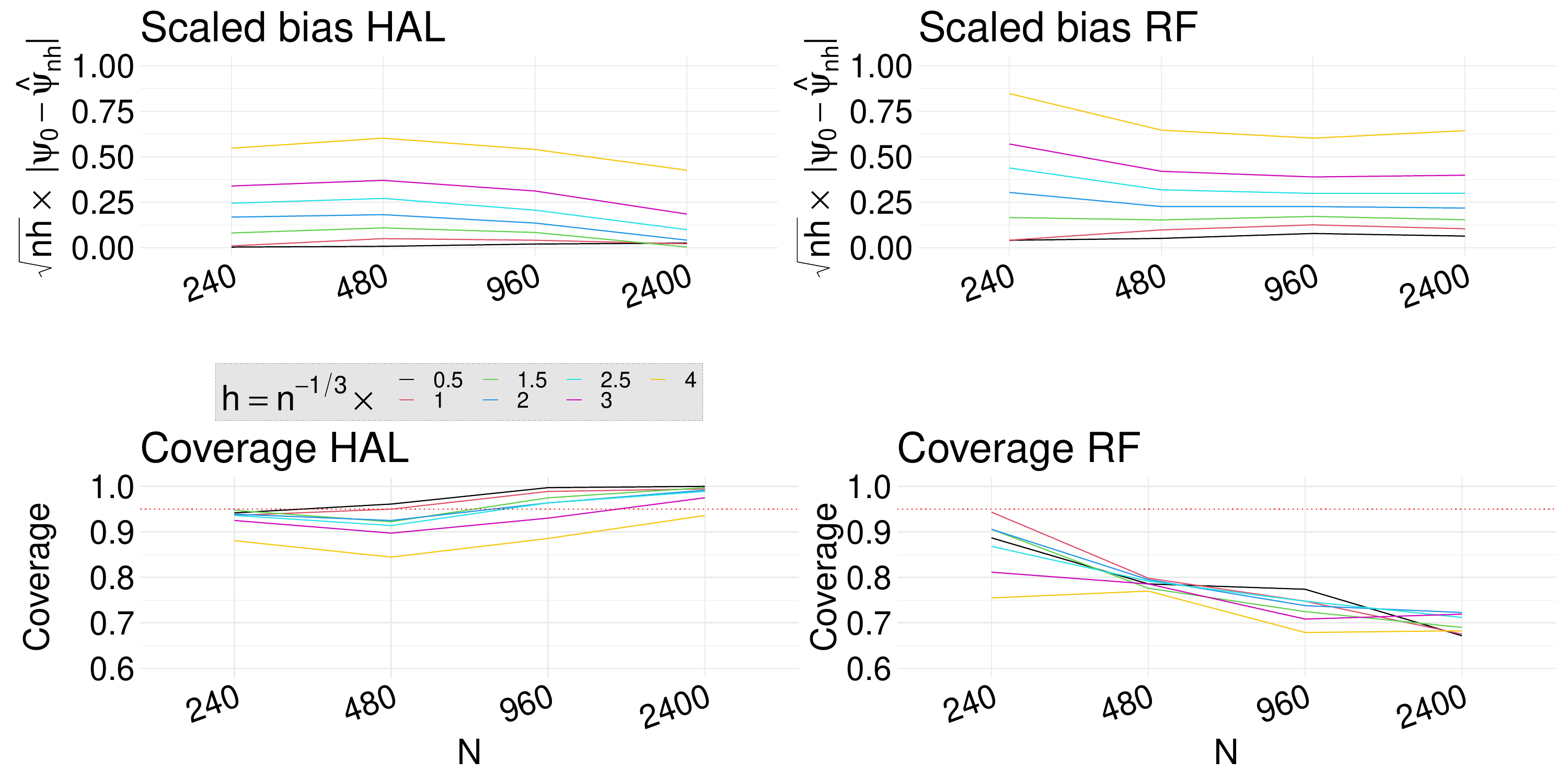}
    \caption{Simulation studies Scenario 2: The target parameter of interest is $\Psi_{0}$. The plots show the scaled bias and coverage rate when the weight functions are estimated using an undersmoothed highly adaptive lasso (HAL) and a random forest (RF).} 
    \label{fig:supconvergh}
\end{figure}

\begin{figure}[ht]
    \centering
    \includegraphics[width = 1\textwidth]{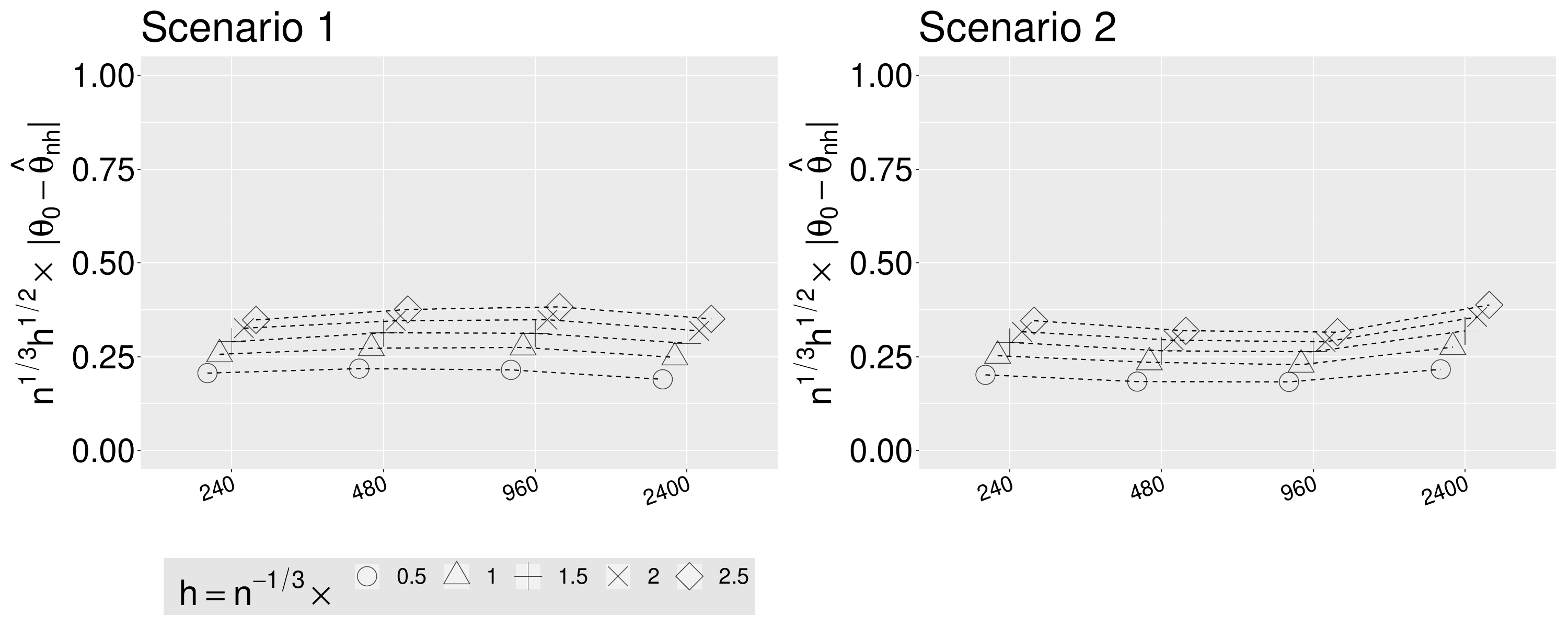}
    \caption{Simulation studies Scenarios 1 and 2: The target parameter of interest is the true minimizer of $\Psi_0$ (i.e., $\theta_{0}$). The plots show the scaled bias of $\theta_{nh}$ when the weight functions are estimated using an undersmoothed highly adaptive lasso.} 
    \label{fig:thetrate}
\end{figure}

  \bibliographystyleapp{biometrika}
  \bibliographyapp{chalipwbib}


\end{document}